\makeatletter

\documentclass[
    11pt,
    a4paper,
    oneside,
    openright,
    center,
    chapterbib,
    crosshair,
    reqno,
    headcount,
    headline,
    indent,
    indentfirst=false,
    portrait,
    phonetic,
    oldernstyle,
    onecolumn,
    sfbold,
    upper,
]{amsart}

\let\th@plain\relax

\PassOptionsToPackage{T2A}{fontenc} 
\PassOptionsToPackage{utf8}{inputenc} 
\PassOptionsToPackage{cyrpart}{cyrillic}
\PassOptionsToPackage{british,ngerman,russian}{babel}
\PassOptionsToPackage{
    foot,
}{amsaddr}
\PassOptionsToPackage{
    style=numeric-comp,
    citestyle=authoryear,
}{biblatex}
\PassOptionsToPackage{
    english,
    ngerman,
    russian,
    capitalise,
}{cleveref}
\PassOptionsToPackage{
    margin=10pt,
    position=bottom,
    justification=centering,
    font=small,
    labelformat=simple,
    labelfont={sc},
    labelsep={period},
    textfont={it},
}{caption}
\PassOptionsToPackage{
    margin=10pt,
    position=bottom,
    justification=centering,
    font=small,
    labelformat=parens,
    labelfont={bf},
    labelsep={space},
    textfont={it},
}{subcaption}
\PassOptionsToPackage{framemethod=TikZ}{mdframed}
\PassOptionsToPackage{
    bookmarks=true,
    bookmarksopen=false,
    bookmarksopenlevel=0,
    bookmarkstype=toc,
    raiselinks=true,
    colorlinks=true,
    hyperfigures=true,
    anchorcolor=red,
    citecolor=green,
    linkcolor=blue,
    urlcolor=blue,
    hypertexnames=false, 
}{hyperref} 
\PassOptionsToPackage{normalem}{ulem}
\PassOptionsToPackage{
    amsmath,
    thmmarks,
}{ntheorem}
\PassOptionsToPackage{
}{thmtools}
\PassOptionsToPackage{
    table,
}{xcolor}
\PassOptionsToPackage{
    all,
    color,
    curve,
    frame,
    import,
    knot,
    line,
    movie,
    rotate,
    textures,
    tile,
    tips,
    web,
    xdvi,
}{xy}
\PassOptionsToPackage{
    reset,
    left=3cm,
    right=3cm,
    top=20mm,
    bottom=20mm,
    heightrounded,
}{geometry}
\PassOptionsToPackage{
    symbol*, 
    multiple, 
}{footmisc}
\PassOptionsToPackage{overload}{textcase}
\PassOptionsToPackage{indentafter}{titlesec}
\PassOptionsToPackage{
    shortcuts, 
}{extdash}
\PassOptionsToPackage{
    algo2e,
    english,
    onelanguage,
    algosection,
    procnumbered,
    ruled,
    vlined,
    linesnumbered,
}{algorithm2e}

\usepackage{amsaddr}
\usepackage{amsfonts}
\usepackage{amsmath}
\usepackage{amssymb}
\usepackage{ntheorem} 
\usepackage{thmtools} 
\usepackage{array}
\usepackage{babel}
\usepackage{braket}
\usepackage{algorithm2e}
\usepackage{bbding}
\usepackage{bm}
\usepackage{bbm}
\usepackage{bibentry}
\usepackage{booktabs}
\usepackage{bold-extra} 
\usepackage{calc}
\usepackage{cancel}
\usepackage{caption} 
\usepackage{changepage}
\usepackage{cjhebrew}
\usepackage{cmlgc}
\usepackage{colonequals}
\usepackage{color}
\usepackage{comment}
\usepackage{datetime}
\usepackage{dsfont}
\usepackage{etex}
\usepackage{etoolbox}
\usepackage{eurosym}
\usepackage{extdash}
\usepackage{fancybox}
\usepackage{fancyhdr}
\usepackage{float}
\usepackage{fontenc}
\usepackage{footmisc}
\usepackage{fp}
\usepackage{geometry}
\usepackage{graphicx}
\usepackage{ifpdf}
\usepackage{ifthen}
\usepackage{ifoddpage}
\usepackage{ifnextok}  
\usepackage{inputenc}
\usepackage{latexsym}
\usepackage{lineno}
\usepackage{listings}
\usepackage{longtable}
\usepackage{lscape}
\usepackage{mathtools} 
\usepackage{mathrsfs}
\usepackage{multicol}
\usepackage{multirow}
\usepackage{nameref}
\usepackage{nowtoaux}
\usepackage{paralist}
\usepackage{enumerate} 
\usepackage{enumitem} 
\usepackage{pgf}
\usepackage{pgfplots}
\usepackage{phonetic}
\usepackage{proof}
\usepackage{qtree}
\usepackage{refcount}
\usepackage{savesym}
\usepackage{stmaryrd}
\usepackage{synttree}
\usepackage{subcaption}
\usepackage{suffix}
\usepackage{yfonts} 
\usepackage{textcase} 
\usepackage{tikz}
\usepackage{xy}
\usepackage{ulem} 
\usepackage{wrapfig}
\usepackage{xcolor}
\usepackage{xspace}
\usepackage{xstring}
\usepackage{arydshln}
\usepackage{hyperref}
\usepackage{cleveref} 

\pgfplotsset{compat=newest}
\usetikzlibrary{math}

\usetikzlibrary{
    angles,
    arrows,
    automata,
    calc,
    decorations,
    decorations.pathmorphing,
    decorations.pathreplacing,
    positioning,
    patterns,
    patterns.meta,
    quotes,
}

\savesymbol{corresponds}
\savesymbol{Diamond}
\savesymbol{emptyset}
\savesymbol{ggg}
\savesymbol{int}
\savesymbol{lll}
\savesymbol{RectangleBold}
\savesymbol{langle}
\savesymbol{rangle}
\savesymbol{hookrightarrow}
\savesymbol{hookleftarrow}
\savesymbol{Asterisk}
\usepackage{mathabx}
\usepackage{wasysym}

\restoresymbol{x}{corresponds}
\restoresymbol{x}{Diamond}
\restoresymbol{x}{emptyset}
\restoresymbol{x}{ggg}
\restoresymbol{x}{int}
\restoresymbol{x}{lll}
\restoresymbol{x}{RectangleBold}
\restoresymbol{x}{langle}
\restoresymbol{x}{rangle}
\restoresymbol{x}{hookrightarrow}
\restoresymbol{x}{hookleftarrow}
\restoresymbol{x}{Asterisk}

\ifpdf
    \usepackage{pdfcolmk}
\fi

\usepackage{mdframed}

\DeclareFontFamily{U}{MnSymbolA}{}
\DeclareFontShape{U}{MnSymbolA}{m}{n}{
    <-6> MnSymbolA5
    <6-7> MnSymbolA6
    <7-8> MnSymbolA7
    <8-9> MnSymbolA8
    <9-10> MnSymbolA9
    <10-12> MnSymbolA10
    <12-> MnSymbolA12
}{}
\DeclareFontShape{U}{MnSymbolA}{b}{n}{
    <-6> MnSymbolA-Bold5
    <6-7> MnSymbolA-Bold6
    <7-8> MnSymbolA-Bold7
    <8-9> MnSymbolA-Bold8
    <9-10> MnSymbolA-Bold9
    <10-12> MnSymbolA-Bold10
    <12-> MnSymbolA-Bold12
}{}
\DeclareSymbolFont{MnSymA}{U}{MnSymbolA}{m}{n}
\DeclareMathSymbol{\lcirclearrowright}{\mathrel}{MnSymA}{252}
\DeclareMathSymbol{\lcirclearrowdown}{\mathrel}{MnSymA}{255}
\DeclareMathSymbol{\rcirclearrowleft}{\mathrel}{MnSymA}{250}
\DeclareMathSymbol{\rcirclearrowdown}{\mathrel}{MnSymA}{251}
\DeclareFontFamily{U}{MnSymbolC}{}
\DeclareSymbolFont{MnSyC}{U}{MnSymbolC}{m}{n}
\DeclareFontShape{U}{MnSymbolC}{m}{n}{
    <-6>  MnSymbolC5
    <6-7>  MnSymbolC6
    <7-8>  MnSymbolC7
    <8-9>  MnSymbolC8
    <9-10> MnSymbolC9
    <10-12> MnSymbolC10
    <12->   MnSymbolC12%
}{}
\DeclareMathSymbol{\powerset}{\mathord}{MnSyC}{180}
\DeclareMathSymbol{\righthalfcap}{\mathbin}{MnSyC}{186}

\DeclareMathAlphabet{\mathpzc}{OT1}{pzc}{m}{it}
\DeclareMathAlphabet{\blackboardfont}{U}{BOONDOX-ds}{m}{n}

\let\braket\relax
\let\ketbra\relax
\DeclarePairedDelimiterX{\braket}[2]{\langle}{\rangle}{#1\,\delimsize\vert\,\mathopen{}#2}
\DeclarePairedDelimiterX{\ketbra}[2]{\lvert}{\rvert}{#1\delimsize\rangle\!\delimsize\langle#2}

\def\boolwahr{true}
\def\boolfalsch{false}
\def\boolleer{}

\let\boolinappendix\boolfalsch
\let\boolinmdframed\boolfalsch

\newcount\bufferctr
\newcount\bufferreplace

\newlength\rtab
\newlength\gesamtlinkerRand
\newlength\gesamtrechterRand
\newlength\ownspaceabovethm
\newlength\ownspacebelowthm
\newlength\aboveequation
\newlength\belowequation
\def\secnumberingpt{.}
\def\secnumberingseppt{.}
\def\subsecnumberingseppt{}
\def\thmnumberingpt{.}
\def\thmnumberingseppt{}
\def\thmForceSepPt{.}

\definecolor{leer}{gray}{1}
\definecolor{boxgrau}{gray}{0.85}
\definecolor{dunkelgrau}{gray}{0.5}
\definecolor{maroon}{rgb}{0.6901961,0.1882353,0.3764706}
\definecolor{dunkelgruen}{rgb}{0.015625,0.363281,0.109375}
\definecolor{dunkelrot}{rgb}{0.5450980392,0,0}
\definecolor{dunkelblau}{rgb}{0,0,0.5450980392}
\definecolor{blau}{rgb}{0,0,1}
\definecolor{newresult}{rgb}{0.6,0.6,0.6}
\definecolor{improvedresult}{rgb}{0.9,0.9,0.9}
\definecolor{hervorheben}{rgb}{0,0.9,0.7}
\definecolor{starkesblau}{rgb}{0.1019607843,0.3176470588,0.8156862745}
\definecolor{achtung}{rgb}{1,0.5,0.5}
\definecolor{frage}{rgb}{0.5,1,0.5}
\definecolor{schreibweise}{rgb}{0,0.7,0.9}
\definecolor{axiom}{rgb}{0,0.3,0.3}
\definecolor{drawing_light_grey}{gray}{0.85}
\definecolor{background_light_grey}{gray}{0.95}

\def\let@name#1#2{
    \expandafter\let\csname #1\expandafter\endcsname\csname #2\endcsname\relax
}
\DeclareRobustCommand\crfamily{\fontfamily{ccr}\selectfont}
\DeclareTextFontCommand{\textcr}{\crfamily}

\def\ifthenelseleer#1#2#3{\ifthenelse{\equal{#1}{}}{#2}{#1#3}}
\def\bedingtesspaceexpand#1#2#3{\ifthenelseleer{\csname #1\endcsname}{#3}{#2#3}}

\def\hraum{\null\hfill\null}

\def\nvraum{\@ifnextchar\bgroup{\nvraum@c}{\nvraum@bes}}
    \def\nvraum@c#1{\vspace*{-#1\baselineskip}}
    \def\nvraum@bes{\vspace*{-\baselineskip}}
\def\forceaddspace{\relax\ifmmode\else\@\xspace\fi}
\def\forceremovespace{\relax\ifmmode\else\expandafter\@gobble\fi}

\def\send@toaux#1{\@bsphack\protected@write\@auxout{}{\string#1}\@esphack}

\def\rlabel#1[#2]#3#4#5{#5\rlabel@aux{#1}[#2]{#3}{#4}{#5}}
    \def\rlabel@aux#1[#2]#3#4#5{%
        \send@toaux{\newlabel{#1}{{\@currentlabel}{\thepage}{{\unexpanded{#5}}}{#2.\csname the#2\endcsname}{}}}\relax%
    }

\def\tag@rawscheme#1#2[#3]#4#5{\@ifnextchar[{\tag@rawscheme@{#1}{#2}[#3]{#4}{#5}}{\tag@rawscheme@{#1}{#2}[#3]{#4}{#5}[*]}}
    \def\tag@rawscheme@#1#2[#3]#4#5[#6]{\@ifnextchar\bgroup{\tag@rawscheme@@{#1}{#2}[#3]{#4}{#5}[#6]}{\tag@rawscheme@@{#1}{#2}[#3]{#4}{#5}[#6]{}}}
    \def\tag@rawscheme@@#1#2[#3]#4#5[#6]#7{%
        \ifthenelse{\equal{#6}{*}}{%
            \ifthenelse{\equal{#7}{\boolleer}}{\refstepcounter{#3}#4\csname the#3\endcsname#5}{#4#7#5}%
        }{%
            \refstepcounter{#3}#4%
            \ifthenelse{\equal{#7}{\boolleer}}{\rlabel{#6}[#3]{#1}{#2}{\csname the#3\endcsname}}{\rlabel{#6}[#3]{#1}{#2}{#7}}%
            #5%
        }%
    }
\def\tag@scheme#1#2[#3]{\tag@rawscheme{#1}{#2}[#3]{\upshape(}{\upshape)}}

\def\eqtag@post#1{\makebox[0pt][r]{#1}}
\def\eqtag@pre{\tag@scheme{Eq}{Equation}[equation]}
\def\eqtag{\@ifnextchar[{\eqtag@}{\eqtag@[*]}}
    \def\eqtag@[#1]{\@ifnextchar\bgroup{\eqtag@@[#1]}{\eqtag@@[#1]{}}}
    \def\eqtag@@[#1]#2{\eqtag@post{\eqtag@pre[#1]{#2}}}

\def\eqcref#1{\text{(\ref{#1})}}
\def\punktlabel#1{\label{it:#1:\beweislabel}}
\def\punktcref#1{\eqcref{it:#1:\beweislabel}}

\def\opfromto[#1]_#2^#3{\underset{#2}{\overset{#3}{#1}}}
\def\textoverset#1#2{\overset{\text{#1}}{#2}}

\def\eqcrefoverset#1#2{\textoverset{\eqcref{#1}}{#2}}

\def\mathclap#1{#1}
\def\oberunterset#1{\@ifnextchar^{\oberunterset@oben{#1}}{\oberunterset@unten{#1}}}
    \def\oberunterset@oben#1^#2_#3{\underset{\mathclap{#3}}{\overset{\mathclap{#2}}{#1}}}
    \def\oberunterset@unten#1_#2^#3{\underset{\mathclap{#2}}{\overset{\mathclap{#3}}{#1}}}
    \def\breitunderbrace#1_#2{\underbrace{#1}_{\mathclap{#2}}}
    \def\breitoverbrace#1^#2{\overbrace{#1}^{\mathclap{#2}}}
    \def\breitunderbracket#1_#2{\underbracket{#1}_{\mathclap{#2}}}
    \def\breitoverbracket#1^#2{\overbracket{#1}^{\mathclap{#2}}}

\def\generatenestedsecnumbering#1#2#3{%
    \expandafter\gdef\csname thelong#3\endcsname{%
        \expandafter\csname the#2\endcsname%
        \secnumberingpt%
        \expandafter\csname #1\endcsname{#3}%
    }%
    \expandafter\gdef\csname theshort#3\endcsname{%
        \expandafter\csname #1\endcsname{#3}%
    }%
}
\def\generatenestedthmnumbering#1#2#3{%
    \expandafter\gdef\csname the#3\endcsname{%
        \expandafter\csname the#2\endcsname%
        \thmnumberingpt%
        \expandafter\csname #1\endcsname{#3}%
    }%
    \expandafter\gdef\csname theshort#3\endcsname{%
        \expandafter\csname #1\endcsname{#3}%
    }%
}

\providecommand{\setcounternach}{}
\renewcommand{\setcounternach}[2]{\setcounter{#1}{#2}\addtocounter{#1}{-1}}
\providecommand{\textsubscript}{}
\renewcommand{\textsubscript}[1]{${}_{\textup{#1}}$}

\def\forcepunkt#1{#1\IfEndWith{#1}{.}{}{.}}

\def\matrix#1{\left(\begin{array}{#1}}
    \def\endmatrix{\end{array}\right)}
\def\smatrix{\left(\begin{smallmatrix}}
    \def\endsmatrix{\end{smallmatrix}\right)}

\def\multiargrekursiverbefehl#1#2#3#4#5#6#7#8{%
    \expandafter\gdef\csname#1\endcsname #2##1#4{\csname #1@anfang\endcsname##1#3\egroup}
    \expandafter\def\csname #1@anfang\endcsname##1#3{#5##1\@ifnextchar\egroup{\csname #1@ende\endcsname}{#7\csname #1@mitte\endcsname}}
    \expandafter\def\csname #1@mitte\endcsname##1#3{#6##1\@ifnextchar\egroup{\csname #1@ende\endcsname}{#7\csname #1@mitte\endcsname}}
    \expandafter\def\csname #1@ende\endcsname##1{#8}
}
\multiargrekursiverbefehl{svektor}{[}{;}{]}{\begin{smatrix}}{}{\\}{\\\end{smatrix}}
\multiargrekursiverbefehl{vektor}{[}{;}{]}{\begin{matrix}{c}}{}{\\}{\\\end{matrix}}
\multiargrekursiverbefehl{vektorzeile}{}{,}{;}{}{&}{}{}
\multiargrekursiverbefehl{matlabmatrix}{[}{;}{]}{\begin{smatrix}\vektorzeile}{\vektorzeile}{;\\}{;\end{smatrix}}

\def\BeweisRichtung[#1]{\@ifnextchar\bgroup{\@BeweisRichtung@c[#1]}{\@BeweisRichtung@bes[#1]}}
    \def\@BeweisRichtung@bes[#1]{{\bfseries (#1)}}
    \def\@BeweisRichtung@c[#1]#2#3{#2~#1~#3}
\def\erzeugeBeweisRichtungBefehle#1#2{
    \expandafter\gdef\csname #1text\endcsname##1##2{\BeweisRichtung[#2]{##1}{##2}}
    \expandafter\gdef\csname #1\endcsname{%
        \@ifnextchar\bgroup{\csname #1@\endcsname}{\csname #1text\endcsname{}{}}%
    }
    \expandafter\gdef\csname #1@\endcsname##1##2{%
        \csname #1text\endcsname{\punktcref{##1}}{\punktcref{##2}}%
    }
}
\erzeugeBeweisRichtungBefehle{hinRichtung}{$\Rightarrow$}
\erzeugeBeweisRichtungBefehle{herRichtung}{$\Leftarrow$}
\erzeugeBeweisRichtungBefehle{hinherRichtung}{$\Leftrightarrow$}

\def\cal#1{\mathcal{#1}}
\def\mathfrak#1{\mbox{\usefont{U}{euf}{m}{n}#1}}

\def\rectangleblack{\text{\RectangleBold}}

\def\squareblack{\blacksquare}

\def\create@abbreviation#1#2{
    \expandafter\gdef\csname #1\endcsname{%
        #2\@ifnextchar.{%
            \relax\ifmmode\else\expandafter\@gobble\fi%
        }{%
            \relax\ifmmode\else\@\xspace\fi%
        }%
    }
}

\create@abbreviation{cf}{\emph{cf.}}
\create@abbreviation{Cf}{\emph{Cf.}}
\create@abbreviation{thatis}{\emph{i.e.}}
\create@abbreviation{idest}{\emph{i.e.}}
\create@abbreviation{Idest}{\emph{I.e.}}
\create@abbreviation{exempli}{\emph{e.g.}}
\create@abbreviation{Exempli}{\emph{E.g.}}
\create@abbreviation{etcetera}{\emph{etc.}}
\create@abbreviation{etAlia}{\emph{et al.}}
\create@abbreviation{andothers}{\emph{et al.}}
\create@abbreviation{initself}{\emph{per se}}
\create@abbreviation{perse}{\emph{per se}}
\create@abbreviation{namely}{\emph{viz.}}
\create@abbreviation{viz}{\emph{viz.}}

\create@abbreviation{Tfae}{T.\,f.\,a.\,e.}
\create@abbreviation{tfae}{t.\,f.\,a.\,e.}
\create@abbreviation{wrt}{wrt.}
\create@abbreviation{randomvar}{r.\,v.}
\create@abbreviation{iid}{i.\,i.\,d.}
\create@abbreviation{almostEverywhere}{a.\,e.}
\create@abbreviation{almostSurely}{a.\,s.}
\create@abbreviation{withoutlog}{w.\,l.\,o.\,g.}
\create@abbreviation{Withoutlog}{W.\,l.\,o.\,g.}
\create@abbreviation{resp}{resp.}
\create@abbreviation{posdef}{p.d.}
\create@abbreviation{akin}{\`{a}~la}

\def\crefname@full#1#2#3#4#5{%
    \crefname{#1}{#2}{#3}
    \Crefname{#1}{#4}{#5}
}
\def\crefname@fullmod#1#2#3#4#5{%
    \crefname@full{#1}{#2}{#3}{#4}{#5}
    \crefname@full{#1@basic}{#2}{#3}{#4}{#5}
    \crefname@full{#1@withName}{#2}{#3}{#4}{#5}
}
\crefname@full{algorithm}{algorithm}{algorithms}{Algorithm}{Algorithms}
\crefname@full{algorithm2e}{algorithm}{algorithms}{Algorithm}{Algorithms}
\crefname@full{chapter}{chapter}{chapters}{Chapter}{Chapters}
\crefname@full{appendix}{appendix}{appendices}{Appendix}{Appendices}
\crefname@full{section}{section}{sections}{Section}{Sections}
\crefname@full{subsection}{section}{sections}{Section}{Sections}
\crefname@full{subsubsection}{section}{sections}{Section}{Sections}
\crefname@full{subsubsubsection}{section}{sections}{Section}{Sections}
\crefname@full{table}{table}{tables}{Table}{Tables}
\crefname@full{figure}{figure}{figures}{Figure}{Figures}
\crefname@full{subfigure}{figure}{figures}{Figure}{Figures}

\crefname@fullmod{thm}{theorem}{theorems}{Theorem}{Theorems}
\crefname@fullmod{thmStar}{theorem}{theorems}{Theorem}{Theorems}
\crefname@fullmod{conj}{conjecture}{conjectures}{Conjecture}{Conjectures}
\crefname@fullmod{cor}{corollary}{corollaries}{Corollary}{Corollaries}
\crefname@fullmod{defn}{definition}{definitions}{Definition}{Definitions}
\crefname@fullmod{conv}{convention}{conventions}{Convention}{Conventions}
\crefname@fullmod{e.g.}{example}{examples}{Example}{Examples}
\crefname@fullmod{prop}{proposition}{propositions}{Proposition}{Propositions}
\crefname@fullmod{proof}{proof}{proofs}{Proof}{Proofs}
\crefname@fullmod{lemm}{lemma}{lemmata}{Lemma}{Lemmata}
\crefname@fullmod{problem}{problem}{problems}{Problem}{Problems}
\crefname@fullmod{qstn}{question}{questions}{Question}{Questions}
\crefname@fullmod{rem}{remark}{remarks}{Remark}{Remarks}

\def\qedEIGEN#1{\@ifnextchar[{\qedEIGEN@c{#1}}{\qedEIGEN@bes{#1}}}
\def\qedEIGEN@bes#1{%
    \parfillskip=0pt
    \widowpenalty=10000
    \displaywidowpenalty=10000
    \finalhyphendemerits=0
    \leavevmode
    \unskip
    \nobreak
    \hfil
    \penalty50
    \hskip.2em
    \null
    \hfill
    #1
    \par%
}
\def\qedEIGEN@c#1[#2]{%
    \parfillskip=0pt
    \widowpenalty=10000
    \displaywidowpenalty=10000
    \finalhyphendemerits=0
    \leavevmode
    \unskip
    \nobreak
    \hfil
    \penalty50
    \hskip.2em
    \null
    \hfill
    {#1~{\small\bfseries\upshape (#2)}}%
    \par%
}
\def\qedVARIANT#1#2{
    \expandafter\def\csname ennde#1Sign\endcsname{#2}
    \expandafter\def\csname ennde#1\endcsname{\@ifnextchar[{\qedEIGEN@c{#2}}{\qedEIGEN@bes{#2}}} 
}
\qedVARIANT{OfProof}{$\squareblack$}
\qedVARIANT{OfWork}{\rectangleblack}
\qedVARIANT{OfSomething}{$\lrcorner$}
\qedVARIANT{OnNeutral}{} 

\def\ra@pretheoremwork{
    \setlength{\theorempreskipamount}{\ownspaceabovethm}
    \setlength{\theorempostskipamount}{\ownspacebelowthm}
}
\def\rathmtransfer#1#2{
    \expandafter\def\csname #2\endcsname{\csname #1\endcsname}
    \expandafter\def\csname end#2\endcsname{\csname end#1\endcsname}
}

\def\ranewthm#1#2#3[#4]{
    \theoremstyle{\current@theoremstyle}
    \theoremseparator{\current@theoremseparator}
    \theoremprework{\ra@pretheoremwork}
    \@ifundefined{#1@basic}{\newtheorem{#1@basic}[#4]{#2}}{\renewtheorem{#1@basic}[#4]{#2}}
    \theoremstyle{\current@theoremstyle}
    \theoremseparator{\thmForceSepPt}
    \theoremprework{\ra@pretheoremwork}
    \@ifundefined{#1@withName}{\newtheorem{#1@withName}[#4]{#2}}{\renewtheorem{#1@withName}[#4]{#2}}
    \theoremstyle{nonumberplain}
    \theoremseparator{\thmForceSepPt}
    \theoremprework{\ra@pretheoremwork}
    \@ifundefined{#1@star@basic}{\newtheorem{#1@star@basic}[#4]{#2}}{\renewtheorem{#1@star@basic}[#4]{#2}}
    \theoremstyle{nonumberplain}
    \theoremseparator{\thmForceSepPt}
    \theoremprework{\ra@pretheoremwork}
    \@ifundefined{#1@star@withName}{\newtheorem{#1@star@withName}[#4]{#2}}{\renewtheorem{#1@star@withName}[#4]{#2}}
    \umbauenenv{#1}{#3}[#4]
    \umbauenenv{#1@star}{#3}[#4]
    \rathmtransfer{#1@star}{#1*}
}

\def\umbauenenv#1#2[#3]{%
    \expandafter\def\csname #1\endcsname{\relax%
        \@ifnextchar[{\csname #1@\endcsname}{\csname #1@\endcsname[*]}%
    }
    \expandafter\def\csname #1@\endcsname[##1]{\relax%
        \@ifnextchar[{\csname #1@@\endcsname[##1]}{\csname #1@@\endcsname[##1][*]}%
    }
    \expandafter\def\csname #1@@\endcsname[##1][##2]{%
        \ifx*##1%
            \def\enndeOfBlock{\csname end#1@basic\endcsname}
            \csname #1@basic\endcsname%
        \else%
            \def\enndeOfBlock{\csname end#1@withName\endcsname}
            \csname #1@withName\endcsname[##1]%
        \fi%
        \def\makelabel####1{%
            \gdef\beweislabel{####1}%
            \label{\beweislabel}%
        }%
        \ifx*##2%
            \def\enndeSymbol{\qedEIGEN{#2}}
        \else%
            \def\enndeSymbol{\qedEIGEN{#2}[##2]}
        \fi
    }
    \expandafter\gdef\csname end#1\endcsname{\enndeSymbol\enndeOfBlock}
}

    \def\current@theoremstyle{plain}
    \def\current@theoremseparator{\thmnumberingseppt}
    \theoremstyle{\current@theoremstyle}
    \theoremseparator{\current@theoremseparator}
    \theoremsymbol{}

\generatenestedthmnumbering{arabic}{section}{equation}

\generatenestedthmnumbering{arabic}{section}{X}
\generatenestedthmnumbering{Roman}{section}{Xsp}

    \theoremheaderfont{\upshape\bfseries}
    \theorembodyfont{\slshape}

\ranewthm{thm}{Theorem}{\enndeOfSomethingSign}[X]
\ranewthm{lemm}{Lemma}{\enndeOfSomethingSign}[X]
\ranewthm{cor}{Corollary}{\enndeOfSomethingSign}[X]
\ranewthm{prop}{Proposition}{\enndeOfSomethingSign}[X]

    \theorembodyfont{\upshape}

\ranewthm{defn}{Definition}{\enndeOfSomethingSign}[X]
\ranewthm{conv}{Convention}{\enndeOfSomethingSign}[X]
\ranewthm{e.g.}{Example}{\enndeOfSomethingSign}[X]
\ranewthm{fact}{Fact}{\enndeOfSomethingSign}[X]
\ranewthm{problem}{Problem}{\enndeOfSomethingSign}[X]
\ranewthm{qstn}{Question}{\enndeOfSomethingSign}[X]
\ranewthm{rem}{Remark}{\enndeOfSomethingSign}[X]

    \theoremheaderfont{\itshape}
    \theorembodyfont{\upshape}

\ranewthm{proof@tmp}{Proof}{\enndeOfProofSign}[Xdisplaynone]
\rathmtransfer{proof@tmp*}{proof}

\def\shortclaim@claim{%
    \iflanguage{british}{Claim}{%
    \iflanguage{english}{Claim}{%
    \iflanguage{ngerman}{Behauptung}{%
    \iflanguage{russian}{Утверждение}{%
    Claim%
    }}}}%
}
\def\shortclaim@pf@kurz{%
    \iflanguage{british}{Pf}{%
    \iflanguage{english}{Pf}{%
    \iflanguage{ngerman}{Bew}{%
    \iflanguage{russian}{Доказательство}{%
    Pf%
    }}}}%
}
\def\shortclaim{\@ifnextchar\bgroup{\shortclaim@c}{\shortclaim@bes}}
    \def\shortclaim@c#1{\item[{\bfseries \shortclaim@claim\forceaddspace #1.}]}
    \def\shortclaim@bes{\item[{\bfseries \shortclaim@claim.}]}
\def\proofofshortclaim{\item[{\bfseries\itshape\shortclaim@pf@kurz.}]}

\newdateformat{standardshort}{\oldstylenums{\THEYEAR}.\oldstylenums{\THEMONTH}.\oldstylenums{\THEDAY}}
\newdateformat{standardcompact}{\THEYEAR\twodigit{\THEMONTH}\twodigit{\THEDAY}}
\newdateformat{standardlong}{\THEYEAR\ \monthname\ \THEDAY}
\newcolumntype{\RECHTS}[1]{>{\raggedleft}p{#1}}
\newcolumntype{\LINKS}[1]{>{\raggedright}p{#1}}
\newcolumntype{m}{>{$}l<{$}}
\newcolumntype{C}{>{$}c<{$}}
\newcolumntype{L}{>{$}l<{$}}
\newcolumntype{R}{>{$}r<{$}}
\newcolumntype{0}{@{\hspace{0pt}}}
\newcolumntype{\LINKSRAND}{@{\hspace{\@totalleftmargin}}}
\newcolumntype{h}{@{\extracolsep{\fill}}}
\newcolumntype{i}{>{\itshape}}
\newcolumntype{t}{@{\hspace{\tabcolsep}}}
\newcolumntype{q}{@{\hspace{1em}}}
\newcolumntype{n}{@{\hspace{-\tabcolsep}}}
\newcolumntype{M}[2]{%
    >{\begin{minipage}{#2}\begin{math}}%
    {#1}%
    <{\end{math}\end{minipage}}%
}
\newcolumntype{T}[2]{%
    >{\begin{minipage}{#2}}%
    {#1}%
    <{\end{minipage}}%
}
\def\punkteumgebung@genbefehl#1#2#3{
    \punkteumgebung@genbefehl@{#1}{#2}{#3}{}{}
    \punkteumgebung@genbefehl@{multi#1}{#2}{#3}{
        \setlength{\columnsep}{10pt}%
        \setlength{\columnseprule}{0pt}%
        \begin{multicols}{\thecolumnanzahl}%
    }{\end{multicols}\nvraum{1}}
}
\def\punkteumgebung@genbefehl@#1#2#3#4#5{
    \expandafter\gdef\csname #1\endcsname{
        \@ifnextchar\bgroup{\csname #1@c\endcsname}{\csname #1@bes\endcsname}
    }
        \expandafter\def\csname #1@c\endcsname##1{
            \@ifnextchar[{\csname #1@c@\endcsname{##1}}{\csname #1@c@\endcsname{##1}[\z@]}
        }
        \expandafter\def\csname #1@c@\endcsname##1[##2]{
            \@ifnextchar[{\csname #1@c@@\endcsname{##1}[##2]}{\csname #1@c@@\endcsname{##1}[##2][\z@]}
        }
        \expandafter\def\csname #1@c@@\endcsname##1[##2][##3]{
            \let\alterlinkerRand\gesamtlinkerRand
            \let\alterrechterRand\gesamtrechterRand
            \addtolength{\gesamtlinkerRand}{##2}
            \addtolength{\gesamtrechterRand}{##3}
            \advance\linewidth -##2%
            \advance\linewidth -##3%
            \advance\@totalleftmargin ##2%
            \parshape\@ne \@totalleftmargin\linewidth%
            #4
            \begin{#2}[\upshape ##1]%
                \setlength{\parskip}{0.5\baselineskip}\relax%
                \setlength{\topsep}{\z@}\relax%
                \setlength{\partopsep}{\z@}\relax%
                \setlength{\parsep}{\parskip}\relax%
                \setlength{\itemsep}{#3}\relax%
                \setlength{\listparindent}{\z@}\relax%
                \setlength{\itemindent}{\z@}\relax%
        }
        \expandafter\def\csname #1@bes\endcsname{
            \@ifnextchar[{\csname #1@bes@\endcsname}{\csname #1@bes@\endcsname[\z@]}
        }
        \expandafter\def\csname #1@bes@\endcsname[##1]{
            \@ifnextchar[{\csname #1@bes@@\endcsname[##1]}{\csname #1@bes@@\endcsname[##1][\z@]}
        }
        \expandafter\def\csname #1@bes@@\endcsname[##1][##2]{
            \let\alterlinkerRand\gesamtlinkerRand
            \let\alterrechterRand\gesamtrechterRand
            \addtolength{\gesamtlinkerRand}{##1}
            \addtolength{\gesamtrechterRand}{##2}
            \advance\linewidth -##1%
            \advance\linewidth -##2%
            \advance\@totalleftmargin ##1%
            \parshape\@ne \@totalleftmargin\linewidth%
            #4
            \begin{#2}%
                \setlength{\parskip}{0.5\baselineskip}\relax%
                \setlength{\topsep}{\z@}\relax%
                \setlength{\partopsep}{\z@}\relax%
                \setlength{\parsep}{\parskip}\relax%
                \setlength{\itemsep}{#3}\relax%
                \setlength{\listparindent}{\z@}\relax%
                \setlength{\itemindent}{\z@}\relax%
        }
    \expandafter\gdef\csname end#1\endcsname{%
        \end{#2}#5
        \setlength{\gesamtlinkerRand}{\alterlinkerRand}
        \setlength{\gesamtlinkerRand}{\alterrechterRand}
    }
}

\def\ritempunkt{{\Large \textbullet}} 
\setdefaultitem{\ritempunkt}{\ritempunkt}{\ritempunkt}{\ritempunkt}
\punkteumgebung@genbefehl{itemise}{compactitem}{\parskip}{}{}
\punkteumgebung@genbefehl{kompaktitem}{compactitem}{\z@}{}{}
\punkteumgebung@genbefehl{kompaktenum}{compactenum}{\z@}{}{}

\def\enumerate{%
    \@ifnextchar\bgroup{%
        \enumerate@legacyarg%
    }{%
        \@ifnextchar[{\enumerate@args}{\enumerate@noargs}
    }%
}
    \def\enumerate@spacing{
        \setlength{\parskip}{0.5\baselineskip}\relax%
        \setlength{\topsep}{\z@}\relax%
        \setlength{\partopsep}{\z@}\relax%
        \setlength{\parsep}{\parskip}\relax%
        \setlength{\itemsep}{\z@}\relax%
        \setlength{\listparindent}{\z@}\relax%
        \setlength{\itemindent}{\z@}\relax%
    }
    \def\enumerate@noargs{
        \begin{oldenumerate}
        \enumerate@spacing
    }
    \def\enumerate@args[#1]{
        \begin{oldenumerate}[#1]
        \enumerate@spacing
    }
    \def\enumerate@legacyarg#1{
        \begin{oldenumerate}[label=#1]
        \enumerate@spacing
    }
    \def\endenumerate{%
        \end{oldenumerate}
    }

\def\shorteqnarray{%
    \bgroup
    \setlength{\abovedisplayshortskip}{\aboveequation}%
    \setlength{\belowdisplayshortskip}{\belowequation}%
    \setlength{\abovedisplayskip}{\aboveequation - \baselineskip}%
    \setlength{\belowdisplayskip}{\belowequation}%
    \begin{eqnarray*}%
}
\def\endshorteqnarray{%
    \end{eqnarray*}%
    \egroup
}
\def\longeqnarray{%
    \bgroup%
    \allowdisplaybreaks%
    \setlength{\abovedisplayshortskip}{\aboveequation}%
    \setlength{\belowdisplayshortskip}{\belowequation}%
    \setlength{\abovedisplayskip}{\aboveequation - \baselineskip}%
    \setlength{\belowdisplayskip}{\belowequation}%
    \begin{eqnarray*}
}
\def\endlongeqnarray{%
    \end{eqnarray*}%
    \egroup%
}
\def\displayarray[#1]#2{
    \bgroup
    \everymath={\displaystyle}
    \begin{array}[#1]{#2}
}
\def\enddisplayarray{
    \end{array}
    \egroup
}

\def\matrix#1{\left(\begin{array}[mc]{#1}}
    \def\endmatrix{\end{array}\right)}
\def\smatrix{\left(\begin{smallmatrix}}
    \def\endsmatrix{\end{smallmatrix}\right)}

\def\multiargrekursiverbefehl#1#2#3#4#5#6#7#8{%
    \expandafter\gdef\csname#1\endcsname #2##1#4{\csname #1@anfang\endcsname##1#3\egroup}
    \expandafter\def\csname #1@anfang\endcsname##1#3{#5##1\@ifnextchar\egroup{\csname #1@ende\endcsname}{#7\csname #1@mitte\endcsname}}
    \expandafter\def\csname #1@mitte\endcsname##1#3{#6##1\@ifnextchar\egroup{\csname #1@ende\endcsname}{#7\csname #1@mitte\endcsname}}
    \expandafter\def\csname #1@ende\endcsname##1{#8}
}
\multiargrekursiverbefehl{svektor}{[}{;}{]}{\begin{smatrix}}{}{\\}{\\\end{smatrix}}
\multiargrekursiverbefehl{vektor}{[}{;}{]}{\begin{matrix}{c}}{}{\\}{\\\end{matrix}}
\multiargrekursiverbefehl{vektorzeile}{}{,}{;}{}{&}{}{}
\multiargrekursiverbefehl{matlabmatrix}{[}{;}{]}{\begin{smatrix}\vektorzeile}{\vektorzeile}{;\\}{;\end{smatrix}}

\def\underbracenodisplay#1{%
    \mathop{\vtop{\m@th\ialign{##\crcr
    $\hfil\displaystyle{#1}\hfil$\crcr
    \noalign{\kern3\p@\nointerlineskip}%
    \upbracefill\crcr\noalign{\kern3\p@}}}}\limits%
}

\def\changemargins{\@ifnextchar[{\indents@}{\indents@[\z@]}}
\def\indents@[#1]{\@ifnextchar[{\indents@@[#1]}{\indents@@[#1][\z@]}}
\def\indents@@[#1][#2]{%
    \begin{list}{}{\relax
        \setlength{\topsep}{\z@}\relax
        \setlength{\partopsep}{\z@}\relax
        \setlength{\parsep}{\parskip}\relax
        \setlength{\listparindent}{\z@}\relax
        \setlength{\itemindent}{\z@}\relax
        \setlength{\leftmargin}{#1}\relax
        \setlength{\rightmargin}{#2}\relax
        \let\alterlinkerRand\gesamtlinkerRand
        \let\alterrechterRand\gesamtrechterRand
        \addtolength{\gesamtlinkerRand}{#1}
        \addtolength{\gesamtrechterRand}{#2}
    }\relax
        \item[]\relax
}
    \def\endchangemargins{%
        \setlength{\gesamtlinkerRand}{\alterlinkerRand}
        \setlength{\gesamtlinkerRand}{\alterrechterRand}
        \end{list}%
    }

\def\indentonce{\begin{changemargins}[\rtab][\rtab]}
    \def\endindentonce{\end{changemargins}}

\def\restoremargins{\begin{changemargins}[-\gesamtlinkerRand][-\gesamtrechterRand]}
    \def\endrestoremargins{\end{changemargins}}

\def\programmiercode{
    \modulolinenumbers[1]
    \begin{changemargins}[\rtab][\rtab]%
    \begin{linenumbers}%
        \fontfamily{cmtt}\fontseries{m}\fontshape{u}\selectfont%
        \setlength{\parskip}{1\baselineskip}%
        \setlength{\parindent}{0pt}%
}
    \def\endprogrammiercode{
        \end{linenumbers}
        \end{changemargins}
    }

\def\schattiertebox@genbefehl#1#2#3{
    \expandafter\gdef\csname #1\endcsname{%
        \@ifnextchar[{\csname #1@args\endcsname}{\csname #1@args\endcsname[#3]}
    }
        \expandafter\def\csname #1@args\endcsname[##1]{%
            \@ifnextchar[{\csname #1@args@l\endcsname[##1]}{\csname #1@args@n\endcsname[##1]}
        }
        \expandafter\def\csname #1@args@l\endcsname[##1][##2]{%
            \@ifnextchar[{\csname #1@args@l@r\endcsname[##1][##2]}{\csname #1@args@l@n\endcsname[##1][##2]}
        }
        \expandafter\def\csname #1@args@n\endcsname[##1]{%
            \let\boolinmdframed\boolwahr
            \begin{mdframed}[#2leftmargin=0,rightmargin=0,outermargin=0,innermargin=0,##1]
        }
        \expandafter\def\csname #1@args@l@n\endcsname[##1][##2]{%
            \let\boolinmdframed\boolwahr
            \begin{mdframed}[#2leftmargin=##2/2,rightmargin=##2/2,outermargin=##2/2,innermargin=##2/2,##1]
        }
        \expandafter\def\csname #1@args@l@r\endcsname[##1][##2][##3]{%
            \let\boolinmdframed\boolwahr
            \begin{mdframed}[#2leftmargin=##2,rightmargin=##3,outermargin=##2,innermargin=##3,##1]
        }
    \expandafter\gdef\csname end#1\endcsname{%
        \end{mdframed}
        \let\boolinmdframed\boolfalsch
    }
}
    \schattiertebox@genbefehl{schattiertebox}{
        splittopskip=0,%
        splitbottomskip=0,%
        frametitleaboveskip=0,%
        frametitlebelowskip=0,%
        skipabove=1\baselineskip,%
        skipbelow=1\baselineskip,%
        linewidth=2pt,%
        linecolor=black,%
        roundcorner=4pt,%
    }{
        backgroundcolor=leer,%
        nobreak=true,%
    }

    \schattiertebox@genbefehl{schattierteboxdunn}{
        splittopskip=0,%
        splitbottomskip=0,%
        frametitleaboveskip=0,%
        frametitlebelowskip=0,%
        skipabove=1\baselineskip,%
        skipbelow=1\baselineskip,%
        linewidth=1pt,%
        linecolor=black,%
        roundcorner=2pt,%
    }{
        backgroundcolor=leer,%
        nobreak=true,%
    }

    \schattiertebox@genbefehl{highlightbox}{
        splittopskip=0,%
        splitbottomskip=0,%
        frametitleaboveskip=0,%
        frametitlebelowskip=0,%
        skipabove=1\baselineskip,%
        skipbelow=1\baselineskip,%
        linewidth=0pt,%
        linecolor=black,%
        roundcorner=2pt,%
    }{
        backgroundcolor=background_light_grey,%
        nobreak=true,%
    }

    \schattiertebox@genbefehl{highlightboxWithBreaks}{
        splittopskip=0,%
        splitbottomskip=0,%
        frametitleaboveskip=0,%
        frametitlebelowskip=0,%
        skipabove=1\baselineskip,%
        skipbelow=1\baselineskip,%
        linewidth=0pt,%
        linecolor=black,%
        roundcorner=2pt,%
    }{
        backgroundcolor=background_light_grey,%
    }

\def\tikzsetzepfeil#1{%
    \begin{tikzpicture}[remember picture,overlay,>=latex]%
        \draw #1;%
    \end{tikzpicture}%
}

\def\tikzsetzekreise[#1]#2#3{%
    \tikzsetzepfeil{%
    [rounded corners,#1]%
        ([shift={(-\tabcolsep,0.75\baselineskip)}]#2)%
        rectangle%
        ([shift={(\tabcolsep,-0.5\baselineskip)}]#3)
    }%
}

\tikzset{
    >=stealth,
    auto,
    node distance=1cm,
    thick,
    main node/.style={
        circle,draw,font=\sffamily\Large\bfseries,minimum size=0pt
    },
    state/.style={minimum size=0pt}
    loop above right/.style={loop,out=30,in=60,distance=0.5cm},
    loop above left/.style={above left,out=150,in=120,loop},
    loop below right/.style={below right,out=330,in=300,loop},
    loop below left/.style={below left,out=240,in=210,loop},
    itria/.style={
        draw,dashed,shape border uses incircle,
        isosceles triangle,shape border rotate=90,yshift=-1.45cm
    },
    rtria/.style={
        draw,dashed,shape border uses incircle,
        isosceles triangle,isosceles triangle apex angle=90,
        shape border rotate=-45,yshift=0.2cm,xshift=0.5cm
    },
    ritria/.style={
        draw,dashed,shape border uses incircle,
        isosceles triangle,isosceles triangle apex angle=110,
        shape border rotate=-55,yshift=0.1cm
    },
    litria/.style={
        draw,dashed,shape border uses incircle,
        isosceles triangle,isosceles triangle apex angle=110,
        shape border rotate=235,yshift=0.1cm
    }
}

\renewenvironment{cases}[0]{\left\{\begin{array}[c]{0lcl}}{\end{array}\right.}

\renewenvironment{displaycases}[0]{%
    \left\{\relax%
    \bgroup\relax%
    \everymath={\displaystyle}\relax%
    \begin{array}[c]{0lcl}\relax%
}{%
    \end{array}\relax%
    \egroup\relax%
    \right.\relax%
}

\providecommand{\usesinglequotes}{}
\renewcommand{\usesinglequotes}[1]{`#1'}

\providecommand{\zeroone}{}
\renewcommand{\zeroone}[0]{\textup{0\=/1}\relax\ifmmode\else\@\xspace\fi}
\providecommand{\onetoone}{}
\renewcommand{\onetoone}[0]{\ensuremath{1\!\!:\!\!1}\relax\ifmmode\else\@\xspace\fi}
\providecommand{\First}{}
\renewcommand{\First}[0]{\text{I\textsuperscript{st}}\relax\ifmmode\else\@\xspace\fi}
\providecommand{\Second}{}
\renewcommand{\Second}[0]{\text{II\textsuperscript{nd}}\relax\ifmmode\else\@\xspace\fi}
\providecommand{\Third}{}
\renewcommand{\Third}[0]{\text{III\textsuperscript{rd}}\relax\ifmmode\else\@\xspace\fi}

\providecommand{\TextCStarAlg}{}
\renewcommand{\TextCStarAlg}[0]{\text{C\textsuperscript{\ensuremath{\ast}}\=/algebra}\relax\ifmmode\else\@\xspace\fi}
\providecommand{\TextCStarSubAlg}{}
\renewcommand{\TextCStarSubAlg}[0]{\text{C\textsuperscript{\ensuremath{\ast}}\=/subalgebra}\relax\ifmmode\else\@\xspace\fi}
\providecommand{\TextCStarAlgs}{}
\renewcommand{\TextCStarAlgs}[0]{\text{C\textsuperscript{\ensuremath{\ast}}\=/algebras}\relax\ifmmode\else\@\xspace\fi}
\providecommand{\TextCStarSubAlgs}{}
\renewcommand{\TextCStarSubAlgs}[0]{\text{C\textsuperscript{\ensuremath{\ast}}\=/subalgebras}\relax\ifmmode\else\@\xspace\fi}

\providecommand{\envPreMathsLong}{}
\renewcommand{\envPreMathsLong}[0]{%
    \bgroup\relax%
    \let\old@arraystretch\arraystretch\relax%
    \renewcommand\arraystretch{1.2}\relax\relax%
}
\providecommand{\envPostMathsLong}{}
\renewcommand{\envPostMathsLong}[0]{%
    \renewcommand\arraystretch{\old@arraystretch}\relax%
    \egroup\relax%
}
\providecommand{\id}{}
\renewcommand{\id}[0]{\mathrm{\textit{id}}}

\providecommand{\complex}{}
\renewcommand{\complex}[0]{\mathbb{C}}

\providecommand{\reals}{}
\renewcommand{\reals}[0]{\mathbb{R}}
\providecommand{\realsNonNeg}{}
\renewcommand{\realsNonNeg}[0]{\reals_{\geq 0}}

\providecommand{\rationals}{}
\renewcommand{\rationals}[0]{\mathbb{Q}}

\providecommand{\naturals}{}
\renewcommand{\naturals}[0]{\mathbb{N}}
\providecommand{\naturalsZero}{}
\renewcommand{\naturalsZero}[0]{\mathbb{N}_{0}}

\providecommand{\HilbertRaum}{}
\renewcommand{\HilbertRaum}[0]{\mathcal{H}}

\providecommand{\GenSet}{}
\renewcommand{\GenSet}[1]{\langle #1 \rangle}
\providecommand{\GenSetLong}{}
\renewcommand{\GenSetLong}[1]{\Big\langle #1 \Big\rangle}
\providecommand{\GenSetBy}{}
\renewcommand{\GenSetBy}[2]{\langle #1 \mathrel{\vert} #2 \rangle}
\providecommand{\GenSetByLong}{}
\renewcommand{\GenSetByLong}[2]{\Big\langle #1 \mathrel{\Big\vert} #2 \Big\rangle}
\providecommand{\Proj}{}
\renewcommand{\Proj}[0]{\mathrm{Proj}}

\providecommand{\oBall}{}
\renewcommand{\oBall}[2]{\cal{B}_{#2}(#1)}

\providecommand{\topSOT}{}
\renewcommand{\topSOT}[0]{\text{\upshape\scshape sot}}

\providecommand{\topWOT}{}
\renewcommand{\topWOT}[0]{\text{\upshape\scshape wot}}

\providecommand{\card}{}
\renewcommand{\card}[1]{\lvert #1 \rvert}

\providecommand{\einser}{}
\renewcommand{\einser}[0]{1\!\!1}

\providecommand{\abs}{}
\renewcommand{\abs}[1]{\lvert #1 \rvert}

\providecommand{\linspann}{}
\renewcommand{\linspann}[0]{\textup{lin}}

\providecommand{\ad}{}
\renewcommand{\ad}[1]{\text{\upshape Ad}_{#1}}
\providecommand{\tr}{}
\renewcommand{\tr}[0]{\text{\upshape tr}}
\providecommand{\onematrix}{}
\renewcommand{\onematrix}[0]{\text{\upshape\bfseries I}}
\providecommand{\zeromatrix}{}
\renewcommand{\zeromatrix}[0]{\mathbf{0}}

\providecommand{\norm}{}
\renewcommand{\norm}[1]{\lVert #1 \rVert}

\providecommand{\normLarge}{}
\renewcommand{\normLarge}[1]{\left\| #1 \right\|}

\providecommand{\BoundedOpsSymbol}{}
\renewcommand{\BoundedOpsSymbol}[0]{\mathfrak{L}}
\providecommand{\FiniteRankOpsSymbol}{}
\renewcommand{\FiniteRankOpsSymbol}[0]{\mathfrak{L}_{0}}
\providecommand{\BaseVector}{}
\renewcommand{\BaseVector}[1]{\mathbf{e}_{#1}}
\providecommand{\ElementaryMatrix}{}
\renewcommand{\ElementaryMatrix}[2]{\mathbf{E}_{#1,#2}}
\providecommand{\Choi}{}
\renewcommand{\Choi}[1]{\mathcal{C}_{#1}}

\def\Cts{\@ifnextchar_{\Cts@tief}{\Cts@tief_{}}}
    \def\Cts@tief_#1#2{\@ifnextchar\bgroup{\Cts@two_{#1}{#2}}{\Cts@one_{#1}{#2}}}
    \def\Cts@one_#1#2{C_{#1}\big(#2\big)}
    \def\Cts@two_#1#2#3{C_{#1}\big(#2,~#3\big)}

\def\BoundedOps#1{\@ifnextchar\bgroup{\BoundedOps@two{#1}}{\mathop{\BoundedOpsSymbol}(#1)}}
    \def\BoundedOps@two#1#2{\mathop{\BoundedOpsSymbol}(#1,#2)}
\def\BoundedOpsInv#1{\@ifnextchar\bgroup{\BoundedOps@two{#1}}{\mathop{\BoundedOpsSymbol}(#1)^{\times}}}
    \def\BoundedOpsInv@two#1#2{\mathop{\BoundedOpsSymbol}(#1,#2)^{\times}}
\def\FiniteRankOps#1{\@ifnextchar\bgroup{\FiniteRankOps@two{#1}}{\mathop{\FiniteRankOpsSymbol}(#1)}}
    \def\FiniteRankOps@two#1#2{\mathop{\FiniteRankOpsSymbol}(#1,#2)}
\def\FiniteRankOpsInv#1{\@ifnextchar\bgroup{\FiniteRankOps@two{#1}}{\mathop{\FiniteRankOpsSymbol}(#1)^{\times}}}
    \def\FiniteRankOpsInv@two#1#2{\mathop{\FiniteRankOpsSymbol}(#1,#2)^{\times}}

\def\restr#1{\vert_{#1}}
\def\without{\mathbin{\setminus}}

\let\altphi\phi
\let\altvarphi\varphi
    \def\phi{\altvarphi}
    \def\varphi{\altphi}
\def\quer#1{\overline{#1}}

\def\lim{\mathop{\ell\mathrm{im}}}
\def\supp{\mathop{\textup{supp}}}
\def\dim{\mathop{\textup{dim}}}

\def\ran{\mathop{\textup{ran}}}

\def\Matr{\mathop{\cal{M}}}

\def\Re{\mathop{\mathfrak{R}\mathrm{e}}}

\def\tinytopSOT{\text{\scriptsize\upshape \scshape sot}}

\def\Generate{\@ifnextchar[{\Generate@named}{\Generate@plain}}
    \def\Generate@named[#1]#2{\@ifnextchar\bgroup{\mathrm{#1}{}\GenSetBy{#2}}{\mathrm{#1}{}\GenSet{#2}}}
    \def\Generate@plain#1{\@ifnextchar\bgroup{\GenSetBy{#1}}{\GenSet{#1}}}
\def\GenerateLong{\@ifnextchar[{\GenerateLong@named}{\GenerateLong@plain}}
    \def\GenerateLong@named[#1]#2{\@ifnextchar\bgroup{\mathrm{#1}{}\GenSetByLong{#2}}{\mathrm{#1}{}\GenSetLong{#2}}}
    \def\GenerateLong@plain#1{\@ifnextchar\bgroup{\GenSetByLong{#1}}{\GenSetLong{#1}}}

\@ifundefined{setcitestyle}{%
}{%
    \setcitestyle{numeric-comp,open={[},close={]}}
}

\allowdisplaybreaks 
\renewcommand{\arraystretch}{1}
\def\firstparagraph{\noindent}
\def\continueparagraph{\noindent}

\generatenestedsecnumbering{arabic}{section}{subsection}
\generatenestedsecnumbering{arabic}{subsection}{subsubsection}

\def\sectionname{}
\def\subsectionname{}
\def\subsubsectionname{}

\def\documentpartnormal{
    \let\boolinappendix\boolfalse
    \addtocontents{toc}{\protect\setcounter{tocdepth}{1}}
    \def\chaptername{Chapter}
    \def\sectionname{}
    \def\subsectionname{}
    \def\subsubsectionname{}
    \generatenestedsecnumbering{arabic}{section}{subsection}
}

\def\documentpartappendix{
    \appendix
    \let\boolinappendix\boolwahr
    \pagenumberinghandledelay
    \addcontentsline{toc}{section}{Appendices}
    \addtocontents{toc}{\protect\setcounter{tocdepth}{0}}
    \def\sectionname{Appendix}
    \generatenestedsecnumbering{Alph}{section}{subsection}
}

\def\@settitle{%
    \bgroup\relax
    \centering
    \LARGE\relax
    \scshape\relax
    \@title\relax
    \egroup\relax
}

\def\@seccntformat#1{%
    \protect\textup{%
        \protect\@secnumfont
        \expandafter\protect\csname format#1\endcsname%
        \csname the#1\endcsname
        \expandafter\protect\csname format#1@pt\endcsname%
        \space
    }%
}

\def\formatsection@text{\centering\Large\scshape}
\def\formatsection@pt{\secnumberingseppt}
\def\section{\@startsection{section}{1}{\z@}{.7\linespacing\@plus\linespacing}{.5\linespacing}{\formatsection@text}}

\def\formatsubsection@text{\flushleft\bfseries\scshape}
\def\formatsubsection@pt{\subsecnumberingseppt}
\def\subsection{\@startsection{subsection}{2}{\z@}{\z@}{\z@\hspace{1em}}{\formatsubsection@text}}

\renewcommand{\paragraph}[1]{%
    {\itshape #1}\:%
}

\def\pagenumbering#1{
    \gdef\thepage{\csname @#1\endcsname\c@page}
}

\def\pagenumberinghandledelay{}

\DefineFNsymbols*{customAlphabet}{abcdefghijklmnopqrstuvwxyz}
\setfnsymbol{customAlphabet}

\def\footnotemark[#1]{\text{\textsuperscript{\getrefnumber{#1}}}}

\def\footnote@custom@period{24}
\providecommand{\footnote@ctr@prebump}{}
\renewcommand{\footnote@ctr@prebump}[1]{%
    \ifnum\value{#1}<\footnote@custom@period%
    \else\relax
        \setcounter{#1}{0}%
    \fi%
}
\providecommand{\footnoteref}{}
\renewcommand{\footnoteref}[1]{\protected@xdef\@thefnmark{\ref{#1}}\@footnotemark}
\let\@old@footnotetext\footnotetext
\def\footnotetext[#1]#2{%
    \footnote@ctr@prebump{footnote}%
    \addtocounter{footnote}{1}%
    \@old@footnotetext[\value{footnote}]{\label{#1}#2}%
}

\let\@old@footnote\footnote
\renewcommand{\footnote}[1]{%
    \footnote@ctr@prebump{footnote}%
    \@old@footnote{#1}%
}

\def\kopfzeiledefault{
    \lhead[]{}
    \lhead[]{}
    \chead[]{}
    \rhead[]{}
    \lfoot[]{}
    \cfoot{\footnotesize\thepage}
    \rfoot[]{}
}

\def\aktuellesfont{\csname lmodern\endcsname}
\def\documentfont{%
    \gdef\aktuellesfont{\csname lmodern\endcsname}%
    \fontfamily{lmr}\fontseries{m}\selectfont%
    \renewcommand{\sfdefault}{phv}%
    \renewcommand{\ttdefault}{pcr}%
    \renewcommand{\rmdefault}{cmr}
    \renewcommand{\bfdefault}{bx}%
    \renewcommand{\itdefault}{it}%
    \renewcommand{\sldefault}{sl}%
    \renewcommand{\scdefault}{sc}%
    \renewcommand{\updefault}{n}%
}

\SetAlgoHangIndent{0pt}
\DontPrintSemicolon
\def\Line#1{#1\;}
\def\LineNoNr#1{#1}
\SetKwInput{Input}{Input}
\SetKwInput{Output}{Output}
\SetKwInput{Inputs}{Inputs}
\SetKwInput{Outputs}{Outputs}
\SetKwComment{Comment}{// }{}
\SetKwIF{If}{ElseIf}{Else}{If}{then:}{Else If}{Else:}{}
\SetKwFor{ForEach}{For each}{do:}{}
\SetKw{Return}{Return}
\SetKwProg{Fn}{Function}{:}{}

\allowdisplaybreaks

\def\startdocumentlayoutoptions{
    \selectlanguage{british}
    \setlength{\parskip}{0.25\baselineskip}
    \setlength{\parindent}{2em}
    \kopfzeiledefault
    \documentfont
    \normalsize
}

\providecommand{\highlightTerm}{}
\renewcommand{\highlightTerm}[1]{\emph{#1}}

\renewenvironment{highlightForReview}[0]{\color{blue}}{}

\def\addresseshere{%
  \bgroup
  \setlength{\parindent}{0pt}
  \enddoc@text
  \egroup
  \let\enddoc@text\relax
}

\makeatother

\begin{document}
\startdocumentlayoutoptions

\documentpartnormal
\thispagestyle{plain}



\def\abstractname{Abstract}
\begin{abstract}
    We establish explicit means via which natural dilations of
    completely positive (CP) maps
    can be constructed
    \akin Kraus's \Second representation theorem.
    To obtain this,
    we rely on the Choi\==Jamio{\l}kowski correspondence
    and develop a Cholesky algorithm for bi-partite systems.
    This enables a canonical construction of adjoint actions
    which recover the behaviour of the original CP\=/maps.
    Our results hold under separability assumptions
    and the requirement that the maps
    are completely bounded
    and preserve the subideal of finite rank operators.
\end{abstract}



\title[The Choi\==Cholesky algorithm for completely positive maps]{%
    \hraum The Choi\==Cholesky algorithm\hraum%
    \relax
    \newline%
    \hraum for completely positive maps\hraum%
}

\author{Raj Dahya}
\address{Fakult\"at f\"ur Mathematik und Informatik\newline
Universit\"at Leipzig, Augustusplatz 10, D-04109 Leipzig, Germany}
\email{raj\,[\!\![dot]\!\!]\,dahya\:[\!\![at]\!\!]\:web\,[\!\![dot]\!\!]\,de}

\def\subjclassname{Mathematics Subject Classification (2020)}
\subjclass{81R15, 46L07, 47C15, 47A20}
\keywords{CPTP-maps; Choi matrices; Cholesky decomposition; dilations; measurability.}

\maketitle



\documentpartnormal
\setcounternach{section}{1}



\section[Introduction]{Introduction}
\label{sec:intro:sig:article-graph-raj-dahya}


\firstparagraph
Completely positive trace-preserving (CPTP) maps
provide a well established mathematical model
of transformations of physical states in quantum systems
(see \exempli
  \cite{%
      Davies1976BookQuantumOpenSys,%
      Kraus1983%
  }%
).
A map ${\Phi : L^{1}(H_{1}) \to L^{1}(H_{2})}$
between trace-class operators defined on Hilbert spaces $H_{1}$ and $H_{2}$
constitutes a CPTP\=/map
if the maps
${\id_{n} \otimes \Phi : L^{1}(\complex^{n} \otimes H_{1}) \to L^{1}(\complex^{n} \otimes H_{2})}$
are positive%
\footnote{%
  \idest
  $(\id_{n} \otimes \Phi)(s)$
  is positive semi-definite
  for all positive semi-definite trace-class operators
  $s \in \BoundedOps{\complex^{n} \otimes H_{1}}$.
}
for all $n\in\naturals$
and
$\tr(\Phi(s)) = \tr(s)$
for all $s \in L^{1}(H_{1})$.
The goal of this paper is to address the question:
\emph{%
  Can dilations of completely positive maps
  be explicitly constructed
  and can this be obtained in a \usesinglequotes{natural}, \idest canonical, fashion?
}



In
  \cite[Theorem~4.1]{Kraus1971Article},
  \cite[\S{}3 and \S{}5]{Kraus1983}
Kraus introduced two cornerstone representation theorems for CPTP\=/maps.
By the \First representation theorem,
$\Phi$ is a CPTP\=/map
if and only if it can be expressed as

\begin{restoremargins}
\begin{equation}
\label{eq:kraus:I:sig:article-graph-raj-dahya}
  \Phi(s)
    =
      \sum_{i}
        w_{i}\:s\:w_{i}^{\ast}
\end{equation}
\end{restoremargins}

\continueparagraph
for all $s \in L^{1}(H_{1})$
and some family $\{w_{i}\}_{i} \subseteq \BoundedOps{H_{1}}{H_{2}}$
satisfying $\sum_{i}w_{i}^{\ast}w_{i} = \onematrix$.%
\footnote{%
  A similar representation was established earlier by de~Pillis
  (see
    \cite[Theorem~2.1]{DePillis1967Article}%
  ).
  However, this representation
  is limited to the finite-dimensional setting,
  involves the use of the matrix transpose of $s$ on the right hand side,
  and characterises positive (instead of completely positive) maps.
}
By Kraus's \Second representation theorem,
which can be directly derived from this,
$\Phi$ is a CPTP\=/map
if and only if it can be expressed as

\begin{restoremargins}
\begin{equation}
\label{eq:kraus:II:sig:article-graph-raj-dahya}
  \Phi(s)
    =
      \tr_{2}(\ad{U}\:(s \otimes \omega))
\end{equation}
\end{restoremargins}

\continueparagraph
for all $s \in L^{1}(H_{1})$,
where
  $U \in \BoundedOps{H_{1} \otimes H}{H_{2} \otimes H}$
  is a unitary operator for some auxiliary Hilbert space $H$,
  $\omega \in L^{1}(H)$ is a state,%
  \footnote{%
    \idest
    $\omega \geq \zeromatrix$ and $\tr(\omega) = 1$.
  }
  $\ad{U}$ denotes the \highlightTerm{adjoint action}
  and
  $\tr_{2}$ the \highlightTerm{partial trace} (see below).

A particular advantage of Kraus's results
is that they are applicable to Hilbert spaces of arbitrary dimensions.%
\footnote{%
  Note that whilst Kraus assumed separability of the underlying Hilbert spaces
  in his proofs, these are an unnecessary requirement
  (see \exempli
    \cite[Theorem~9.2.3]{Davies1976BookQuantumOpenSys},
    \cite[Proposition~2.3.10, Remark~2.3.11, and Appendix~A.5.3]{vomEnde2020PhdThesis}%
  ).
}
However, under the hood, the \First representation theorem
relies on the Stinespring dilation theorem
  \cite[Theorem~4.8]{Pisier2001bookCBmaps},
  \cite[Theorem~9.2.1]{Davies1976BookQuantumOpenSys}
and Naimark's representation theorem
of normal \TextCStarAlg representations
  \cite[Theorem~3]{Naimark1972normedalg},
  \cite[Lemma~9.2.2]{Davies1976BookQuantumOpenSys},
which in turn relies on Zorn's lemma.
That is, Kraus's representations do not allow us
to compute the dilations of CPTP\=/maps
via constructible means.

During the same period,
Choi \cite{Choi1975Article}
developed constructive tools
via which Kraus's \First representation theorem
could be achieved.
This approach has the following limitations:
\textbf{1)}~Choi's proof of the \First representation theorem
is restricted to the finite-dimensional setting.
\textbf{2)}~The means via which the parameters
in the \First and thereby \Second representation theorems
are derived involve a certain level of arbitrary choice,
and therefore do not establish a \emph{canonical} construction
(\cf
  \cite[Remark~4]{Choi1975Article}%
).

To expand on \textbf{2)}, consider a CPTP\=/map
${\Phi : L^{1}(H_{1}) \to L^{1}(H_{2})}$,
where $H_{1}$ is finite-dimensional
with a fixed orthonormal basis (ONB) $\{\BaseVector{i}\}_{i=1}^{N}$
for some $N \in \naturals$.
In this setting,
a certain positive operator,%
\footnote{%
 for operators on Hilbert spaces
 \highlightTerm{positive} shall always mean \highlightTerm{positive semi-definite}.
}
$\Choi{\Phi} \in \BoundedOps{H_{1} \otimes H_{2}}$,
referred to as the \highlightTerm{Choi matrix},
can be associated to $\Phi$ in a bijective manner
(\cf
  \cite[Definition~4.1.1 and Theorem~4.1.8]{Stoermer2013BookPosOps},
  see also
  \S{}\ref{sec:choi:sig:article-graph-raj-dahya}%
).
Due to positivity, the Choi matrix admits a diagonalisation,
which in turn is used to construct the $w_{i}$
operators in \eqcref{eq:kraus:I:sig:article-graph-raj-dahya}
(see
  \cite[Theorem~1]{Choi1975Article},
  \cite[Theorem~4.1.8]{Stoermer2013BookPosOps}%
).
The issue in this approach is that there is in general no natural choice
for the diagonalisation,
in particular where eigenvalues are repeated.
Relying on selection theorems
(\cf \exempli
  \cite{%
      Michael1956ArticleI,%
      Michael1956ArticleII%
  },
  \cite[Chapter~18]{Kechris1995BookDST}%
),
one can produce Borel-measurable means via which the diagonalisation
and thereby the dilation can be constructed.
However, such selections cannot in general be explicitly described
nor are they unique.



The present paper attempts to bridge this gap,
by achieving a result \akin Kraus's \Second representation theorem,
which establishes a \emph{canonical} construction.
Given a fixed choice of an ONB for $H_{1}$ (\viz as an ordered sequence of basis vectors),
our dilations can be constructed in a Borel-measurable fashion
which can be explicitly described without any reliance on arbitrary choice.
The key ingredient is to replace \emph{diagonalisations}
by \emph{Cholesky decompositions}
of the positive Choi matrices associated to completely positive (CP) maps.
Our main challenge here is that Choi matrices are defined on \emph{bi-partite systems},
requiring the Cholesky algorithm to be reworked for this setting.
This is the goal of \S{}\ref{sec:bipartite:sig:article-graph-raj-dahya}.
In \S{}\ref{sec:results:sig:article-graph-raj-dahya}
we then derive our representation theorem and its properties.



\subsection[General notation]{General notation}
\label{sec:intro:notation:sig:article-graph-raj-dahya}

\firstparagraph
Throughout this paper we use the following notation:

\begin{itemize}
  \item
    $\naturals = \{1,2,\ldots\}$,
    $\naturalsZero = \{0,1,2,\ldots\}$,
    $\realsNonNeg = \{r\in\reals \mid r\geq 0\}$.

  \item
    For any Hilbert space,
    $\onematrix$ shall denote the \highlightTerm{identity operator}.
    For \TextCStarAlgs, $\id$ shall denote the \highlightTerm{identity map}.
    In ambivalent circumstances we use subscripts to denote the space
    on which an identity operator lives.

  \item
    For any Hilbert space $H$,
    $
      \FiniteRankOps{H},
      \BoundedOps{H}_{\text{s-a}}
      \subseteq
      \BoundedOps{H}
    $
    denote the subspaces
    of finite rank
    \resp
    self-adjoint
    operators.

  \item
    Letting $H_{1}$, $H_{2}$ be Hilbert spaces,
    a linear transformation ${\Phi : L^{1}(H_{1}) \to L^{1}(H_{2})}$
    which satisfies
    $
      \norm{\Phi}_{\text{cb}}
      \coloneqq
        \sup_{n\in\naturals}
          \norm{\id_{n} \otimes \Phi}
      < \infty
    $
    \resp
    $\norm{\Phi}_{\text{cb}} \leq 1$
    is called
    \highlightTerm{completely bounded} (CB)
    \resp
    \highlightTerm{completely contractive} (CC).

  \item
    Maps that are CP and trace-preserving (TP) \resp CC \resp CB
    are referred to as CPTP \resp CPCC \resp CPCB.
    Note that CPTP\=/maps are automatically CC
    (\cf
      \cite[Lemma~2.1]{Davies1976BookQuantumOpenSys},
      \cite[\S{}2,~(2.21)]{Kraus1983}%
    )
    and that CB-, CC-, CP-, and TP\=/maps
    are closed under composition.

  \item
    The \highlightTerm{Hermitian conjugate}
    of a bounded operator $u \in \BoundedOps{H_{1}}{H_{2}}$,
    shall be denoted as $u^{\ast} \in \BoundedOps{H_{2}}{H_{1}}$.
    If $u$ (equivalently: $u^{\ast}$) is invertible,
    then $u^{-\ast}$ denotes $(u^{-1})^{\ast} = (u^{\ast})^{-1}$.
    If it exists, the \highlightTerm{Moore-Penrose pseudo-inverse}
    (see Appendix \ref{app:spectral+mb:sig:article-graph-raj-dahya})
    is denoted by $u^{\dagger} \in \BoundedOps{H_{2}}{H_{1}}$.
    The \highlightTerm{adjoint action}
    of $u$ is the linear map defined by
      $\ad{u}\:a = u\:a\:u^{\ast}$
    for $a \in \BoundedOps{H_{1}}{H_{1}}$.

  \item
    It shall be convenient to adopt the \usesinglequotes{bra-ket}
    notation from mathematical physics
    for inner products on a Hilbert space $H$,
    \viz
      $\braket{\eta}{\xi}$
    for vectors $\xi, \eta \in H$,
    which is linear in the second argument
    and conjugate linear in the first.
    The expressions
      $\ket{\xi}$
      \resp
      $\bra{\eta}$
      \resp
      $\ketbra{\xi}{\eta}$,
    denote the operators defined by
      ${
        \complex \ni t
        \mapsto
        t\:\xi
        \in H
      }$
      \resp
      ${
        \HilbertRaum \ni x
        \mapsto
        \braket{\eta}{x} \in \complex
      }$
      \resp
      ${
        \HilbertRaum \ni x
        \mapsto
        \braket{\eta}{x} \: \xi
        \in \linspann\{\xi\}
        \subseteq H
      }$.
    This notation makes the following
    algebraic and geometric expressions possible:
    $\bra{\xi}^{\ast} = \ket{\xi}$,
    $\ket{\xi}^{\ast} = \bra{\xi}$,
    $(\ketbra{\xi}{\eta})^{\ast} = \ketbra{\eta}{\xi}$,
    $(\ket{\xi})\:(\bra{\eta}) = \ketbra{\eta}{\xi}$,
    and
    $
      (\bra{\eta})\:(\ket{\xi})
      = \braket{\eta}{\xi}
      = \tr(\ketbra{\xi}{\eta})
    $.

  \item
    If it is unclear from the context,
    we shall add subscripts to the bra-kets
    \idest
      $\bra{\xi}_{H}$,
      $\ket{\xi}_{H}$,
      $\braket{\eta}{\xi}_{H}$,
      $\ketbra{\xi}{\eta}_{H}$,
    to indicate in which Hilbert space
    the constructions occur.

  \item
    Given a Hilbert space $H$,
    we may identify the \highlightTerm{dual space}
    $H^{\ast}$
    with
    $\{x^{\ast} \coloneqq \bra{x} \mid x \in H\}$
    endowed with the inner product structure
    $
      \braket{y^{\ast}}{x^{\ast}}
      \coloneqq \braket{x}{y}
    $.
    In particular,
    ${x \mapsto x^{\ast}}$
    defines a conjugate-linear isomorphism
    between $H$ and $H^{\ast}$.

  \item
    Given Hilbert spaces $H_{1},H_{2},\ldots,H_{n}$
    and any $l \in \{1,2,\ldots,n-1\}$
    the \highlightTerm{partial trace}
    ${
      \tr_{l+1,\ldots,n}
      : L^{1}(H_{1}\otimes H_{2} \otimes \ldots \otimes H_{n})
      \to
      L^{1}(H_{1}\otimes H_{2} \otimes \ldots \otimes H_{l})
    }$
    is the linear operation
    which uniquely satisfies

      \begin{shorteqnarray}
        \tr(\tr_{l+1,\ldots,n}(s)a)
          =
            \tr(
              s
              \: (
                a
                \otimes
                \onematrix_{H_{l+1} \otimes H_{l+2} \otimes \ldots \otimes H_{n}}
              )
            )
      \end{shorteqnarray}

    \continueparagraph
    for all $s \in L^{1}(H_{1}\otimes H_{2} \otimes \ldots \otimes H_{n})$,
    $a \in \BoundedOps{H_{1}\otimes H_{2} \otimes \ldots \otimes H_{l}}$
    (\cf
      \cite[\S{}10.2~(1.5)]{Davies1976BookQuantumOpenSys}%
    ).

    As a simple example one can readily verify
    that

      \begin{restoremargins}
      \begin{equation}
      \label{eq:partial-trace:tensorproduct:sig:article-graph-raj-dahya}
        \tr_{l+1,\ldots,n}(
          \bigotimes_{k=1}^{n}
            s_{k}
        )
          =
            \prod_{k=l+1}^{n}
              \tr(s_{k})
            \:
            \bigotimes_{i=1}^{l}
              s_{i}
      \end{equation}
      \end{restoremargins}

    \continueparagraph
    for all
    $s_{1} \in L^{1}(H_{1})$,
    $s_{2} \in L^{1}(H_{2})$,
    \ldots
    $s_{n} \in L^{1}(H_{n})$.
\end{itemize}



\subsection[Statement of results]{Statement of results}
\label{sec:intro:results:sig:article-graph-raj-dahya}

\firstparagraph
To formulate our representation theorem,
we make use of the following terminology.
Letting $H_{1}$ and $H_{2}$ be arbitrary Hilbert spaces,
we shall call a map
${\Phi : L^{1}(H_{1}) \to L^{1}(H_{2})}$
\highlightTerm{finite rank preserving (FP)}
if $\Phi(\FiniteRankOps{H_{1}}) \subseteq \FiniteRankOps{H_{2}}$.
This can be motivated by physical systems
whose evolving state is described by finitely many eigenvectors.
Clearly, FP\=/maps
include all adjoint maps
(in particular the identity transformation)
and are closed under composition.
If $H_{2}$ is finite-dimensional,
then all linear transformations
${\Phi : L^{1}(H_{1}) \to L^{1}(H_{2})}$
are trivially FP\=/maps.

We shall further make use of a certain CPTP\=/map,
via which our representations shall be factored:
By basic understanding of Hilbert\==Schmidt spaces,
$
  \HilbertRaum \coloneqq L^{2}(H_{1} \otimes H_{2}, H_{2})
$
is isomorphic to
$
  H_{21^{\ast}2^{\ast}} \coloneqq H_{2} \otimes H_{1}^{\ast} \otimes H_{2}^{\ast}
$
via a unitary transformation
${\theta : \HilbertRaum \to H_{21^{\ast}2^{\ast}}}$
which uniquely satisfies
$
  \theta^{\ast}(z \otimes \bra{x} \otimes \bra{y})
  = \bra{x} \otimes \ketbra{z}{y}
$
for $x \in H_{1}$, $y, z \in H_{2}$
(\cf \exempli
  \cite[Propositions~2.6.9]{KadisonRingrose1983volI},
  \cite[Propositions~3.4.14--15]{Pedersen1989analysisBook}%
).
It is also well understood that the partial trace
${\tr_{2,3} : L^{1}(H_{21^{\ast}2^{\ast}}) \to L^{1}(H_{2})}$
constitutes a CPTP\=/map
(\cf \exempli
  \cite[\S{}5,~(5.3)]{Kraus1971Article}%
).
It follows that the composition

  \begin{restoremargins}
  \begin{equation}
  \label{eq:defn:universal-cptp:sig:article-graph-raj-dahya}
    \Psi_{H_{1},H_{2}}
    \coloneqq
    \tr_{2,3} \circ \ad{\theta}
    : L^{1}(\HilbertRaum) \to L^{1}(H_{2})
  \end{equation}
  \end{restoremargins}

\continueparagraph
constitutes a CPTP\=/map.

Consider now elements of the form
$
  w_{1} \coloneqq \ket{x_{1}} \otimes \ketbra{z_{1}}{y_{1}},
  w_{2} \coloneqq \ket{x_{2}} \otimes \ketbra{z_{2}}{y_{2}}
  \in L^{1}(\HilbertRaum)
$,
where
$x_{1},x_{2} \in H_{1}$,
$y_{1},z_{1},y_{2},z_{2} \in H_{2}$.
Observe that
$
  w_{1} = \theta^{\ast}(z_{1} \otimes x_{1}^{\prime} \otimes y_{1}^{\prime})
$
and
$w_{2} = \theta^{\ast}(z_{2} \otimes x_{2}^{\prime} \otimes y_{2}^{\prime})$
where
$x_{1}^{\prime} = \bra{x_{1}}$,
$x_{2}^{\prime} = \bra{x_{2}}$,
$y_{1}^{\prime} = \bra{y_{1}}$,
$y_{2}^{\prime} = \bra{y_{2}}$.
Thus

  \begin{longeqnarray}
    \Psi_{H_{1},H_{2}}(\ketbra{w_{1}}{w_{2}})
      &=
        &\Psi_{H_{1},H_{2}}(
          \theta^{\ast}
          \:\ketbra{
            (z_{1} \otimes x_{1}^{\prime} \otimes y_{1}^{\prime})
          }{
            (z_{2} \otimes x_{2}^{\prime} \otimes y_{2}^{\prime})
          }_{\HilbertRaum}
          \:\theta
        )
        \\
      &=
        &(
          \tr_{2,3}
          \circ
          \ad{\theta}
          \circ
          \ad{\theta^{\ast}}
        )
        (
          \ketbra{z_{1}}{z_{2}}_{H_{2}}
          \otimes
          \ketbra{x_{1}^{\prime}}{x_{2}^{\prime}}_{H_{1}^{\prime}}
          \otimes
          \ketbra{y_{1}^{\prime}}{y_{2}^{\prime}}_{H_{1}^{\prime}}
        )
        \\
      &=
        &(
          \tr_{2,3}
        )
        (
          \ketbra{z_{1}}{z_{2}}_{H_{2}}
          \otimes
          \ketbra{x_{1}^{\prime}}{x_{2}^{\prime}}_{H_{1}^{\prime}}
          \otimes
          \ketbra{y_{1}^{\prime}}{y_{2}^{\prime}}_{H_{1}^{\prime}}
        )
        \\
      &\eqcrefoverset{eq:partial-trace:tensorproduct:sig:article-graph-raj-dahya}{=}
        &\tr(\ketbra{x_{1}^{\prime}}{x_{2}^{\prime}}_{H_{1}^{\prime}})
        \:\tr(\ketbra{y_{1}^{\prime}}{y_{2}^{\prime}}_{H_{1}^{\prime}})
        \:\ketbra{z_{1}}{z_{2}}_{H_{2}}
        \\
      &=
        &\braket{x_{2}^{\prime}}{x_{1}^{\prime}}_{H_{1}^{\prime}}
        \:\braket{y_{2}^{\prime}}{y_{1}^{\prime}}_{H_{2}^{\prime}}
        \:\ketbra{z_{1}}{z_{2}}_{H_{2}}
        \\
      &=
        &\braket{x_{1}}{x_{2}}_{H_{1}}
        \:\braket{y_{1}}{y_{2}}_{H_{2}}
        \:\ketbra{z_{1}}{z_{2}}_{H_{2}}
        \\
      &=
        &(
          \bra{x_{1}}
          \otimes
          \ketbra{z_{1}}{y_{1}}
        )
        \:(
          \bra{x_{2}}
          \otimes
          \ketbra{y_{2}}{z_{2}}
        )^{\ast}
      =
        w_{1}\:w_{2}^{\ast}.
  \end{longeqnarray}

Letting $\mathcal{W}$ be the linear span of elements of the form
$\ketbra{w_{1}}{w_{2}}$
as chosen above,
it follows by the
$L^{1}$-density of $\mathcal{W}$
in $L^{1}(\HilbertRaum)$
(\cf
  \cite[Theorem~2.4.17]{Murphy1990}%
)
and the continuity of CPTP\=/maps in the $L^{1}$-norm,
that

  \begin{restoremargins}
  \begin{equation}
  \label{eq:property:universal-cptp:sig:article-graph-raj-dahya}
    \Psi_{H_{1},H_{2}}(\ketbra{w_{1}}{w_{2}})
    = w_{1}\:w_{2}^{\ast}
  \end{equation}
  \end{restoremargins}

\continueparagraph
for all $w_{1},w_{2} \in \HilbertRaum = L^{2}(H_{1} \otimes H_{2}, H_{2})$.

We can now formulate our first main result:

\begin{highlightbox}
\begin{thm}[Representation of CPCB FP\=/maps]
\makelabel{thm:choi-cholesky-rep:sig:article-graph-raj-dahya}
  Let $H_{1}$ be a separable Hilbert space
  and $H_{2}$ an arbitrary Hilbert space.
  Further let $\{\BaseVector{i}\}_{i \in I}$ be an orthonormal basis (ONB) for $H_{1}$,
  whereby $I = \naturals$ or $\{1,2,\ldots,N\}$ for some $N\in\naturals$
  and set $\HilbertRaum \coloneqq L^{2}(H_{1} \otimes H_{2}, H_{2})$.
  Consider an FP\=/map
  ${
    \Phi:L^{1}(H_{1}) \to L^{1}(H_{2})
  }$.
  Then $\Phi$ is a CPCB- (\resp CPCC- \resp CPTP-) map
  if and only if

  \begin{restoremargins}
  \begin{equation}
  \label{eq:choi-cholesky-representation:sig:article-graph-raj-dahya}
    \Phi = \Psi_{H_{1},H_{2}} \circ \ad{V_{\Phi}}
  \end{equation}
  \end{restoremargins}

  \continueparagraph
  for a bounded operator (\resp contraction \resp an isometry)
  ${V_{\Phi} : H_{1} \to \HilbertRaum}$.
\end{thm}
\end{highlightbox}



Our second main result
exploits the manner in which the representations in
\Cref{thm:choi-cholesky-rep:sig:article-graph-raj-dahya}
are constructed
to address the main question posed at the start.
Consider the class
$\mathbb{O}$
of operations on operators on (finite-dimensional) Hilbert spaces
obtained via compositions of:
  constants (the identity, elementary operators, \etcetera),
  addition,
  scalar multiplication,
  operator multiplication,
  Hermitian conjugation,
  tensor products,
  square roots of positive operators,
  and pseudo-inverses
  of finite rank positive operators.%
\footnote{%
  \cf Appendix \ref{app:spectral+mb:sig:article-graph-raj-dahya} for definitions.
}
Note that under the norm topology,
each of these operations, and thus also their compositions,
are Borel-measurable,
provided the underlying Hilbert spaces are separable.%
\footnote{%
  For the reader's convenience,
  proofs of the measurability
  of the latter two operations
  have been provided in the appendix,
  see
  \Cref{e.g.:norm-measurability-of-sqrt:sig:article-graph-raj-dahya}
  and
  \Cref{prop:existence-of-pinv:sig:article-graph-raj-dahya}.
  The remaining operations are clearly continuous
  and do not require separability.
}
We note further that in the finite-dimensional setting,
the operations in the class $\mathbb{O}$
can all be practically implemented via modern programming languages.

\begin{highlightbox}
\begin{thm}[Properties of dilations]
\makelabel{thm:choi-cholesky-rep-computability:sig:article-graph-raj-dahya}
  Working with the setup of \Cref{thm:choi-cholesky-rep:sig:article-graph-raj-dahya},
  let
  $
    \mathcal{X}
    \subseteq
    \BoundedOps{L^{1}(H_{1})}{L^{1}(H_{2})}
  $
  be the space of CPCB FP\=/maps
  endowed with the strong operator topology ($\topSOT$),%
  \footnoteref{ft:1:\beweislabel}
  and
  $
    \mathcal{Y}
    \coloneqq \BoundedOps{\HilbertRaum}
  $
  be the space of bounded operators
  endowed with the weak operator topology ($\topWOT$).
  A map
    ${\mathfrak{C} : \mathcal{X} \to \mathcal{Y}}$
  can be chosen such that the following hold:

  \begin{enumerate}[
      label={\bfseries{(\alph*)}},
      ref={\alph*},
      left=\rtab,
  ]
    \item\punktlabel{1}
      For each CPCB (\resp CPCC \resp CPTP) FP\=/map $\Phi$,
      $V_{\Phi} \coloneqq \mathfrak{C}(\Phi)$
      is a bounded operator
      (\resp a contration \resp an isometry)
      satisfying \eqcref{eq:choi-cholesky-representation:sig:article-graph-raj-dahya}.

    \item\punktlabel{2}
      For each $\Phi \in \mathcal{X}$ and $n \in I$,
      the element $\mathfrak{C}(\Phi)\BaseVector{n}$
      viewed as an operator can be constructed
      from
      $\{\Phi(\ElementaryMatrix{i}{j})\}_{i,j=1}^{n}$
      via explicitly definable operations from $\mathbb{O}$.

    \item\punktlabel{3}
      If $H_{2}$ is separable,
      then the restriction of $\mathfrak{C}$
      to the subspace of CPCC\=/maps
      is Borel-measurable.
  \end{enumerate}

  \nvraum{1}
\end{thm}
\end{highlightbox}

\footnotetext[ft:1:\beweislabel]{%
  \idest the topology given by
  ${\Phi^{(\alpha)} \underset{\alpha}{\overset{\tinytopSOT}{\longrightarrow}} \Phi}$
  if and only if
  ${\norm{\Phi^{(\alpha)}(s) - \Phi(s)}_{1} \underset{\alpha}{\longrightarrow} 0}$
  for all $s \in L^{1}(H_{1})$.
}



\begin{rem}[Unitarity and relation to Kraus representations]
  Restricting to the CPTP case,
  the above dilation can be easily modified in the usual manner
  to obtain a unitary action,
  \exempli by replacing
  $\HilbertRaum$ by $\HilbertRaum' \coloneqq L^{2}(H_{1} \otimes H_{2}, H_{2}) \otimes \complex^{2}$,
  $\mathfrak{C}(\Phi) = V_{\Phi}$
  by

    \begin{shorteqnarray}
      \mathfrak{C}'(\Phi)
      \coloneqq
        U_{\Phi}
      \coloneqq
      \begin{matrix}{cc}
        V_{\Phi}
          &(\onematrix - V_{\Phi}\:V_{\Phi}^{\ast})
          \\
        \zeromatrix
          &V_{\Phi}^{\ast}
      \end{matrix},
    \end{shorteqnarray}

  \continueparagraph
  (see
    \cite{Halmos1950dilation},
    \cf also
    \cite[\S{}1.3]{Choi2004LectureNotes},
    \cite[\S{}1]{Shalit2021DilationBook}%
  )
  and
  $\Psi_{H_{1},H_{2}}$
  by
  $\Psi' \coloneqq \tr_{2,3,4} \circ \ad{\theta \otimes \onematrix}$.
  Under these modifications one can readily demonstrate that

    \begin{shorteqnarray}
      \Phi(s)
      =
        \Psi'(\ad{U_{\Phi}}(s \otimes \ketbra{\BaseVector{1}}{\BaseVector{1}}))
      \eqcrefoverset{eq:defn:universal-cptp:sig:article-graph-raj-dahya}{=}
        \tr_{2,3,4}(
          \ad{
            (\theta\otimes\onematrix)
            \:U_{\Phi}
          }
          (s \otimes \ketbra{\BaseVector{1}}{\BaseVector{1}})
        )
    \end{shorteqnarray}

  \continueparagraph
  for $s \in L^{1}(H_{1})$.
  \Cref{thm:choi-cholesky-rep:sig:article-graph-raj-dahya}
  thus provides an alternative root to establishing
  Kraus's \Second representation theorem \eqcref{eq:kraus:II:sig:article-graph-raj-dahya}.
  It is also a straightforward exercise to verify that the claims
  in \Cref{thm:choi-cholesky-rep-computability:sig:article-graph-raj-dahya}
  continue to hold with $\mathfrak{C}$
  replaced by $\mathfrak{C}'$ restricted to CPTP FP\=/maps.
\end{rem}



\begin{rem}[Related works]
  At the time of circulation,
  recent work \cite{BelovDubovIvanov2026Misc}
  in \highlightTerm{quantum channel tomography}%
  \footnote{%
    an important area of research
    with applications to quantum computing,
    concerned with estimations and approximate inversions
    of CPTP\=/maps
    \cite{%
        BenderskyPaz2013Article,%
        BaldwinKalevDeutsch2014Article%
    }.
  }
  was brought to our attention.
  In contrast to our aims
  (the establishment of explicit algorithms to derive canonical representations of CPTP\=/maps),
  the work in \cite{BelovDubovIvanov2026Misc}
  focusses on formulating channel reconstruction as a semi-definite programming problem.
  To this end, the authors similarly apply Cholesky decompositions to Choi matrices
  (\cf expression 16 and Appendix A).
  Their decomposition consists of multiple stages including QR-decomposition,
  and is applied to matrices with flattened dimensions.
  Our constructions by contrast consist of a single stage,
  preserve the bi-partite structure,
  and extend coherently as the dimension of $H_{1}$ is increased
  (\cf \Cref{rem:coherence:sig:article-graph-raj-dahya} below).
\end{rem}




\section[Bi-partite systems]{Bi-partite systems}
\label{sec:bipartite:sig:article-graph-raj-dahya}

\firstparagraph
The main instrument we shall use to study CP\=/maps
are the so-called \highlightTerm{Choi matrices} (see below),
which are operators defined on tensor products.
This section is thus dedicated to providing groundwork to work with such operators.

Throughout we shall consider Hilbert spaces $H_{1}$ and $H_{2}$
as well as an ONB $\{\BaseVector{i}\}_{i \in I}$ for $H_{1}$
indexed by a linearly ordered set $(I, \leq)$,
\exempli
  $I = \naturals$
  or
  $\{1,2,\ldots,N\}$
  for some $N \in \naturals$.
We shall also make use of the following definitions:

\begin{itemize}
  \item
    For each $i,j \in I$ define
    $
      \ElementaryMatrix{i}{j}
      \coloneqq
      \ketbra{\BaseVector{i}}{\BaseVector{j}}
      \in
      L^{1}(H_{1})
      \subseteq
      \BoundedOps{H_{1}}
    $.

  \item
    Given a bounded operator $C \in \BoundedOps{H_{1} \otimes H_{2}}$
    we define

      \begin{shorteqnarray}
        C_{i,j}
        \coloneqq
          (\bra{\BaseVector{i}} \otimes \onematrix)
          \:C
          \:(\ket{\BaseVector{j}} \otimes \onematrix)
        =
          (\ket{\BaseVector{i}} \otimes \onematrix)^{\ast}
          \:C
          \:(\ket{\BaseVector{j}} \otimes \onematrix)
        \in
          \BoundedOps{H_{2}}
      \end{shorteqnarray}

    \continueparagraph
    for each $i,j \in I$.

  \item
    We shall say that an operator
      $C \in \BoundedOps{H_{1} \otimes H_{2}}$
    has
    \highlightTerm{finite support}
    if
      $C_{i,j} = \zeromatrix$
    for all $(i,j) \in (I \times I) \without (F \times F)$
    and some finite $F \subseteq I$,
    and we define $\supp C$ to be the smallest set $F \subseteq I$
    for which this condition holds.
    Letting $F \subseteq I$ be finite,
    it is easy to verify that
      $\supp C \subseteq F$
    if and only if
    $C$ can be expressed as

      \begin{restoremargins}
      \begin{equation}
      \label{eq:bipartite:sig:article-graph-raj-dahya}
        C
        =
        \sum_{i,j \in F}
          \ElementaryMatrix{i}{j} \otimes C_{i,j}.
      \end{equation}
      \end{restoremargins}

    \continueparagraph
    For this reason we shall refer to each
    $
      C_{i,j}
      \coloneqq
      (\bra{\BaseVector{i}} \otimes \onematrix)
      \:C
      \:(\ket{\BaseVector{j}} \otimes \onematrix)
    $
    as the $(i,j)$-th (operator-valued) \highlightTerm{entry} of $C$.

  \item
    It is a simple exercise to verify
    that algebraic operations on operators with finite
    support correspond to matrix operations.
    That is,
    letting $C,C^{\prime} \in \BoundedOps{H_{1} \otimes H_{2}}$
    with $\supp C \subseteq F$
    and $\supp C^{\prime} \subseteq F$
    for some finite $F \subseteq I$,
    one has

      \begin{restoremargins}
      \begin{equation}
      \label{eq:bipartite:operations:sig:article-graph-raj-dahya}
      \everymath={\displaystyle}
      \begin{array}[m]{rcl}
        C^{\ast}
          &=
            &\sum_{i,j \in F}
              \ElementaryMatrix{i}{j}
              \otimes
              C_{j,i}^{\ast},
            \\
        C + C^{\prime}
          &=
            &\sum_{i,j \in F}
              \ElementaryMatrix{i}{j}
              \otimes
              (C_{i,j} + C^{\prime}_{j,i}),
              \quad\text{and}
            \\
        C \: C^{\prime}
          &=
            &\sum_{i,j \in F}
              \ElementaryMatrix{i}{j}
              \otimes
              \sum_{k \in F}
                C_{i,k}\:C^{\prime}_{k,j}.
      \end{array}
      \end{equation}
      \end{restoremargins}

    In particular
      $\supp(C^{\ast}) \subseteq F$,
      $\supp(C + C^{\prime}) \subseteq F$,
      and
      $\supp(C \: C^{\prime}) \subseteq F$.
\end{itemize}


\subsection[Choi--Jamio{\l}kowski correspondence]{Choi\==Jamio{\l}kowski correspondence}
\label{sec:choi:sig:article-graph-raj-dahya}


\firstparagraph
We now present standard tools to analyse CP(TP)\=/maps.
For a linear transformation
${\Phi : L^{1}(H_{1}) \to L^{1}(H_{2})}$
and finite $F \subseteq I$,
letting

  \begin{restoremargins}
  \begin{equation}
  \label{eq:choi-definitions:sig:article-graph-raj-dahya}
  \everymath={\displaystyle}
  \begin{array}[m]{rcl}
    \mathcal{O}^{(F)}
      &\coloneqq
        &\sum_{i \in F}
          \BaseVector{i} \otimes \BaseVector{i}
        \in L^{1}(H_{1} \otimes H_{1}),
      \\
    \mathcal{E}^{(F)}
      &\coloneqq
        &\ketbra{\mathcal{O}^{(F)}}{\mathcal{O}^{(F)}}
      =
        \sum_{i, j \in F}
          \ElementaryMatrix{i}{j}
          \otimes
          \ElementaryMatrix{i}{j}
      \in L^{1}(H_{1} \otimes H_{1}),
      \quad\text{and}
      \\
    \Choi{\Phi}^{(F)}
      &\coloneqq
          &(\id \otimes \Phi)(\mathcal{E}^{(F)})
      =
          \sum_{i, j \in F}
              \ElementaryMatrix{i}{j}
              \otimes
              \Phi(\ElementaryMatrix{i}{j})
      \in L^{1}(H_{1} \otimes H_{2}),
  \end{array}
  \end{equation}
  \end{restoremargins}

\continueparagraph
we refer to $\{\Choi{\Phi}^{(F)}\}_{F}$ as the \highlightTerm{Choi matrices}
associated to $\Phi$.

If $H_{1}$ is finite-dimensional, it is well-known that
${\Phi \mapsto \Choi{\Phi} \coloneqq \Choi{\Phi}^{(I)}}$,
referred to as the \highlightTerm{Choi\==Jamio{\l}kowski isomorphism}
(\cf
  \cite[\S{}4.1]{Stoermer2013BookPosOps}%
),
constitutes a bijection between
$\BoundedOps{\BoundedOps{H_{1}}}{\BoundedOps{H_{2}}}$
and
$\BoundedOps{H_{1} \otimes H_{2}}$.
In particular, one can verify that

\begin{restoremargins}
\begin{equation}
\label{eq:inverse-choi:sig:article-graph-raj-dahya}
  \Phi(s)
    =
      (\ket{\mathcal{O}} \otimes \onematrix)^{\ast}
      \:(s \otimes \Choi{\Phi})
      \:(\ket{\mathcal{O}} \otimes \onematrix)
\end{equation}
\end{restoremargins}

\continueparagraph
for all $s \in \BoundedOps{H_{1}}$,
where $\mathcal{O} \coloneqq \mathcal{O}^{(I)}$.
Choi matrices allow for natural characterisations of properties of linear operations,
a selection of which are summarised
in \Cref{table:correspondence:choi:sig:article-graph-raj-dahya}.
The first correspondence here is due to
de~Pilles
  \cite[Proposition~1.2]{DePillis1967Article},
the second owes to
Choi
  \cite[Theorem~2]{Choi1975Article},
and the final to Jamio{\l}kowski
  \cite[Theorems~1--2]{Jamiolkowski1972Article}.%
\footnote{%
  \cf also
  \cite[Theorem~4.1.8]{Stoermer2013BookPosOps},
  \cite[\S{}I and \S{}IV]{JiangLuoFu2013ArticleChoiMatrices},
  \cite[Propositions~5--6]{HomaOrtegaKoniorczyk2024Article}.
}

\begin{table}[!htb]
  \begin{tabular}[t]{|p{0.25\textwidth}|p{0.25\textwidth}|}
    \hline
      Property of $\Phi$
      &Property of $\Choi{\Phi}$
      \\
    \hline
    \hline
      Hermitian%
      \footnoteref{ft:1:table:correspondence:choi:sig:article-graph-raj-dahya}
      &self-adjoint
      \\
      completely positive
      &positive
      \\
      trace-preserving
      &$\tr_{2}(\Choi{\Phi}) = \onematrix$
      \\
    \hline
  \end{tabular}
  \caption{%
    Correspondence of properties between linear operations
    ${\Phi : \protect\BoundedOps{H_{1}} \to \protect\BoundedOps{H_{2}}}$
    and their Choi matrices
    $\Choi{\Phi} \in \protect\BoundedOps{H_{1} \otimes H_{2}}$
    under the assumption that $H_{1}$ is finite-dimensional.
  }
  \label{table:correspondence:choi:sig:article-graph-raj-dahya}
\end{table}

\footnotetext[ft:1:table:correspondence:choi:sig:article-graph-raj-dahya]{%
  \idest $\Phi(a) \in \BoundedOps{H_{2}}$
  is self-adjoint
  for all self-adjoint elements $a \in \BoundedOps{H_{1}}$.
}



Relying on these dualities we can immediately derive some useful properties about CP(TP)\=/maps
defined for Hilbert spaces of arbitrary dimensions.

\begin{highlightbox}
\begin{lemm}[Choi\==Jamio{\l}kowski correspondence]
\makelabel{lemm:choi-jamiolkowski:sig:article-graph-raj-dahya}
  Let ${\Phi : L^{1}(H_{1}) \to L^{1}(H_{2})}$
  be a bounded linear transformation
  and define the family
  $
    \{C_{i,j} \coloneqq \Phi(\ElementaryMatrix{i}{j})\}_{i,j \in I}
    \subseteq
    L^{1}(H_{2})
  $
  of trace-class operators.
  Then $\Phi$ constitutes a CP- (\resp CPTP-) map
  if and only if
  \punktcref{1} holds
  (\resp \punktcref{1} and \punktcref{2} hold),
  where:

  \begin{enumerate}[
    label={\bfseries{(\alph*)}},
    ref={\alph*},
    left=\rtab,
  ]
    \item\punktlabel{1}
      The trace-class operators
        $
          C^{(F)}
          \coloneqq
            \Choi{\Phi}^{(F)}
          =
            \sum_{i,j \in F}
              \ElementaryMatrix{i}{j}
              \otimes
              C_{i,j}
          \in L^{1}(H_{1} \otimes H_{2})
        $
      are positive for all finite $F \subseteq I$.

    \item\punktlabel{2}
      $\tr(C_{i,j}) = \delta_{i,j}$
      for all $i,j \in I$.
  \end{enumerate}

  \nvraum{1}

\end{lemm}
\end{highlightbox}

  \begin{proof}
    Towards the \usesinglequotes{only if}-direction:
    To show \punktcref{1},
    let $F \subseteq I$ be finite
    and consider the finite-dimensional subspace
    $H^{(F)}_{1} \coloneqq \linspann\{\BaseVector{i} \mid i \in F\}$
    of $H_{1}$.
    One has that
      $
        \Choi{\Phi}^{(F)}
        = (\id_{H^{(F)}_{1}} \otimes \Phi)(\mathcal{E}^{(F)})
      $,
    whereby $\mathcal{E}^{(F)} \in L^{1}(H^{(F)}_{1} \otimes H_{2})$
    is a positive element.
    By definition of $\Phi$ being completely positive,
    it follows that $\Choi{\Phi}^{(F)} \geq \zeromatrix$.
    In the CPTP case, \punktcref{2}
    follows by trace-preservation,
    since
    $
      \tr_{2}(C_{i,j})
      = \tr(\Phi(\ElementaryMatrix{i}{j}))
      = \tr(\ElementaryMatrix{i}{j})
      = \delta_{ij}
    $
    for $i,j \in I$.

    Towards the \usesinglequotes{if}-direction,
    first note that by boundedness and complete positivity,
    $\Phi$ is necessarily \emph{completely} bounded
    \wrt the $L^{1}$-norm
    (\cf
      \cite[\S{}2,~(2.21)]{Kraus1971Article},
      \cite[Proposition~3.6]{Paulsen2002book}%
    ).
    Suppose now that \punktcref{1} holds.
    To prove complete positivity,
    let $n \in \naturals$
    and $s \in L^{1}(\complex^{n} \otimes H_{1})$
    be arbitrary
    with $s \geq \zeromatrix$.
    Set $u \coloneqq \sqrt{s} \in L^{2}(\complex^{n} \otimes H_{1})$.
    Note that
    $
      \mathcal{W}_{n}
      = \linspann\{
        x \otimes \ElementaryMatrix{i}{j}
        \mid
        x \in \Matr_{n \times n}(\complex),
        \:i,j \in I
      \}
      \subseteq \FiniteRankOps{\complex^{n} \otimes H_{1}}
    $
    is an $L^{2}$-dense subspace of
    $L^{2}(\complex^{n} \otimes H_{1})$.%
    \footnoteref{ft:L1-density:\beweislabel}
    \footnotetext[ft:L1-density:\beweislabel]{%
      this can be proved
      by relying on
      the H{\"o}lder\==von~Neumann inequality,
      \cf also
      \cite[Theorem~2.4.17]{Murphy1990}%
    }
    There thus exist elements
    $\{u_{k}\}_{k \in \naturals} \subseteq \mathcal{W}_{n}$
    for which
    ${\norm{u_{k} - u}_{2} \underset{k}{\longrightarrow} 0}$.
    For each $k \in \naturals$ we have
    $
      u_{k}^{\ast}\:u_{k}
      \in \mathcal{W}_{n}
    $,
    since $\mathcal{W}_{n}$
    is clearly a subalgebra of $\BoundedOps{\complex^{n} \otimes H_{1}}$.
    Moreover, by the H{\"o}lder\==von~Neumann inequality
    (\cf
      \cite[Exercise~3.4.3]{Pedersen1989analysisBook}%
    )

    \begin{shorteqnarray}
      \norm{
        u_{k}^{\ast}\:u_{k}
        -
        s
      }_{1}
        &=
          &\norm{
              u_{k}^{\ast}\:u_{k}
              -
              u^{\ast}\:u
            }_{1}
          \\
        &\leq
          &\norm{(u_{k} - u)^{\ast}\:(u_{k} - u)}_{1}
            + \norm{(u_{k} - u)^{\ast}\:u}_{1}
            + \norm{u^{\ast}\:(u_{k} - u)}_{1}
          \\
        &\leq
          &\norm{(u_{k} - u)^{\ast}}_{2}\:\norm{u_{k} - u}_{2}
            + \norm{(u_{k} - u)^{\ast}}_{2}\:\norm{u}_{2}
            + \norm{u^{\ast}}_{2}\:\norm{u_{k} - u}_{2}
          \\
        &=
          &\norm{u_{k} - u}_{2}^{2}
            + 2\norm{u_{k} - u}_{2}\:\norm{u}_{2},
    \end{shorteqnarray}

    \continueparagraph
    which converges to $0$ as ${k \longrightarrow \infty}$.
    By the $L^{1}$-boundedness of $\id_{n} \otimes \Phi$
    (see above)
    and since positive trace-class operators are closed under the $L^{1}$-norm,
    it thus suffices to prove that
    $(\id_{n} \otimes \Phi)(s) \geq \zeromatrix$
    for $s \in L^{1}(\complex^{n} \otimes H_{1})$
    of the form $s = u^{\ast}\:u$ for some $u \in \mathcal{W}_{n}$.
    Letting $u \in \mathcal{W}_{n}$ be arbitrary,
    we can find a finite subset $F \subseteq I$
    as well as elements $\{x_{i,j}\}_{i,j \in F} \subseteq \Matr_{n \times n}(\complex)$
    such that
    $
      u
      = \sum_{i,j \in F}
          x_{i,j} \otimes \ElementaryMatrix{i}{j}
    $.
    Considering the positive element
    $
      s \coloneqq u^{\ast}\:u
      = \sum_{i,j,k \in F}
            x_{k,i}^{\ast}\:x_{k,j}
          \otimes
          \ElementaryMatrix{i}{j}
      \in L^{1}(\complex^{n} \otimes H_{1})
    $,
    one obtains

      \begin{shorteqnarray}
        (\id_{n} \otimes \Phi)(s)
        &=
          &\sum_{i,j,k \in F}
            x_{k,i}^{\ast}\:x_{k,j}
            \otimes
            C_{i,j}
          \\
        &=
          &\sum_{i,j,k \in F}
            x_{k,i}^{\ast}\:x_{k,j}
            \otimes
            (\bra{i} \otimes \onematrix_{H_{2}})
            \:C^{(F)}
            \:(\ket{j} \otimes \onematrix_{H_{2}})
          \\
        &=
          &\sum_{k \in F}
            \Big(
              \underbrace{
                \sum_{i \in F}
                  x_{k,i} \otimes \ket{i} \otimes \onematrix_{H_{2}}
              }_{X_{k} \coloneqq}
            \Big)^{\ast}
            \:(\onematrix_{\complex^{n}} \otimes C^{(F)})
            \:\Big(
              \underbrace{
                \sum_{j \in F}
                  x_{k,j} \otimes \ket{j} \otimes \onematrix_{H_{2}}
              }_{=X_{k}}
            \Big),
      \end{shorteqnarray}

    \continueparagraph
    which is positive since by \punktcref{1} $C^{(F)}$ is positive.
    This completes the proof of the complete positivity of $\Phi$.
    If furthermore \punktcref{2} holds,
    then one can readily verify that $\tr(\Phi(s)) = \tr(s)$
    for $
      s \in \mathcal{W}
      \coloneqq
      \linspann\{
        \ElementaryMatrix{i}{j}
        \mid i,j \in I
      \}
      \subseteq L^{1}(H_{1})
    $.
    By the $L^{1}$-density%
    \footnoteref{ft:L1-density:\beweislabel}
    of $\mathcal{W}$
    and the assumed $L^{1}$-boundedness of $\Phi$,
    it follows that $\Phi$ is trace-preserving.
  \end{proof}

The characterisation of CP(TP)\=/maps in \Cref{lemm:choi-jamiolkowski:sig:article-graph-raj-dahya}
thus motivates us to study bi-partite systems
and families of trace-class operators
$\{C_{i,j}\}_{i,j \in I} \subseteq L^{1}(H_{2})$
satisfying
\eqcref{it:1:lemm:choi-jamiolkowski:sig:article-graph-raj-dahya}
\resp
  \eqcref{it:1:lemm:choi-jamiolkowski:sig:article-graph-raj-dahya}
  and
  \eqcref{it:2:lemm:choi-jamiolkowski:sig:article-graph-raj-dahya}.




\subsection[Positive operators]{Positive operators}
\label{sec:positive:sig:article-graph-raj-dahya}

\firstparagraph
We now present some properties of positive operators on bi-partite systems.
Consider an arbitrary operator
$C \in \BoundedOps{H_{1} \otimes H_{2}}$
and assume that $C$ is positive,
in symbols $C \geq \zeromatrix$.%
\footnote{%
  when referring to operators on Hilbert spaces,
  recall that
  \highlightTerm{positive}
  means \highlightTerm{positive semi-definite}.
}
Observe immediately that
$
  C_{j,i}
  = (\ket{\BaseVector{j}} \otimes \onematrix)^{\ast}
    \:C
    \:(\ket{\BaseVector{i}} \otimes \onematrix)
  = \Big(
    (\ket{\BaseVector{i}} \otimes \onematrix)^{\ast}
    \:C
    \:(\ket{\BaseVector{j}} \otimes \onematrix)
    \Big)^{\ast}
  = C_{i,j}^{\ast}
$
and
$
  C_{i,i}
  = (\ket{\BaseVector{i}} \otimes \onematrix)^{\ast}
    \:C
    \:(\ket{\BaseVector{i}} \otimes \onematrix)
  \geq \zeromatrix
$
for each $i,j \in I$.
The following result is a simple consequence
of Cauchy\==Schwarz inequalities.


\begin{prop}
\makelabel{prop:cauchy-schwarz:sig:article-graph-raj-dahya}
  Let $C \in \BoundedOps{H_{1} \otimes H_{2}}$.
  If $C \geq \zeromatrix$, then

  \begin{restoremargins}
  \begin{equation}
  \label{eq:ineq-positive-bipartite:sig:article-graph-raj-dahya}
    \abs{\braket{\eta}{C_{i,j}\:\xi}}
    \leq
    \norm{\sqrt{C_{i,i}}\:\eta}
    \norm{\sqrt{C_{j,j}}\:\xi}
  \end{equation}
  \end{restoremargins}

  \continueparagraph
  for all $i,j \in I$,
  and $\xi, \eta \in H_{2}$.
  In particular
  $C_{i,j} = P_{i}\:C_{i,j}\:P_{j}$
  for all $i,j \in I$,
  where $P_{k} \coloneqq \Proj_{\quer{\ran}(C_{k,k})}$
  for each $k \in I$.
\end{prop}

  \begin{proof}
    Choose $z \in \complex$ with $\abs{z} = 1$
    such that
    $
      z^{\ast}\braket{\eta}{C_{i,j}\:\xi}
      = \abs{\braket{\eta}{C_{i,j}\:\xi}}
    $.
    For $\alpha,\beta > 0$,
    setting
      $x \coloneqq \alpha \BaseVector{i} \otimes \eta - z \beta \BaseVector{j} \otimes \xi$
    positivity of $C$ yields

      \begin{shorteqnarray}
        0
          &\leq
            &\braket{x}{C\:x}
          \\
          &= &\alpha^{2}\braket{\eta}{C_{i,i}\:\eta}
            + \beta^{2}\braket{\xi}{C_{j,j}\:\xi}
            - z^{\ast}\alpha\beta\braket{\eta}{C_{i,j}\:\xi}
            - z\alpha\beta\braket{\xi}{C_{j,i}\:\eta}
          \\
          &\overset{(\ast)}{=}
            &\alpha^{2}\braket{\eta}{C_{i,i}\:\eta}
            + \beta^{2}\braket{\xi}{C_{j,j}\:\xi}
            - z^{\ast}\alpha\beta\braket{\eta}{C_{i,j}\:\xi}
            - z\alpha\beta\braket{C_{i,j}\:\xi}{\eta}
          \\
          &=
            &\alpha^{2}\braket{\eta}{C_{i,i}\:\eta}
            + \beta^{2}\braket{\xi}{C_{j,j}\:\xi}
            - 2\alpha\beta \Re(z^{\ast}\braket{\eta}{C_{i,j}\:\xi})
          \\
          &=
            &\alpha^{2}\braket{\eta}{C_{i,i}\:\eta}
            + \beta^{2}\braket{\xi}{C_{j,j}\:\xi}
            - 2\alpha\beta\abs{\braket{\eta}{C_{i,j}\:\xi}}
      \end{shorteqnarray}

    \continueparagraph
    whereby ($\ast$) holds by virtue of $C$ being positive (see above).
    Thus

      \begin{restoremargins}
      \begin{equation}
      \label{eq:0:\beweislabel}
        \abs{\braket{\eta}{C_{i,j}\:\xi}}
          \leq
            \frac{
              \alpha^{2}\braket{\eta}{C_{i,i}\:\eta}
              + \beta^{2}\braket{\xi}{C_{j,j}\:\xi}
            }{2\alpha\beta}
      \end{equation}
      \end{restoremargins}

    \continueparagraph
    for all $\alpha, \beta > 0$.

    If $\braket{\eta}{C_{i,i}\:\eta} = 0$,
    letting $\beta > 0$ be arbitrary
    and $\alpha = \beta^{-1}$,
    \eqcref{eq:0:\beweislabel}
    yields
    $
      \abs{\braket{\eta}{C_{i,j}\:\xi}}
      \leq
      \frac{1}{2}\beta^{2}\braket{\xi}{C_{j,j}\:\xi}
    $.
    Since $\beta$ can be chosen to be arbitrarily small,
    one obtains
    $
      \abs{\braket{\eta}{C_{i,j}\:\xi}}
      = 0
      = \sqrt{
        \braket{\eta}{C_{i,i}\:\eta}
        \braket{\xi}{C_{j,j}\:\xi}
      }
    $.
    If $\braket{\xi}{C_{j,j}\:\xi} = 0$,
    then one similarly obtains
    $
      \abs{\braket{\eta}{C_{i,j}\:\xi}}
      = 0
      = \sqrt{
        \braket{\eta}{C_{i,i}\:\eta}
        \braket{\xi}{C_{j,j}\:\xi}
      }
    $.
    Otherwise plugging
      $\alpha \coloneqq \frac{1}{\sqrt{\braket{\eta}{C_{i,i}\:\eta}}}$
      and
      $\beta \coloneqq \frac{1}{\braket{\xi}{C_{j,j}\:\xi}}$
    into \eqcref{eq:0:\beweislabel}
    yields
    $
      \abs{\braket{\eta}{C_{i,j}\:\xi}}
      \leq
        \frac{1+1}{2\alpha\beta}
      = \sqrt{
        \braket{\eta}{C_{i,i}\:\eta}
        \braket{\xi}{C_{j,j}\:\xi}
      }
    $.
    So in all cases, and since $C_{i,i}$ and $C_{j,j}$ are positive,
    the expression in
    \eqcref{eq:ineq-positive-bipartite:sig:article-graph-raj-dahya}
    follows.

    The inequality in
    \eqcref{eq:ineq-positive-bipartite:sig:article-graph-raj-dahya}
    implies that
      $
        \braket{\eta}{C_{i,j}\:(\onematrix - P_{j})\:\xi} = 0
      $
    for $\xi, \eta \in H_{2}$.
    So $C_{i,j}\:(\onematrix - P_{j}) = \zeromatrix$,
    which implies that $C_{i,j} = C_{i,j}\:P_{j}$.
    Similar reasoning yields
      $(\onematrix - P_{i})\:C_{i,j} = \zeromatrix$
    and thus $C_{i,j} = P_{i}\:C_{i,j}$.
    The final claim thus follows.
  \end{proof}



\begin{highlightbox}
\begin{lemm}[Majorisation lemma]
\makelabel{lemm:majorisation:sig:article-graph-raj-dahya}
  Let $C \in \BoundedOps{H_{1} \otimes H_{2}}$ be positive
  for which each $C_{k,k} \in \FiniteRankOps{H_{2}}$.
  Then for each $i,j \in I$

  \begin{restoremargins}
  \begin{equation}
  \label{eq:radon-nikodym:sig:article-graph-raj-dahya}
    C_{i,j}
    = \hat{C}_{i,j}
      \:C_{j,j}
  \end{equation}
  \end{restoremargins}

  \continueparagraph
  for some $\hat{C}_{i,j} \in \BoundedOps{H_{2}}$.
  In particular, one can choose
  $\hat{C}_{i,j} \coloneqq C_{i,j}\:C_{j,j}^{\dagger}$.
\end{lemm}
\end{highlightbox}

  \begin{proof}
    For each $k \in I$, since $C$ is positive, so too is $C_{k,k}$.
    Since the latter is assumed to have finite rank,
    the pseudo-inverse $C_{k,k}^{\dagger}$ exists,
    noting that $C_{k,k}^{\dagger}C_{k,k} = P_{k}$,
    where $P_{k} \coloneqq \Proj_{\quer{\ran}(C_{k,k})}$
    (\cf Appendix \ref{app:spectral+mb:sig:article-graph-raj-dahya}).
    By \Cref{prop:cauchy-schwarz:sig:article-graph-raj-dahya}
    it follows that
      $
        C_{i,j}
        = C_{i,j}\:P_{j}
        = C_{i,j}\:(C_{j,j}^{\dagger}\:C_{j,j})
      $
    for each $i,j \in I$.
  \end{proof}

\begin{rem}
\makelabel{rem:majorisation-fp-necessary:sig:article-graph-raj-dahya}
  Without the finite rank assumption,
  one may apply Douglas's Lemma
  \cite[Theorem~1]{Douglas1966Article}
  to obtain a slightly weaker result,
  \viz
  $C_{i,j} = R_{i,j}\:\sqrt{C_{j,j}}$
  for some operator $R_{i,j} \in \BoundedOps{H_{2}}$.
  In general, however, the finite rank requirement in
  is essential to obtain
  \eqcref{eq:radon-nikodym:sig:article-graph-raj-dahya}:
  Consider $H_{1} = \complex^{2}$
  and $H_{2} = \ell^{2}(\naturals)$
  with the standard ONBs
  $\{\BaseVector{1}, \BaseVector{2}\}$
  \resp
  $\{\BaseVector{n}\}_{n\in\naturals}$.
  For $\alpha \in (0,\:1)$
  let $G_{\alpha} \coloneqq \sum_{n\in\naturals}\alpha^{n}\ketbra{n}{n}$.
  One can easily verify that
  $
    C
    \coloneqq
    \begin{smatrix}
        G_{1/2} &G_{1/\sqrt{6}}\\
        G_{1/\sqrt{6}} &G_{1/3}
    \end{smatrix}
  $
  is positive
  with each $C_{k,k}$ having infinite rank.
  Suppose $C_{1,2} = A\:C_{2,2}$ for some bounded operator $A \in \BoundedOps{H_{2}}$.
  Then
  $
    \sqrt{6}^{-n}
    = \braket{\BaseVector{n}}{C_{1,2}\:\BaseVector{n}}
    = \braket{\BaseVector{n}}{A\:C_{2,2}\:\BaseVector{n}}
    = 3^{-n}\braket{\BaseVector{n}}{A\:\BaseVector{n}}
  $
  and thus
  $\norm{A} \geq \sqrt{9/6}^{n}$
  for all $n \in \naturals$,
  which contradicts the boundedness of $A$.
\end{rem}




\subsection[Lower triangular operators]{Lower triangular operators}
\label{sec:triangular:sig:article-graph-raj-dahya}

\firstparagraph
In the proof of the main theorem, our goal shall be
to develop a Cholesky-like decomposition
for positive operators defined on bi-partite systems.
In order to achieve this we require some basic definitions and results
for lower triangular operators.
Recall that we have fixed an ONB $\{\BaseVector{i}\}_{i \in I}$ for $H_{1}$
where $I$ is a linearly ordered index set.


We call an operator
$D \in \BoundedOps{H_{1} \otimes H_{2}}$
\highlightTerm{diagonal},
if $D_{i,j} = \zeromatrix$
for all $i,j \in I$
with $i \neq j$.
It is easy to check that a diagonal operator
$D \in \BoundedOps{H_{1} \otimes H_{2}}$
is positive,
if and only if $D_{i,i} \geq \zeromatrix$
for all $i \in I$.
Say that an operator
$L \in \BoundedOps{H_{1} \otimes H_{2}}$
is \highlightTerm{(strictly) lower triangular},
if $L_{i,j} = \zeromatrix$
for all $i,j \in I$
with $j > i$
(\resp $j \geq i$).
An lower triangular operator
$\hat{L} \in \BoundedOps{H_{1} \otimes H_{2}}$
shall be called \highlightTerm{uni\=/triangular}
if $\hat{L}_{i,i} = \onematrix$.
In other words
$\hat{L}$ is lower uni\=/triangular
if and only if
$\hat{L} - \onematrix$ is strictly lower triangular.
And we say that
$L \in \BoundedOps{H_{1} \otimes H_{2}}$
is \highlightTerm{scaled lower triangular}
if $L = \hat{L}D$
where $\hat{L}$ is lower uni\=/triangular
and $D$ is diagonal and positive.



The following are well understood facts about triangular matrices,
simply reformulated for the setting of triangular operators defined over bi-partite systems.
We present them for completeness.

\begin{prop}
\makelabel{prop:triangular:support-product:sig:article-graph-raj-dahya}
  Let $F \subseteq I$ be finite.
  If
    $L, L^{\prime} \in \BoundedOps{H_{1} \otimes H_{2}}$
  are strictly lower triangular
  with
    $\supp(L) \subseteq F$
    and
    $\supp(L^{\prime}) \subseteq F$,
  then
    $\supp(L + L^{\prime}) \subseteq F$
    and
    $\supp(L \: L^{\prime}) \subseteq F$.
  And if
    $\hat{L}, \hat{L}^{\prime} \in \BoundedOps{H_{1} \otimes H_{2}}$
  are lower uni\=/triangular
  with
    $\supp(\hat{L} - \onematrix) \subseteq F$
    and
    $\supp(\hat{L}^{\prime} - \onematrix) \subseteq F$,
  then
    $\supp(\hat{L} \: \hat{L}^{\prime} - \onematrix) \subseteq F$.
\end{prop}

  \begin{proof}
    Towards the first claims \cf the observations
    made in \eqcref{eq:bipartite:operations:sig:article-graph-raj-dahya}
    at the start of \S{}\ref{sec:bipartite:sig:article-graph-raj-dahya}
    The second claim follows from these inclusions,
    since
    $
      \hat{L} \: \hat{L}^{\prime} - \onematrix
      = L + L^{\prime} + L \:L^{\prime}
    $
    where
    $
      L \coloneqq \hat{L} - \onematrix
    $
    and
    $
      L^{\prime} \coloneqq \hat{L}^{\prime} - \onematrix
    $
    are strictly lower triangular
    with
      $\supp(L) \subseteq F$
      and
      $\supp(L^{\prime}) \subseteq F$.
  \end{proof}

\begin{prop}
\makelabel{prop:triangular:inverse:sig:article-graph-raj-dahya}
  Let $\hat{L} \in \BoundedOps{H_{1} \otimes H_{2}}$
  be lower uni\=/triangular.
  If the strictly lower triangular operator
    $\hat{L} - \onematrix$
  has finite support,
  then $\hat{L}$ is invertible.
  Moreover, $\hat{L}^{-1}$ is lower uni\=/triangular
  with
  $
    \supp (\hat{L}^{-1} - \onematrix)
    = \supp (\hat{L} - \onematrix)
  $.
\end{prop}

  \begin{proof}
    Let $F \coloneqq \supp(\hat{L} - \onematrix)$.
    We recursively define for each $i \in F$
    \wrt the ordering on $I$

    \begin{restoremargins}
    \begin{equation}
    \label{eq:lower-uni-tri-inverse:sig:article-graph-raj-dahya}
      R_{i,j}
        \coloneqq
            \sum_{
              k \in F:
              ~j < k < i
            }
                \hat{L}_{i,k}
                \:R_{k,j}
            +
            \hat{L}_{i,j},
    \end{equation}
    \end{restoremargins}

    \continueparagraph
    for each $j \in F$
    with $j < i$.
    Setting
    $
      R
      \coloneqq
      \sum_{
        i,j \in F\colon
        j < i
      }
        \ElementaryMatrix{i}{j} \otimes R_{ij}
    $
    one clearly has that
    $\hat{R} \coloneqq \onematrix - R$
    is lower uni\=/triangular with
    $
      \supp (\hat{R} - \onematrix)
      = \supp(R)
      \subseteq F
      = \supp(\hat{L} - \onematrix)
    $.
    It is a simple exercise to verify that
    $\hat{R}$ is a right inverse for $\hat{L}$.
    By an analogous argument,
    the lower uni\=/triangular operator $\hat{R}$ also has a right inverse.
    Simple algebraic arguments thus yield that $\hat{R}^{-1} = \hat{L}$.%
    \footnote{%
      If elements in a unital ring $a,b,c$  satisfy
      $ab = 1 = bc$, then $a = a(bc) = (ab)c = c$.
      Thus $b$ is a left and right inverse and thus \emph{the} inverse of $a$.
    }
    By symmetry one has
    $
      \supp(\hat{L} - \onematrix)
      = \supp(\hat{R}^{-1} - \onematrix)
      \subseteq \supp(\hat{R} - \onematrix)
    $.
    Thus
    $
      \supp(\hat{L}^{-1} - \onematrix)
      = \supp(\hat{R} - \onematrix)
      = \supp(\hat{L} - \onematrix)
    $.
  \end{proof}

\begin{rem}[Determination of entries of inverse]
\makelabel{rem:compute-restrictions-of-inverse:sig:article-graph-raj-dahya}
  Suppose $I = \naturals$ or $\{1,2,\ldots,N\}$ for some $N \in \naturals$.
  Then the finite support
  satisfies
  $\supp(\hat{L} - \onematrix) \subseteq F$
  for some initial segment of indices
  $F \coloneqq \{1,2,\ldots,N'\} \subseteq I$
  for some $N' \in I$.
  By \eqcref{eq:lower-uni-tri-inverse:sig:article-graph-raj-dahya}
  for each $i \in \{1,2,\ldots,N'\}$,
  the entries
  $\{\hat{L}^{-1}_{i,j}\}_{j=1}^{i}$
  can be completely purely based on
  the entries in
  $\{\hat{L}^{-1}_{i',j'}\}_{i',j'=1}^{i}$.
\end{rem}




\subsection[Cholesky decomposition]{Cholesky decomposition}
\label{sec:cholesky:sig:article-graph-raj-dahya}


\firstparagraph
In this subsection we present a bi-partite variant
of methods to factorise positive matrices.
These methods trace their origins to unpublished works by Cholesky
(see
  \cite[\S{}4.4]{BrezinskiTournes2014InbookChapter4},
  \cite[\S{}1--3.1]{DeFalguerolles2019Article},
  \cf also
  \cite[Algorithm~23.1]{TrefethenBau1997Book}%
)
and are more closely related to \usesinglequotes{block} versions
of the classical techniques.
Given a (necessarily positive) operator
$C \in \BoundedOps{H_{1} \otimes H_{2}}$,
a \highlightTerm{bi-partite Cholesky decomposition}
of $C$ shall be taken to be any expression of the form

\begin{restoremargins}
\begin{equation}
\label{eq:cholesky:LDL-star:sig:article-graph-raj-dahya}
  C = \hat{L}\:D\:\hat{L}^{\ast}
  = L\:L^{\ast}
\end{equation}
\end{restoremargins}

\continueparagraph
where $\hat{L}$ is a lower uni\=/triangular operator,
$D$ is diagonal and positive,
and $L$ is the scaled lower triangular operator
defined by $L \coloneqq \hat{L}\:\sqrt{D}$.
Note that $\sqrt{D}$ is itself diagonal and positive
with entries $(\sqrt{D})_{i,i} = \sqrt{D_{i,i}}$ for $i \in I$.

Observe further that if $\hat{L} - \onematrix$ and $D$ have finite support,
say $\supp(D),\supp(\hat{L} - \onematrix) \subseteq F$
for some finite set $F \subseteq I$,
then so too does $L \coloneqq \hat{L}\:\sqrt{D}$
with $\supp(L) \subseteq F$,
since

\begin{restoremargins}
\begin{equation}
\label{eq:finite-support-LD:sig:article-graph-raj-dahya}
\everymath={\displaystyle}
\begin{array}[m]{rcl}
  \hat{L}\:\sqrt{D}
    &=
      &\sqrt{D} +  (\hat{L} - \onematrix)\:\sqrt{D}
      \\
    &=
      &\sum_{i \in F}
        \ElementaryMatrix{i}{i}
        \otimes
        \underbrace{
          \sqrt{D_{i,i}}
        }_{
          =\hat{L}_{i,i}\:\sqrt{D_{i,i}}
        }
      -
      \sum_{i,j \in F\colon j < i}
        \ElementaryMatrix{i}{i}
        \otimes
        \hat{L}_{i,j}
        \:\sqrt{D_{j,j}}
      \\
    &=
      &\sum_{i,j \in F\colon j \leq i}
        \ElementaryMatrix{i}{i}
        \otimes
        \hat{L}_{i,j}
        \:\sqrt{D_{j,j}}.
\end{array}
\end{equation}
\end{restoremargins}

Under appropriate assumptions, the existence and uniqueness of bi-partite decompositions
can be established analogously to the classical result for positive operators on $H_{1}$
(\cf
  \cite[\S{}2]{Rutishauser1958Article},
  \cite[\S{}8.3 and \S{}9.19]{Wilkinson1988Book}%
),
whereby the main challenge is that we are essentially treating operator- instead of scalar-valued matrices.
The key ingredient to this result is the Cauchy\==Schwarz result
in \Cref{prop:cauchy-schwarz:sig:article-graph-raj-dahya}
and its consequence in \Cref{lemm:majorisation:sig:article-graph-raj-dahya}.



\begin{highlightbox}
\begin{lemm}[Existence of bi-partite Cholesky\=/decompositions]
\makelabel{lemm:cholesky:existence:sig:article-graph-raj-dahya}
  Let
    $C \in \BoundedOps{H_{1} \otimes H_{2}}$
  be positive
  with $\supp(C) \subseteq F$
  for some finite set $F \subseteq I$.
  Letting $\mathcal{K} \coloneqq \FiniteRankOps{H_{2}}$
  denote the ideal of finite rank operators,
  if each $C_{i,j} \in \mathcal{K}$,
  then a bi-partite Cholesky decomposition
  \akin \eqcref{eq:cholesky:LDL-star:sig:article-graph-raj-dahya}
  exists such that
  $
    \supp(\hat{D}),
    \supp(\hat{L} - I)
    \subseteq F
  $
  and each $D_{i,i} \in \mathcal{K}$.
\end{lemm}
\end{highlightbox}

  \begin{proof}[of \Cref{\beweislabel}]
    The claim shall be proved by induction over the size of $F$.
    For the base case, $F = \emptyset$,
    one has
    $
      C = \zeromatrix
    $,
    which is already a positive diagonal operator.
    So
    $
      C = \hat{L}\:D\:(\hat{L})^{\ast}
    $,
    where
      $D = \zeromatrix$,
      which is positive diagonal
      with $\supp(D) = \emptyset = F$,
    and
      $\hat{L} = \onematrix$,
      which is lower uni\=/triangular
      with $\supp(\hat{L} - \onematrix) = \emptyset = F$.
    And clearly each $D_{i,i} = \zeromatrix \in \mathcal{K}$.

    For the inductive case let $F \subseteq I$ with $\card{F} \geq 1$.
    Let $i_{0} \coloneqq \max F$
    and set $F^{\prime} \coloneqq F \without \{i_{0}\}$.
    Letting
    $
      p_{F^{\prime}} \coloneqq \sum_{i \in F^{\prime}}\ElementaryMatrix{i}{i} \otimes I
    $
    be the projection onto the subspace
    $
      \linspann\{\BaseVector{i} \mid i \in F^{\prime}\} \otimes H_{2}
    $
    of $H_{1} \otimes H_{2}$,
    one clearly has that
    $
      C^{\prime}
      \coloneqq
        p_{F^{\prime}}\:C\:p_{F^{\prime}}^{\ast}
      = \sum_{i,j \in F^{\prime}}
          \ElementaryMatrix{i}{j}
          \otimes
          C_{i,j}
    $
    is positive
    with $\supp(C^{\prime}) \subseteq F^{\prime}$
    and $C^{\prime}_{i,j} \in \{C_{i,j}, \zeromatrix\} \subseteq \mathcal{K}$
    for $i,j \in I$.
    By induction we may assume that $C^{\prime}$
    possesses a bi-partite Cholesky decomposition
    $
      C^{\prime} = \hat{L}^{\prime}\:D^{\prime}\:(\hat{L}^{\prime})^{\ast}
    $
    \akin \eqcref{eq:cholesky:LDL-star:sig:article-graph-raj-dahya},
    such that
    $
      \supp(D^{\prime}),
      \supp(\hat{L}^{\prime} - I)
      \subseteq F^{\prime}
    $
    and each $D^{\prime}_{i,i} \in \mathcal{K}$.
    Our goal is to extend this construction
    to a decomposition for $C$.
    To this end we make a few observations:

    \begin{enumerate}[
      label={\bfseries{\arabic*.}},
      ref={\arabic*},
      left=\rtab,
    ]
    \item
      We can express $C$ in terms of $C^{\prime}$ as follows:

        \begin{restoremargins}
        \begin{equation}
        \label{eq:0:\beweislabel}
        \everymath={\displaystyle}
        \begin{array}[m]{rcl}
          C
            &= &C^{\prime}
              + \sum_{i \in F^{\prime}}
                  \ElementaryMatrix{i}{i_{0}}
                  \otimes
                  C_{i,i_{0}}
              + \underbrace{
                  \sum_{j \in F^{\prime}}
                    \ElementaryMatrix{i_{0}}{j}
                    \otimes
                    C_{i_{0},j}
                }_{\eqqcolon W}
              + \ElementaryMatrix{i_{0}}{i_{0}}
                \otimes
                C_{i_{0},i_{0}}
              \\
            &=
              &\hat{L}^{\prime}\:D^{\prime}\:(\hat{L}^{\prime})^{\ast}
              + W^{\ast} + W
              + \ElementaryMatrix{i_{0}}{i_{0}}
                \otimes
                C_{i_{0},i_{0}}.
        \end{array}
        \end{equation}
        \end{restoremargins}

    \item
      By \Cref{prop:triangular:inverse:sig:article-graph-raj-dahya}
      $
        \supp((\hat{L}^{\prime})^{-1} - \onematrix)
        = \supp(\hat{L}^{\prime} - \onematrix)
        \subseteq F^{\prime}
      $.
      Since $i_{0} \notin F^{\prime}$,
      this entails
        $
          ((\hat{L}^{\prime})^{-1} - \onematrix)
          \:(\ElementaryMatrix{i_{0}}{i_{0}} \otimes \onematrix)
          = \zeromatrix
        $
      and thus
        $
          (\hat{L}^{\prime})^{-1}
          \:(\ElementaryMatrix{i_{0}}{i_{0}} \otimes \onematrix)
          = \ElementaryMatrix{i_{0}}{i_{0}} \otimes \onematrix
        $.
      So since
      $
        \ElementaryMatrix{i_{0}}{i_{0}} \otimes C_{i_{0},i_{0}}
        = (\ElementaryMatrix{i_{0}}{i_{0}} \otimes \onematrix)
          \:(\ElementaryMatrix{i_{0}}{i_{0}} \otimes C_{i_{0},i_{0}})
          \:(\ElementaryMatrix{i_{0}}{i_{0}} \otimes \onematrix)^{\ast}
      $,
      one has

        \begin{shorteqnarray}
          (\hat{L}^{\prime})^{-1}
          \:(\ElementaryMatrix{i_{0}}{i_{0}} \otimes C_{i_{0},i_{0}})
          \:(\hat{L}^{\prime})^{-\ast}
          = \ElementaryMatrix{i_{0}}{i_{0}} \otimes C_{i_{0},i_{0}},
        \end{shorteqnarray}

      \continueparagraph
      and since $W = (\ElementaryMatrix{i_{0}}{i_{0}} \otimes \onematrix)\:W$,
      one also obtains

        \begin{shorteqnarray}
          \widetilde{W}
            \coloneqq
              (\hat{L}^{\prime})^{-1}
              \:W
              \:(\hat{L}^{\prime})^{-\ast}
            &=
              &W\:(\hat{L}^{\prime})^{-\ast}
              \\
            &=
              &W
              +
              W\:\underbrace{
                ((\hat{L}^{\prime})^{-\ast} - \onematrix_{H_{1} \otimes H_{2}})
              }_{
                \supp(\cdot)
                \subseteq F^{\prime}
              }
              \\
            &=
              &W
              +
              W
              \sum_{j,k \in F^{\prime}}
                \ElementaryMatrix{k}{j}
                \otimes
                ((\hat{L}^{\prime})^{-\ast} - \onematrix_{H_{1}} \otimes \onematrix_{H_{2}})_{k,j}
              \\
            &=
              &W\:p_{F'}
              +
              W
              \:\Big(
                \sum_{j,k \in F^{\prime}}
                  \ElementaryMatrix{k}{j}
                  \otimes
                  ((\hat{L}^{\prime})^{-\ast})_{k,j}
                -
                p_{F'}
              \Big)
              \\
            &=
              &\Big(
                \sum_{k \in F^{\prime}}
                  \ElementaryMatrix{i_{0}}{k}
                  \otimes
                  C_{i_{0},k}
              \Big)
              \:\Big(
                \sum_{j,k \in F^{\prime}}
                  \ElementaryMatrix{k}{j}
                  \otimes
                  ((\hat{L}^{\prime})^{-1})^{\ast}_{j,k}
              \Big)
              \\
            &=
              &\sum_{j \in F^{\prime}}
                \ElementaryMatrix{i_{0}}{j}
                \otimes
                \sum_{k \in F^{\prime}}
                C_{i_{0},k}
                \:((\hat{L}^{\prime})^{-1})^{\ast}_{j,k}
        \end{shorteqnarray}

    \item
      Applying the above two expressions to \eqcref{eq:0:\beweislabel}
      yields
      $
        C
        = \hat{L}^{\prime}\:A\:(\hat{L}^{\prime})^{\ast}
      $,
      where

        \begin{shorteqnarray}
          A
            \coloneqq
              D^{\prime}
              + \widetilde{W}^{\ast}
              + \widetilde{W}
              + \ElementaryMatrix{i_{0}}{i_{0}}
                \otimes
                C_{i_{0},i_{0}},
        \end{shorteqnarray}

      \continueparagraph
      which is positive,
      since $A = (\hat{L}^{\prime})^{-1}\:C\:(\hat{L}^{\prime})^{-\ast}$.

    \item\label{it:majorisation:\beweislabel}
      Let $j \in F^{\prime}$.
      By the majorisation lemma (\Cref{lemm:majorisation:sig:article-graph-raj-dahya}),
      since $D_{j,j}$ is a finite rank operator,
      there exists $\hat{A}_{i_{0},j} \in \BoundedOps{H_{2}}$
      such that

        \begin{shorteqnarray}
          \widetilde{W}_{i_{0},j}
            = A_{i_{0},j}
            = \hat{A}_{i_{0},j}\:A_{j,j}
            = \hat{A}_{i_{0},j}\:D^{\prime}_{j,j}
        \end{shorteqnarray}

      \continueparagraph
      holds.
      In particular one may choose

      \begin{restoremargins}
      \begin{equation}
      \label{eq:cholesky-construction:1:sig:article-graph-raj-dahya}
        \hat{A}_{i_{0},j}
          \coloneqq
            \widetilde{W}_{i_{0},j}
            \:(D^{\prime}_{j,j})^{\dagger}
          = \sum_{k \in F^{\prime}}
              C_{i_{0},k}
              \:((\hat{L}^{\prime})^{-1})^{\ast}_{j,k}
              \:(D^{\prime}_{j,j})^{\dagger}
      \end{equation}
      \end{restoremargins}

      \continueparagraph
      for each $j \in F^{\prime}$.

    \item
      Since $D^{\prime}$ is diagonal
      the above computation yields

        \begin{restoremargins}
        \begin{equation}
        \label{eq:2:\beweislabel}
          \widetilde{W}
            =
              \sum_{j \in F^{\prime}}
                \ElementaryMatrix{i_{0}}{j} \otimes \widetilde{W}_{i_{0},j}
              \\
            =
              \Big(
                \underbrace{
                  \sum_{j \in F^{\prime}}
                    \ElementaryMatrix{i_{0}}{j} \otimes \hat{A}_{i_{0},j}
                }_{\eqqcolon \hat{W}}
                \Big)
              \:D^{\prime}
        \end{equation}
        \end{restoremargins}

    \end{enumerate}

    \continueparagraph
    We now claim that the decomposition in
    \eqcref{eq:cholesky:LDL-star:sig:article-graph-raj-dahya}
    holds with

    \begin{restoremargins}
    \begin{equation}
    \label{eq:cholesky-construction:2:sig:article-graph-raj-dahya}
    \everymath={\displaystyle}
    \begin{array}[m]{rcl}
      \hat{L}
        &\coloneqq
          &\hat{L}^{\prime} + \hat{W}
          = \hat{L}^{\prime}
            + \sum_{j \in F^{\prime}}
                \ElementaryMatrix{i_{0}}{j}
                \otimes
                \hat{A}_{i_{0},j}
          ~\text{and}
        \\
      D
        &\coloneqq
          &D^{\prime}
          +
          \ElementaryMatrix{i_{0}}{i_{0}}
          \otimes
          \Big(
            C_{i_{0},i_{0}}
            -
            \underbrace{
              \sum_{j \in F^{\prime}}
                \hat{A}_{i_{0},j}
                \:D^{\prime}_{j,j}
                \:\hat{A}_{i_{0},j}^{\ast}
            }_{
              \eqqcolon B
            }
          \Big).
    \end{array}
    \end{equation}
    \end{restoremargins}

    \continueparagraph
    To this end first observe that

    \begin{restoremargins}
    \begin{equation}
    \label{eq:3:\beweislabel}
    \everymath={\displaystyle}
    \begin{array}[m]{rcl}
      \hat{W}\:D^{\prime}\:\hat{W}^{\ast}
        &\eqcrefoverset{eq:2:\beweislabel}{=}
          &\sum_{j,k,l \in F^{\prime}}
            (\ElementaryMatrix{i_{0}}{j} \otimes \hat{A}_{i_{0},j})
            \:(\ElementaryMatrix{k}{k} \otimes D^{\prime}_{k,k})
            \:(\ElementaryMatrix{l}{i_{0}} \otimes \hat{A}_{i_{0},l}^{\ast})
          \\
        &=
          &\sum_{j \in F^{\prime}}
            \ElementaryMatrix{i_{0}}{i_{0}}
            \otimes
              \hat{A}_{i_{0},j}
              \:D^{\prime}_{j,j}
              \:\hat{A}_{i_{0},j}^{\ast}
          \\
        &=
          &\ElementaryMatrix{i_{0}}{i_{0}}
          \otimes
          B
    \end{array}
    \end{equation}
    \end{restoremargins}

    \continueparagraph
    and

    \begin{restoremargins}
    \begin{equation}
    \label{eq:4:\beweislabel}
      R
        \coloneqq
          \hat{W}\:(\ElementaryMatrix{i_{0}}{i_{0}} \otimes D_{i_{0},i_{0}})
        \eqcrefoverset{eq:2:\beweislabel}{=}
          \sum_{j \in F^{\prime}}
            (
              \ElementaryMatrix{i_{0}}{j}
              \otimes
              \hat{W}_{i_{0},j}
            )
            \:(
              \ElementaryMatrix{i_{0}}{i_{0}}
              \otimes
              D_{i_{0},i_{0}}
            )
        = \zeromatrix,
    \end{equation}
    \end{restoremargins}

    \continueparagraph
    since $i_{0} \notin F^{\prime}$.
    By
      \eqcref{eq:2:\beweislabel},
      \eqcref{eq:3:\beweislabel},
      and
      \eqcref{eq:4:\beweislabel}
    one obtains

    \begin{shorteqnarray}
      (\onematrix + \hat{W})
        \:D
        \:(\onematrix + \hat{W})^{\ast}
        &=
          &(\onematrix + \hat{W})
          \:\Big(
            D^{\prime} + \ElementaryMatrix{i_{0}}{i_{0}} \otimes (C_{i_{0},i_{0}} - B)
          \Big)
          \:(\onematrix + \hat{W})^{\ast}
        \\
        &=
          &D^{\prime} + \ElementaryMatrix{i_{0}}{i_{0}} \otimes (C_{i_{0},i_{0}} - B)
          \\
          &&+ \hat{W}\:D^{\prime} + \cancel{R}
          \\
          &&+ (\hat{W}\:D^{\prime})^{\ast} + \cancel{R}^{\ast}
          \\
          &&+ \hat{W}\:D^{\prime}\:\hat{W}^{\ast} + \cancel{R}\:\hat{W}^{\ast}
        \\
        &=
          &D^{\prime} + \ElementaryMatrix{i_{0}}{i_{0}} \otimes (C_{i_{0},i_{0}} - B)
          + \widetilde{W}
          + \widetilde{W}^{\ast}
          + \ElementaryMatrix{i_{0}}{i_{0}} \otimes B
        \\
        &= &A.
    \end{shorteqnarray}

    \continueparagraph
    And since
      $\supp(\hat{L}^{\prime} - \onematrix) \subseteq F^{\prime} \notni i_{0}$,
    one has
    $(\hat{L}^{\prime} - \onematrix)\:\hat{W} = \zeromatrix$
    and thus
    $
      \hat{L}
      = \hat{L}^{\prime} \: (\onematrix + \hat{W})
    $.
    So

    \begin{shorteqnarray}
      \hat{L}\:D\:(\hat{L})^{\ast}
        &=
          &\hat{L}^{\prime}
          \:(\onematrix + \hat{W})
          \:D
          \:(\onematrix + \hat{W})^{\ast}
          \:(\hat{L}^{\prime})^{\ast}
        \\
        &=
          &\hat{L}^{\prime}
          \:A
          \:(\hat{L}^{\prime})^{\ast}
        = C.
    \end{shorteqnarray}

    Thus \eqcref{eq:cholesky:LDL-star:sig:article-graph-raj-dahya} holds.
    Note by construction that $\hat{L}$ is lower uni\=/triangular
    and $D$ is diagonal
    with $
      \supp(D)
      \subseteq
        \supp(D^{\prime}) \cup \{i_{0}\}
      \subseteq F^{\prime} \cup \{i_{0}\}
      = F
    $.
    Since $D = (\hat{L})^{-1}\:C\:(\hat{L})^{-\ast}$
    and $C$ is positive,
    one has that $D$ is positive.
    Moreover, by induction
    for each $i \in I \without \{i_{0}\}$ one has
    $
      D_{i,i}
      = D^{\prime}_{i,i}
      \in \mathcal{K}
    $
    and for $i=i_{0}$ one has
    $
      D_{i,i}
      = C_{i_{0},i_{0}}
        -
        \sum_{j \in F^{\prime}}
          \hat{A}_{i_{0},j}
          D^{\prime}_{j,j}
          \hat{A}_{i_{0},j}^{\ast}
      \in \mathcal{K}
    $,
    since $C_{i_{0},i_{0}}$ and each $D^{\prime}_{j,j}$
    are in the ideal $\mathcal{K}$.
    Finally, since $\hat{L} = \hat{L}^{\prime} + \hat{W}$,
    by construction of $\hat{W}$
    one can verify that
      $
        \supp(\hat{L} - \onematrix)
        \subseteq
          F^{\prime} \cup \{i_{0}\}
        = F
      $.
    We have thus achieved a bi-partite Cholesky decomposition for $C$
    satisfying the desired property.
  \end{proof}



\begin{lemm}[Uniqueness of bi-partite Cholesky\=/decomposition]
\makelabel{lemm:cholesky:uniqueness:sig:article-graph-raj-dahya}
  Suppose that

    \begin{restoremargins}
    \begin{equation}
    \label{eq:cholesky-2:\beweislabel}
      C
      \coloneqq \hat{L}\:D\:\hat{L}^{\ast}
      = \hat{L}^{\prime}\:D^{\prime}\:(\hat{L}^{\prime})^{\ast}
    \end{equation}
    \end{restoremargins}

  \continueparagraph
  where
    $D, D^{\prime} \in \BoundedOps{H_{1} \otimes H_{2}}$
    are diagonal and positive
    and
    $\hat{L},\hat{L}^{\prime} \in \BoundedOps{H_{1} \otimes H_{2}}$
    are lower uni\=/triangular
    such that
    $D$,
    $D^{\prime}$,
    $\hat{L} - \onematrix$,
    and
    $\hat{L}^{\prime} - \onematrix$
    have finite support.
  Then
    $L = L^{\prime}$,
  where
    $L \coloneqq \hat{L}\:\sqrt{D}$
    and
    $L^{\prime} \coloneqq \hat{L}^{\prime}\:\sqrt{D}^{\prime}$.
\end{lemm}

    \begin{proof}
      By assumption,
        $
          \supp(D),
          \supp(D^{\prime}),
          \supp(\hat{L} - \onematrix),
          \supp(\hat{L}^{\prime} - \onematrix) \subseteq F
        $
      for some finite set $F \subseteq I$.
      By \Cref{prop:triangular:inverse:sig:article-graph-raj-dahya},
      $\hat{L}^{-1}$ is lower uni\=/triangular
      with
      $
        \supp(\hat{L}^{-1} - \onematrix)
        = \supp(\hat{L} - \onematrix)
        \subseteq F
      $.
      By \Cref{prop:triangular:support-product:sig:article-graph-raj-dahya}
      and since lower uni\=/triangular operators
      are closed under products,
      one has that
        $
          \hat{L}^{\prime\prime}
          \coloneqq
          \hat{L}^{-1}\:\hat{L}^{\prime}
        $
      is a lower uni\=/triangular operator
      with
      $
        \supp(\hat{L}^{\prime\prime} - \onematrix)
        \subseteq F
      $.
      Finally define
      the lower triangular operator

        \begin{shorteqnarray}
          L^{\prime\prime}
            \coloneqq
              &\hat{L}^{\prime\prime}\:\sqrt{D^{\prime}}
            =
              \sqrt{D^{\prime}}
              + \underbrace{
                (\hat{L}^{\prime\prime} - \onematrix)
              }_{\supp(\cdot) \subseteq F}
              \:\sqrt{D^{\prime}}
        \end{shorteqnarray}

      \continueparagraph
      which has support $\supp(L^{\prime\prime}) \subseteq F$
      by \Cref{prop:triangular:inverse:sig:article-graph-raj-dahya}.
      Observe that

        \begin{restoremargins}
        \begin{equation}
        \label{eq:1:\beweislabel}
        \everymath={\displaystyle}
        \begin{array}[m]{rcl}
          L^{\prime\prime}\:(L^{\prime\prime})^{\ast}
            &= &\hat{L}^{-1}
              \:\hat{L}^{\prime}
              \:\sqrt{D^{\prime}}
              \:\sqrt{D^{\prime}}
              \:(\hat{L}^{\prime})^{\ast}
              \:(\hat{L}^{-1})^{\ast}
            \\
            &= &\hat{L}^{-1}
              \:\Big(
              \hat{L}^{\prime}
              \:D^{\prime}
              \:(\hat{L}^{\prime})^{\ast}
              \Big)
              \:(\hat{L}^{-1})^{\ast}
            \\
            &\eqcrefoverset{eq:cholesky-2:\beweislabel}{=}
              &\hat{L}^{-1}\:(\hat{L}\:D\:\hat{L}^{\ast})\:(\hat{L}^{-1})^{\ast}
              = D.
        \end{array}
        \end{equation}
        \end{restoremargins}

      We now claim that $L^{\prime\prime}$ is diagonal.
      If this were not the case,
      then $L^{\prime\prime}_{i_{0},j_{0}} \neq \zeromatrix$.
      for some $i_{0},j_{0} \in I$
      with $j_{0} < i_{0}$.
      Since
        $L^{\prime\prime}_{i,j} = \zeromatrix$
      for $i,j \in I \without F$ with $j < i$ (see above),
      we necessarily have
        $i_{0},j_{0} \in F$.
      We can thus choose $j_{0}$ to be the minimal index $j \in F$
      for which $L^{\prime\prime}_{i_{0},j} \neq \zeromatrix$.
      Since $L^{\prime\prime}$ has finite support,
      the above expression yields

        \begin{shorteqnarray}
          \zeromatrix
            = D_{i_{0},j_{0}}
            &\eqcrefoverset{eq:1:\beweislabel}{=}
              &(L^{\prime\prime}\:(L^{\prime\prime})^{\ast})_{i_{0},j_{0}}
              \\
            &= &\sum_{k \in F}
                L^{\prime\prime}_{i_{0},k}
                \:(L^{\prime\prime}_{j_{0},k})^{\ast}
              \\
            &= &\sum_{k \in F\colon k \leq i_{0}, j_{0}}
                L^{\prime\prime}_{i_{0},k}
                \:(L^{\prime\prime}_{j_{0},k})^{\ast}
              \\
            &&\text{
              since $L^{\prime\prime}$ is lower triangular
            }\\
            &= &\sum_{k \in F\colon k \leq j_{0}}
                L^{\prime\prime}_{i_{0},k}
                \:(L^{\prime\prime}_{j_{0},k})^{\ast}
              \\
            &\overset{(\ast)}{=}
              &L^{\prime\prime}_{i_{0},j_{0}}
                \:(L^{\prime\prime}_{j_{0},j_{0}})^{\ast}
              \\
            &=
              &(
                \hat{L}^{\prime\prime}
                \:\sqrt{D^{\prime}}
              )_{i_{0},j_{0}}
              \:(
                \hat{L}^{\prime\prime}
                \:\sqrt{D^{\prime}}
              )_{j_{0},j_{0}}^{\ast}
              \\
            &=
              &\hat{L}^{\prime\prime}_{i_{0},j_{0}}
              \:\sqrt{D^{\prime}_{j_{0},j_{0}}}
              \:\sqrt{D^{\prime}_{j_{0},j_{0}}}
              \:(\hat{L}^{\prime\prime}_{j_{0},j_{0}})^{\ast}
              \\
            &=
              &\hat{L}^{\prime\prime}_{i_{0},j_{0}}
              \:D^{\prime}_{j_{0},j_{0}}
              \:\onematrix,
        \end{shorteqnarray}

      \continueparagraph
      whereby ($\ast$) holds
      since $L^{\prime\prime}_{i_{0},k} = \zeromatrix$
      for $k < j_{0}$ by minimality of $j_{0}$.
      The final expression implies that
      $
        \zeromatrix
        = \hat{L}^{\prime\prime}_{i_{0},j_{0}}
        \:D^{\prime}_{j_{0},j_{0}}
        \:(\hat{L}^{\prime\prime}_{i_{0},j_{0}})^{\ast}
        = L^{\prime\prime}_{i_{0},j_{0}}\:(L^{\prime\prime}_{i_{0},j_{0}})^{\ast}
      $,
      which in turn implies that
      $L^{\prime\prime}_{i_{0},j_{0}} = \zeromatrix$,
      a contradiction.

      Now since $L^{\prime\prime}$ is diagonal,
      by \eqcref{eq:1:\beweislabel} one has
      $
        L^{\prime\prime}_{i,i}\:(L^{\prime\prime}_{i,i})^{\ast}
        = (L^{\prime\prime}\:(L^{\prime\prime})^{\ast})_{i,i}
        = D_{i,i}
      $
      for each $i \in I$.
      Note that each $L^{\prime\prime}_{i,i} \geq \zeromatrix$,
      since
      $
        L^{\prime\prime}_{i,i}
        = (\hat{L}^{\prime\prime}\:\sqrt{D^{\prime}})_{i,i}
        = \hat{L}^{\prime\prime}_{i,i}\:\sqrt{D^{\prime}_{i,i}}
        = \onematrix\:\sqrt{D^{\prime}_{i,i}}
      $.
      Thus
      $
        L^{\prime\prime}_{i,i}
        = \sqrt{D_{i,i}}
      $
      for each $i \in I$,
      whence
      $
        L^{\prime\prime}
        = \sum_{i \in F}
            \ElementaryMatrix{i}{i}
            \otimes
            L^{\prime\prime}_{i,i}
        = \sum_{i \in F}
            \ElementaryMatrix{i}{i}
            \otimes
            \sqrt{D_{i,i}}
        = \sqrt{D}
      $,
      since $L^{\prime\prime}$ is diagonal with finite support.
      Hence
      $
        D
        = L^{\prime\prime}
        = \hat{L}^{\prime\prime}\:\sqrt{D^{\prime}}
        = \hat{L}^{-1}\:\hat{L}^{\prime}\:\sqrt{D^{\prime}}
      $,
      from which
      $
        \hat{L}\:\sqrt{D}
        = \hat{L}^{\prime}\:\sqrt{D^{\prime}}
      $
      follows.
    \end{proof}

By \Cref{lemm:cholesky:uniqueness:sig:article-graph-raj-dahya}
we may speak of \emph{the} bi-partite Cholesky decomposition,
which by \Cref{lemm:cholesky:existence:sig:article-graph-raj-dahya}
always exists for finitely supported finite rank positive operators.




\subsection[Resolution of CP-/maps]{Resolution of CP\=/maps}
\label{sec:separation:sig:article-graph-raj-dahya}

\firstparagraph
Via the Choi\==Jamio{\l}kowski correspondence
and bi-partite Cholesky decomposition,
we obtain our primary means to analyse CP(TP)\=/maps.


\begin{highlightbox}
\begin{lemm}[Resolution of CP\=/maps]
\makelabel{lemm:resolution:sig:article-graph-raj-dahya}
  Suppose that $H_{1}$ is separable
  with ONB $\{\BaseVector{i}\}_{i \in I}$,
  where $I = \naturals$
  or $\{1,2,\ldots,N\}$ for some $N \in \naturals$.
  Let ${\Phi : L^{1}(H_{1}) \to L^{1}(H_{2})}$
  be a CPCB FP\=/map.
  Then a family
  $
    \{\zeta^{(\Phi)}_{i}\}_{i \in I}
    \subseteq
    L^{2}(H_{1} \otimes H_{2}, H_{2})
  $
  exists such that

  \begin{restoremargins}
  \begin{equation}
  \label{eq:separation:sig:article-graph-raj-dahya}
    \Phi(\ElementaryMatrix{i}{j})
    = \zeta^{(\Phi)}_{i}\:(\zeta^{(\Phi)}_{j})^{\ast}
  \end{equation}
  \end{restoremargins}

  \continueparagraph
  for $i,j \in I$.
  If $\Phi$ is moreover CPTP, then
  $\{\zeta^{(\Phi)}_{i}\}_{i \in I}$
  constitutes an orthonormal family
  \wrt the Hilbert\==Schmidt structure.
\end{lemm}
\end{highlightbox}

  \begin{proof}
    By the Choi\==Jamio{\l}kowski correspondence
    (see \Cref{lemm:choi-jamiolkowski:sig:article-graph-raj-dahya}),
    letting
    ${C_{i,j} \coloneqq \Phi(\ElementaryMatrix{i}{j})}$
    for $i,j \in I$,
    we have that
    $
      C^{(F)}
      \coloneqq
        \sum_{i,j \in F}
          \ElementaryMatrix{i}{j}
          \otimes
          C_{i,j}
    $
    is positive
    for finite $F \subseteq I$.
    And since $\Phi$ is an FP\=/map,
    each $C_{i,j}$ has finite rank.

    We now introduce some notation:
    For $n \in \{0\} \cup I$
    let
      $F_{n} \coloneqq \emptyset$
      if $n = 0$
      $F_{n} \coloneqq \{1,2,\ldots,n\}$
      otherwise,
    and set
      $C^{(n)} \coloneqq C^{(F_{n})}$.
    Observe that
      $\supp C^{(n)} \subseteq F_{n}$
    and that
    each $C^{(n)}_{i,j} \in \{C_{i,j},\zeromatrix\} \subseteq \FiniteRankOps{H_{2}}$.
    Since the assumptions of \Cref{lemm:cholesky:existence:sig:article-graph-raj-dahya}
    are fulfilled,
    there exists a bi-partite Cholesky decomposition
    $C^{(n)} = \hat{L}^{(n)}\:D^{(n)}\:(\hat{L}^{(n)})^{\ast}$
    satisfying
    $\supp(D^{(n)}),\supp(\hat{L}^{(n)} - \onematrix) \subseteq F_{n}$
    and
    $D^{(n)}_{i,i} \in \FiniteRankOps{H_{2}}$
    for each $i \in I$.

    Let $n \in I$ and note that
    $
      F_{n}^{\prime}
      \coloneqq
        F_{n} \without \{\max F_{n}\}
      = F_{n} \without \{n\}
      = F_{n-1}
    $.
    Observe that the construction in \Cref{lemm:cholesky:existence:sig:article-graph-raj-dahya}
    yields that the decomposition for $C^{(n)}$
    is derived from the decomposition of
    $
      p_{F_{n}^{\prime}}\:C^{(n)}\:p_{F_{n}^{\prime}}
      = p_{F_{n-1}}\:C^{(n)}\:p_{F_{n-1}}
      = C^{(n-1)}
    $.
    In particular, by \eqcref{eq:cholesky-construction:2:sig:article-graph-raj-dahya}

      \begin{shorteqnarray}
        \hat{L}^{(n)}
          &=
            &\hat{L}^{(n-1)}
            + \sum_{j=1}^{n}
                \ElementaryMatrix{n}{j}
                \otimes
                \hat{L}^{(n)}_{n,j}
            ~\text{and}
            \\
        D^{(n)}
          &=
            &D^{(n-1)}
            + \ElementaryMatrix{n}{n}
              \otimes
              D^{(n)}_{n,n}
      \end{shorteqnarray}

    \continueparagraph
    for some operators
    $\{\hat{L}^{(n)}_{n,j}\}_{j=1}^{n} \subseteq \BoundedOps{H_{2}}$
    and
    $D^{(n)}_{n,n} \in \FiniteRankOps{H_{2}}$.
    Noting that for the base case
    $\hat{L}^{(0)} = \onematrix_{H_{1} \otimes H_{2}}$
    and
    $D^{(0)} = \zeromatrix$
    (see the start of the proof of \Cref{lemm:cholesky:existence:sig:article-graph-raj-dahya}),
    by a simple induction argument we thus obtain families
    $\{\hat{L}_{i,j}\}_{i,j \in I}$
    and
    $\{D_{i,i}\}_{i \in I} \subseteq \FiniteRankOps{H_{2}}$
    satisfying

      \begin{shorteqnarray}
        \hat{L}^{(n)}
          &=
            &\onematrix_{H_{1} \otimes H_{2}}
            + \sum_{i=1}^{n}
              \sum_{j=1}^{i}
                \ElementaryMatrix{i}{j}
                \otimes
                \hat{L}_{i,j}
            ~\text{and}
            \\
        D^{(n)}
          &=
            &\sum_{i=1}^{n}
              \ElementaryMatrix{i}{i}
              \otimes
              D_{i,i}
      \end{shorteqnarray}

    \continueparagraph
    for each $n \in I$,
    whereby we define $\hat{L}_{i,j} \coloneqq \zeromatrix$
    for $i,j \in I$ with $j > i$.

    Let $n \in I$ be arbitrary.
    By defining
      $
        L_{i,j}
        \coloneqq
          \hat{L}_{i,j}\:\sqrt{D_{j,j}}
        \in L^{2}(H_{2})
      $
    for $i, j \in I$,
    one can express
    the lower triangular operator
    $
      L^{(n)}
      \coloneqq
        \hat{L}^{(n)}
        \sqrt{D^{(n)}}
    $
    as

      \begin{shorteqnarray}
        L^{(n)}
          &\eqcrefoverset{eq:finite-support-LD:sig:article-graph-raj-dahya}{=}
            &\sum_{i=1}^{n}
            \sum_{j=1}^{i}
              \ElementaryMatrix{i}{j}
              \otimes
              \hat{L}_{i,j}
              \:\sqrt{D_{j,j}}
            \\
          &=
            &\sum_{i=1}^{n}
            \sum_{j=1}^{i}
              \ketbra{\BaseVector{i}}{\BaseVector{j}}
              \otimes
              L_{i,j}
            \\
          &=
            &\sum_{i=1}^{n}
              (\ket{\BaseVector{i}} \otimes \onematrix)
              \:\underbrace{
                \Big(
                  \sum_{j=1}^{i}
                    \bra{\BaseVector{j}}
                    \otimes
                    L_{i,j}
                \Big)
              }_{\eqqcolon \zeta^{(\Phi)}_{i}},
      \end{shorteqnarray}

    \continueparagraph
    whereby each $
      \zeta^{(\Phi)}_{i}
      \in \FiniteRankOps{H_{1} \otimes H_{2}, H_{2}}
      \subseteq L^{2}(H_{1} \otimes H_{2}, H_{2})
    $,
    by virtue of each $D_{j,j}$
    being a finite rank operator.

    So for arbitrary $i,j \in I$,
    choosing $n \coloneqq \max\{i,j\} \in I$
    one obtains

      \begin{shorteqnarray}
        C_{i,j}
          =
            C^{(n)}_{i,j}
          &=
            &(\bra{\BaseVector{i}} \otimes \onematrix)
            \:C^{(n)}
            \:(\ket{\BaseVector{j}} \otimes \onematrix)
            \\
          &=
            &(\bra{\BaseVector{i}} \otimes \onematrix)
            \:L^{(n)}
            \:(L^{(n)})^{\ast}
            \:(\ket{\BaseVector{j}} \otimes \onematrix)
            \\
          &=
            &((\bra{\BaseVector{i}} \otimes \onematrix)\:L^{(n)})
            \:((\bra{\BaseVector{j}} \otimes \onematrix)\:L^{(n)})^{\ast}
            \\
          &=
            &\zeta^{(\Phi)}_{i}\:(\zeta^{(\Phi)}_{j})^{\ast},
      \end{shorteqnarray}

    \continueparagraph
    which proves \eqcref{eq:separation:sig:article-graph-raj-dahya}.

    Finally, if $\Phi$ is trace-preserving,
    then the final claim immediately follows from this,
    since under the Hilbert\==Schmidt structure
    $
      \braket{\zeta^{(\Phi)}_{j}}{\zeta^{(\Phi)}_{i}}_{\tr}
      = \tr(\zeta^{(\Phi)}_{i}\:(\zeta^{(\Phi)}_{j})^{\ast})
      = \tr(C_{i,j})
      = \delta_{i,j}
    $
    for $i,j \in I$.
  \end{proof}

\begin{cor}
\makelabel{cor:lemm:resolution:sig:article-graph-raj-dahya}
  Under the assumptions of \Cref{lemm:resolution:sig:article-graph-raj-dahya},
  letting $\HilbertRaum \coloneqq L^{2}(H_{1} \otimes H_{2}, H_{2})$,
  there exists a unique bounded operator
  $V_{\Phi} \in \BoundedOps{H_{1}}{\HilbertRaum}$
  such that
  $V_{\Phi}\BaseVector{i} = \zeta^{(\Phi)}_{i}$
  for all $i \in I$.
  Moreover,
  $V_{\Phi}$ is a contraction (\resp an isometry)
  if $\Phi$ is completely contractive
  (\resp trace-preserving).
\end{cor}

  \begin{proof}
    Using the vectors
    $
      \{\zeta^{(\Phi)}_{i}\}_{i \in I}
      \subseteq \HilbertRaum
    $
    let
    $V : \linspann\{\BaseVector{i} \mid i \in I\} \to \BoundedOps{\HilbertRaum}$
    denote the unique linear operation satisfying
    $V\BaseVector{i} = \zeta^{(\Phi)}_{i}$ for $i \in I$.
    Consider an arbitrary
      $x \in \linspann\{\BaseVector{i} \mid i \in I\}$,
    \idest
      $x = \sum_{i \in F}x_{i}\BaseVector{i}$
    for some finite $F \subseteq I$
    where $\{x_{i}\}_{i \in F} \subseteq \complex$.
    Then $V x = \sum_{i \in F}x_{i}\:\zeta^{(\Phi)}_{i}$
    and

      \begin{longeqnarray}
        \norm{V x}^{2}
          &=
            &\sum_{i,j \in F}
              x_{i}\:x_{j}^{\ast}
              \:\braket{\zeta^{(\Phi)}_{j}}{\zeta^{(\Phi)}_{i}}_{\HilbertRaum}
            \\
          &=
            &\sum_{i,j \in F}
              x_{i}\:x_{j}^{\ast}
              \:\tr(\zeta^{(\Phi)}_{i}\:(\zeta^{(\Phi)}_{j})^{\ast})
            \\
          &\eqcrefoverset{eq:separation:sig:article-graph-raj-dahya}{=}
            &\sum_{i,j \in F}
              x_{i}\:x_{j}^{\ast}
              \:\tr(\Phi(\ElementaryMatrix{i}{j}))
            \\
          &=
            &\tr(
              \Phi(
                \ketbra{
                  \sum_{i \in F}
                    x_{i}\BaseVector{i}
                }{
                  \sum_{j \in F}
                    x_{j}\BaseVector{j}
                }
              )
            )
            \\
          &=
            &\tr(\Phi(\ketbra{x}{x}))
            \\
          &\overset{(\ast)}{=}
            &\norm{\Phi(\ketbra{x}{x})}_{1}
            \\
          &\overset{(\ast\ast)}{\leq}
            &\norm{\Phi}\:\norm{\ketbra{x}{x}}_{1}
          =
            \norm{\Phi}\:\norm{x}^{2},
      \end{longeqnarray}

    \continueparagraph
    where ($\ast$) holds since $\Phi$ is (completely) positive
    and ($\ast\ast$) holds by virtue of $\Phi$ being (completely) bounded.
    It follows that $V$ uniquely extends to a bounded operator $V_{\Phi}$
    with $\norm{V_{\Phi}} \leq \sqrt{\norm{\Phi}}$
    defined on $H_{1} = \quer{\linspann}\{\BaseVector{i} \mid i \in I\}$.
    If $\Phi$ is completely contractive, then $\norm{\Phi} \leq 1$,
    making $V_{\Phi}$ a contraction.
    And if $\Phi$ is trace-preserving,
    then the inequality at ($\ast\ast$) is an equality
    with $\norm{\Phi} = 1$,
    rendering $V_{\Phi}$ an isometry.
  \end{proof}



We shall refer to the families
$
  \{
    L^{(\Phi)}_{i,j}
    \coloneqq
      \hat{L}_{i,j}\:\sqrt{D_{j,j}}
    =
      \hat{L}^{(i)}_{i,j}\:\sqrt{D^{(j)}_{j,j}}
  \}_{i,j \in I}
$
and
$\{\zeta^{(\Phi)}_{i}\}_{i \in I}$
in \Cref{lemm:resolution:sig:article-graph-raj-dahya}
as
the \highlightTerm{Choi\==Cholesky decomposition}
\resp \highlightTerm{resolution} of $\Phi$.
Note that whilst the resolution itself may not be unique,
its construction via the Choi\==Cholesky decomposition
provides a canonical approach.

\begin{rem}[Constructibility of the resolution]
\makelabel{rem:constructibility:sig:article-graph-raj-dahya}
  By inspecting the proofs of
  \Cref{lemm:cholesky:existence:sig:article-graph-raj-dahya}
  and \Cref{lemm:resolution:sig:article-graph-raj-dahya},
  the construction of the entries
  of the Choi\==Cholesky decompositions
  as well as the resolution vectors
  can be explicitly described.
  For convenience, we have captured these in Appendix \ref{app:alg-chol:sig:article-graph-raj-dahya}.
  \Cref{alg:bi-partite-cholesky:sig:article-graph-raj-dahya}
  makes clear that for each $i, j \in I$
  the operator
  $L^{(\Phi)}_{i, j}$
  can be constructed from
  $\{\Phi(\ElementaryMatrix{i'}{j'})\}_{i',j'=1}^{i}$
  using operations in the class $\mathbb{O}$
  defined in \S{}\ref{sec:intro:results:sig:article-graph-raj-dahya}.
  In particular, a map
  ${F_{i, j} : \FiniteRankOps{H_{2}}^{i \times i} \to \BoundedOps{H_{2}}}$
  in $\mathbb{O}$
  exists,
  such that
  $L^{(\Phi)}_{i,j} = F_{i,j}(\{\Phi(\ElementaryMatrix{i'}{j'})\}_{i',j'=1}^{i})$.
  And by \Cref{alg:resolution:sig:article-graph-raj-dahya},
  one has
  $
    \zeta^{(\Phi)}_{n}
    =
      \sum_{k=1}^{n}
        \bra{\BaseVector{k}}
        \otimes
        F_{n,k}(\{\Phi(\ElementaryMatrix{i}{j})\}_{i,j=1}^{n})
  $
  for each $n \in I$.
\end{rem}

\begin{rem}[Coherence]
\makelabel{rem:coherence:sig:article-graph-raj-dahya}
  A key feature of the construction in \Cref{lemm:resolution:sig:article-graph-raj-dahya}
  (\cf also \Cref{alg:bi-partite-cholesky:sig:article-graph-raj-dahya} in the appendix)
  is the coherence of the recursive construction.
  That is, the constructions
  $\{\hat{L}^{(m)}_{i,j}\}_{i,j=1}^{m}$,
  $\{D^{(m)}_{i,j}\}_{i,j=1}^{m}$,
  $\{L^{(m)}_{i,j}\}_{i,j=1}^{m}$
  are subsequences of
  $\{\hat{L}^{(n)}_{i,j}\}_{i,j=1}^{n}$,
  $\{D^{(n)}_{i,j}\}_{i,j=1}^{n}$,
  $\{L^{(n)}_{i,j}\}_{i,j=1}^{n}$
  for each $m,n \in I$ with $m < n$.
  In other words, the partial computations of the decompositions for $\Phi$
  restricted to lower dimensions remain preserved as the dimension increases.
  By contrast, diagonalisations of the finite dimensional Choi matrices
  do not in general yield a coherent sequence of eigenvectors.
\end{rem}

\begin{rem}[Dimension of $H_{1}$]
  The separability of $H_{1}$ was needed
  in \Cref{lemm:resolution:sig:article-graph-raj-dahya}
  in order to coherently patch together the bi-partite Cholesky decompositions
  of the Choi matrices $\Choi{\Phi}^{(F)}$.
  If $H_{1}$ is not separable,
  then one may for example turn to the \highlightTerm{compactness theorem}
  from mathematical logic and model theory
  (\cf
    \cite[Theorem~10.6]{Kaye2007BookMathematicalLogic}%
  ),
  via which one can demonstrate
  the consistency of asserting the existence of a family
    $\{\zeta^{(\Phi)}_{i}\}_{i \in I}$
  satisfying
  $
    \Phi(\ElementaryMatrix{i}{j})
    = \zeta^{(\Phi)}_{i}\:(\zeta^{(\Phi)}_{j})^{\ast}
  $
  for all $i,j \in I$.
  Use of the compactness theorem, however,
  comes at the price of sacrificing constructibility.
\end{rem}



\begin{e.g.}[Resolution of adjoints]
\makelabel{e.g.:resolution-adjoints:sig:article-graph-raj-dahya}
  Let $H_{1}$, $H_{2}$ be finite-dimensional Hilbert spaces
  and let $\{\BaseVector{i}\}_{i \in I}$
  be an ONB for $H_{1}$
  where
  $I = \naturals$
  or
  $\{1,2,\ldots,N\}$ for some $N \in \naturals$.
  Consider an isometry $V \in \BoundedOps{H_{1}}{H_{2}}$
  and the CPTP\=/map $\Phi = \ad{V}$.
  Applying
  \Cref{alg:bi-partite-cholesky:sig:article-graph-raj-dahya}
  and
  \Cref{alg:resolution:sig:article-graph-raj-dahya} to $\Phi$,
  one can verify that

  \begin{shorteqnarray}
      L^{(\ad{V})}_{i,j}
        = \delta_{j,1}\:\ad{V}\ElementaryMatrix{i}{1}
  \end{shorteqnarray}

  \continueparagraph
  are the entries
  of the Choi\==Cholesky decomposition of $\Phi$
  for $i, j \in I$,
  and that

  \begin{shorteqnarray}
    \zeta^{(\ad{V})}_{n}
      =
        \sum_{k = 1}^{i}
          \bra{\BaseVector{k}}
          \otimes
          \delta_{k,1}\:\ad{V}\ElementaryMatrix{n}{1}
      =
        \bra{\BaseVector{1}}
        \otimes
        \ad{V}\ElementaryMatrix{n}{1}
  \end{shorteqnarray}

  \continueparagraph
  for $n \in I$
  are the elements of the resolution of $\Phi$.
\end{e.g.}





\section[Proof of the main results]{Proof of the main results}
\label{sec:results:sig:article-graph-raj-dahya}


\firstparagraph
We now possess the means to prove the main results.

\begin{proof}[\Cref{thm:choi-cholesky-rep:sig:article-graph-raj-dahya}]
  Towards the \usesinglequotes{if}-direction,
  if a contraction (\resp an isometry) $V_{\Phi} \in \BoundedOps{H_{1}}{\HilbertRaum}$
  satisfying \eqcref{eq:choi-cholesky-representation:sig:article-graph-raj-dahya} exists,
  then $\Phi$ is the composition of two CPCB- (\resp CPCC- \resp CPTP-) maps
  and therefore itself a CPCB- (\resp CPCC- \resp CPTP-) map.

  Towards the \usesinglequotes{only if}-direction,
  due to the assumptions we may apply the resolution lemma
  (see \Cref{lemm:resolution:sig:article-graph-raj-dahya})
  which yields a family of vectors
  $
    \{\zeta^{(\Phi)}_{i}\}_{i \in I}
    \subseteq L^{2}(H_{1} \otimes H_{2}, H_{2})
    = \HilbertRaum
  $
  satisfying \eqcref{eq:separation:sig:article-graph-raj-dahya}.
  By \Cref{cor:lemm:resolution:sig:article-graph-raj-dahya}
  there exists a unique bounded operator
  (\resp contraction in the CPCC case
  \resp isometry in the CPTP case)
  ${
    V \coloneqq V_{\Phi}
    : H_{1} \to H
  }$
  satisfying
  $
    V\BaseVector{i}
    = \zeta^{(\Phi)}_{i}
  $
  for $i \in I$.
  The CPTP\=/map
  $
    \Psi
    \coloneqq \Psi_{H_{1},H_{2}}
    : L^{1}(H) \to L^{1}(H_{2})
  $
  from \S{}\ref{sec:intro:results:sig:article-graph-raj-dahya}
  yields

    \begin{shorteqnarray}
      \Phi(\ElementaryMatrix{i}{j})
        \eqcrefoverset{eq:separation:sig:article-graph-raj-dahya}{=}
          \zeta^{(\Phi)}_{i}\:(\zeta^{(\Phi)}_{j})^{\ast}
        \eqcrefoverset{eq:property:universal-cptp:sig:article-graph-raj-dahya}{=}
          \Psi(\ketbra{\zeta^{(\Phi)}_{i}}{\zeta^{(\Phi)}_{j}})
        =
          \Psi(\ketbra{V \BaseVector{i}}{V \BaseVector{j}})
        =
          \Psi(\ad{V}\ElementaryMatrix{i}{j})
    \end{shorteqnarray}

  \continueparagraph
  for all $i,j \in I$.

  Recall that $\Psi$ is a CPTP\=/map and thus a complete contraction.
  Similarly $\Phi$ is assumed to be bounded
  and $\ad{V}$ is bounded by virtue of $V$ being a bounded operator.
  So since $\Phi$ and $\Psi \circ \ad{V}$
  are bounded \wrt the $L^{1}$-norm
  and
    $
      \linspann\{\ElementaryMatrix{i}{j} \mid i,j \in I\}
    $
  is $L^{1}$-dense in $L^{1}(H_{1})$
  (\cf
    \cite[Theorem~2.4.17]{Murphy1990}%
  ),
  it follows from the above computation that
  $
    \Phi(s) = \Psi(\ad{V}\:s)
  $
  for all $s \in L^{1}(H_{1})$.
  This establishes the existence of a bounded operator
  (\resp contraction \resp an isometry)
  which satisfies \eqcref{eq:choi-cholesky-representation:sig:article-graph-raj-dahya}.
\end{proof}



To prove the second main result, we rely on the construction of the resolution.

\def\beweislabel{thm:choi-cholesky-rep-computability:sig:article-graph-raj-dahya}
\begin{proof}[\Cref{\beweislabel}]
  For each CPCB FP\=/map $\Phi$ we let $\mathfrak{C}(\Phi) \coloneqq V_{\Phi}$,
  where $V_{\Phi}$ is constructed as in the proof of \Cref{thm:choi-cholesky-rep:sig:article-graph-raj-dahya}.
  In particular, \punktcref{1} is immediately satisfied.

  Towards \punktcref{2},
  by \Cref{rem:constructibility:sig:article-graph-raj-dahya},
  there exist maps
  ${F_{n,k} : \BoundedOps{H_{2}}^{n \times n} \to \BoundedOps{H_{2}}}$
  in the class $\mathbb{O}$
  for $k,n \in I$ with $k \leq n$
  such that
  $
    \mathfrak{C}(\Phi)\BaseVector{n}
    = V_{\Phi}\BaseVector{n}
    = \zeta^{(\Phi)}_{n}
    = \sum_{k=1}^{n}
        \bra{\BaseVector{k}}
        \otimes
        F_{n,k}(\{\Phi(\ElementaryMatrix{i}{j})\}_{i,j=1}^{n})
  $
  for each CPCB FP\=/map $\Phi$
  and $n \in I$.
  Clearly then, $\mathfrak{C}(\Phi)\BaseVector{n}$
  can be constructed from $\{\Phi(\ElementaryMatrix{i}{j})\}_{i,j=1}^{n}$
  via an operation in the class $\mathbb{O}$.

  Towards \punktcref{3}, assume that $H_{2}$ is separable.
  By restricting $\mathfrak{C}$
  to the subspace $\mathcal{X}_{1}$ of CPCC\=/maps,
  we can replace $\mathcal{Y}$
  by the subspace $\mathcal{Y}_{1}$ of contractions.
  By uniform boundedness of the operators in $\mathcal{Y}_{1}$,
  and since the underlying Hilbert space is
  $\HilbertRaum = L^{2}(H_{1} \otimes H_{2}, H_{2})$,
  the $\topWOT$-topology is induced by the maps
  ${
      \mathcal{Y}_{1} \ni V
      \mapsto
      \braket{
        (
          \bra{\BaseVector{k}}
          \otimes
          \ketbra{\xi}{\eta}
        )
      }{
        V\BaseVector{n}
      }_{\HilbertRaum}
      \in\complex
  }$
  for $n, k \in I$, $\xi,\eta \in H_{2}$.
  It thus suffices to prove that

    \begin{shorteqnarray}
      \mathcal{X} &\to &\complex\\
      \Phi
      &\mapsto
      &\begin{array}[t]{0l}
        \braket{
          \Big(
            \bra{\BaseVector{k}}
            \otimes
            \ketbra{\xi}{\eta}
          \Big)
        }{
          V_{\Phi}\BaseVector{n}
        }_{\HilbertRaum}\\
        = \tr\Big(
            \Big(
              \bra{\BaseVector{k}}
              \otimes
              \ketbra{\xi}{\eta}
            \Big)^{\ast}
            \:V_{\Phi}\BaseVector{n}
          \Big)
      \end{array}
    \end{shorteqnarray}

  \continueparagraph
  is Borel-measurable
  for fixed $n,k \in I$ and $\xi,\eta \in H_{2}$.
  Now for fixed $\Phi \in \mathcal{X}$
  one has

    \begin{shorteqnarray}
      \tr\Big(
        \Big(
          \bra{\BaseVector{k}}
          \otimes
          \ketbra{\xi}{\eta}
        \Big)^{\ast}
        \:V_{\Phi}\BaseVector{n}
      \Big)
        &=
          &\sum_{k'=1}^{n}
            \tr\Big(
              \Big(
                \bra{\BaseVector{k}}
                \otimes
                \ketbra{\xi}{\eta}
              \Big)^{\ast}
              \:\Big(
                \bra{\BaseVector{k'}}
                \otimes
                F_{n,k'}(\{\Phi(\ElementaryMatrix{i}{j})\}_{i,j=1}^{n})
                \Big)
            \Big)
          \\
        &=
          &\einser_{k \leq n}
            \:\tr(
              \ketbra{\eta}{\xi}
              \:F_{n,k}(\{\Phi(\ElementaryMatrix{i}{j})\}_{i,j=1}^{n})
            ),
    \end{shorteqnarray}

  \continueparagraph
  whence it suffices to only consider the case $k \leq n$.
  Noting that the map
  ${\BoundedOps{H_{2}} \ni T \to \tr(\ketbra{\eta}{\xi}\:T) \in \complex}$
  is continuous \wrt the norm topology,
  it further suffices to prove the Borel-measurability of
  ${
    \mathcal{X}
    \ni \Phi
    \mapsto
    F_{n,k}(\{\Phi(\ElementaryMatrix{i}{j})\}_{i,j=1}^{n})
    \in \BoundedOps{H_{2}}
  }$.

  Now as explained \S{}\ref{sec:intro:results:sig:article-graph-raj-dahya},
  by separability of $H_{2}$,
  the operations in the class $\mathbb{O}$
  are Borel-measurable under the norm topology.
  So since $F_{n,k}$ is in $\mathbb{O}$,
  it suffices to prove that
  ${
    \mathcal{X} \ni \Phi
    \mapsto
    \{\Phi(\ElementaryMatrix{i}{j})\}_{i,j=1}^{n}
    \in \BoundedOps{H_{2}}^{n \times n}
  }$
  is Borel-measurable.
  By definition of the $\topSOT$-topology on $\mathcal{X}$
  and since the identity map
  ${L^{1}(H_{2}) \to \BoundedOps{H_{2}}}$
  is a contraction and thus continuous,%
  \footnote{%
    this can be proved via the H{\"o}lder\==von~Neumann inequality.
  }
  one has that in fact
  ${
    \mathcal{X} \ni \Phi
    \mapsto
    \Phi(\ElementaryMatrix{i}{j})
    \in \BoundedOps{H_{2}}
  }$
  is continuous for each $i,j \in I$.
  Hence the restriction
  ${
    \mathfrak{C}\restr{\mathcal{X}_{1}}
    :\mathcal{X}_{1}
    \to
    \mathcal{Y}_{1}
  }$
  is Borel-measurable as claimed.
\end{proof}



\begin{rem}[Finite rank restriction]
  The arguments the proof of the bi-partite Cholesky decomposition
  (\Cref{lemm:cholesky:existence:sig:article-graph-raj-dahya})
  appear to work if the FP requirement is dropped
  and the ideal is replaced by $\mathcal{K} = L^{1}(H_{2})$.
  Only step \ref{it:majorisation:lemm:cholesky:existence:sig:article-graph-raj-dahya}
  of this proof fails,
  which appeals to the majorisation lemma
  (\Cref{lemm:cholesky:existence:sig:article-graph-raj-dahya})
  and makes critical use of pseudo-inverses
  (\cf \Cref{rem:majorisation-fp-necessary:sig:article-graph-raj-dahya}).
  It would be useful to know if a constructive representation of CP(TP)\=/maps
  \akin \Cref{thm:choi-cholesky-rep:sig:article-graph-raj-dahya}
  can be achieved without the FP restriction.
\end{rem}

\begin{rem}[Continuity]
  At present,
  no version of Kraus's \Second representation theorem appears to be known,
  in which the unitary operator in \eqcref{eq:kraus:II:sig:article-graph-raj-dahya}
  depends continuously on the CPTP\=/map.
  At best there exists a continuous version for Stinespring dilations
  (on which Kraus's \First representation theorem rests),
  applicable to one-parameter families
  \cite{Parthasarathy1990Incollection}.
  By
  \Cref{thm:choi-cholesky-rep-computability:sig:article-graph-raj-dahya},
  a Borel-measurable version of Kraus's \Second representation has been achieved
  via a canonical explicitly described construction.
  Our reliance on spectral theory (in particular for pseudo-inverses)
  appears to be the chief hindrance to continuity.
\end{rem}

\begin{rem}[Computability]
  Outside of mathematical logic,
  where the \emph{computability} term has been settled,
  different notions exist for the field of analysis
  (\cf
    \cite[\S{}9.8]{Weihrauch2000BookComputableAnalysis},
    \cite[\S{}1]{Braverman2005Inproceedings},
    \cite[\S{}1]{BravermanCook2006Article}%
  ),
  which go beyond the scope of this paper.
  Nonetheless, a few points are worth mentioning:
  In the finite-dimensional setting,
  whilst Choi's proof of Kraus's \First (and therefore \Second)
  representation of CPTP\=/maps is not canonical
  due to reliance on matrix diagonalisation
  \cite[Theorem~1 and Remark~4]{Choi1975Article},
  his approach appears to yield computable constructions
  in the Borel\==Turing sense.%
  \footnote{%
    loosely, this notion of computability incorporates numerical stability,
    \cf
    \cite[Chapters~4,5,8, and 9]{Weihrauch2000BookComputableAnalysis},
    \cite[\S{}2.9]{AvigadBrattka2014Incollection}.
  }
  By contrast,
  our reliance on pseudo-inverses suggests
  that our construction might be non-computable in the same sense.
  Indeed, for separable Hilbert spaces $H_{1}$ and $H_{2}$
  with $\dim(H_{1}) \geq 2$,
  our methods yield

    \begin{shorteqnarray}
      \{L^{(\Phi)}_{i,j}\}_{i,j=1}^{2}
      = \begin{matrix}{cc}
          \sqrt{\Phi(\ElementaryMatrix{1}{1})}
            &\zeromatrix
          \\
          \Phi(\ElementaryMatrix{2}{1})
          \:\sqrt{\Phi(\ElementaryMatrix{1}{1})^{\dagger}}
            &\sqrt{
              \Phi(\ElementaryMatrix{2}{2})
              -
              \Phi(\ElementaryMatrix{2}{1})
              \Phi(\ElementaryMatrix{1}{1})^{\dagger}
              \Phi(\ElementaryMatrix{2}{1})^{\ast}
            }
      \end{matrix}
    \end{shorteqnarray}

  \continueparagraph
  and thus

    \begin{shorteqnarray}
      \zeta^{(\Phi)}_{1}
        &=
          &\bra{\BaseVector{1}}
          \otimes
          \sqrt{\Phi(\ElementaryMatrix{1}{1})},
          \quad\text{and}
          \\
      \zeta^{(\Phi)}_{2}
          &=
            &\bra{\BaseVector{1}}
            \otimes
            \Phi(\ElementaryMatrix{2}{1})
            \:\sqrt{\Phi(\ElementaryMatrix{1}{1})^{\dagger}}
            \\
            &&\:+
            \bra{\BaseVector{2}}
            \otimes
            \sqrt{
              \Phi(\ElementaryMatrix{2}{2})
              -
              \Phi(\ElementaryMatrix{2}{1})
              \Phi(\ElementaryMatrix{1}{1})^{\dagger}
              \Phi(\ElementaryMatrix{2}{1})^{\ast}
            },
    \end{shorteqnarray}

  \continueparagraph
  for any CPCB\=/map
  ${\Phi : L^{1}(H_{1}) \to L^{1}(H_{2})}$.
  Recent developments \cite{BocheFonoKutyniok2026ArticleNonComputablePseudoInverse}
  indicate that $\mathfrak{C}(\Phi)\BaseVector{2} = \zeta^{(\Phi)}_{2}$
  may fail to be Borel\==Turing computable
  in $\{\Phi(\ElementaryMatrix{i}{j})\}_{i,j=1}^{2}$
  \wrt the norm topology.
  However, our constructions appear to be at least \highlightTerm{effectively computable}.%
  \footnote{%
    roughly, this extends Turing-computability
    to separable spaces,
    \cf
    \cite[\S{}40.B]{Kechris1995BookDST},
    \cite[\S{}3D--E]{Moschovakis2009DstBook}.
  }
\end{rem}

This paper has achieved its primary goal of obtaining
dilations via standard operations which can all be implemented
in modern programming languages.
The chief advantages of our construction
over the approach taken in the literature
include its applicability to the infinite-dimensional setting
and its canonical nature relative to the choice of ONB.
Addressing the question posed at the start of the paper,
we have achieved
explicitly described
natural dilations of CP\=/maps.
Possible applications \exempli
to the study of quantum states,
quantum channel tomography
(\cf
  \cite{%
      BenderskyPaz2013Article,%
      BaldwinKalevDeutsch2014Article,%
      BelovDubovIvanov2026Misc%
  }%
),
\etcetera
remain an area of keen interest.




\setcounternach{section}{1}
\documentpartappendix



\section[Measurability of spectral maps]{Measurability of spectral maps}
\label{app:spectral+mb:sig:article-graph-raj-dahya}

\firstparagraph
Consider the subspaces
$
  \FiniteRankOps{H}^{+}
  \subseteq \FiniteRankOps{H}_{\text{s-a}}
  \subseteq \BoundedOps{H}
$
of positive \resp self-adjoint finite rank operators
over a Hilbert space $H$.
If $H$ is separable,
these form separable metrisable spaces
under the operator norm
(\cf
  \cite[Remark~4.1.6]{Murphy1990}%
).
In the following we demonstrate the measurability
of certain operator theoretic constructions
\wrt the norm topology,
which are readily derived
from a basic understanding of spectral theory.


\begin{prop}
\makelabel{prop:norm-measurability-spectral-maps:sig:article-graph-raj-dahya}
  Let $H$ be a separable Hilbert space.
  For every
  $f \in \Cts{\reals}$,
  the corresponding restricted spectral map
  ${\hat{f} : \FiniteRankOps{H}_{\text{s-a}} \to \FiniteRankOps{H}_{\text{s-a}}}$
  is Borel-measurable.
\end{prop}

  \begin{proof}
    For each $n\in\naturals$ let $h_{n} \in \Cts{\reals}$
    be any continuous map with $h_{n} \equiv 1$ on $[-n,\:n]$
    and $h_{n} \equiv 0$ on $\reals \without [-(n+1),\:(n+1)]$.
    Then each $f_{n} \coloneqq h_{n} \cdot f$ is a bounded continuous function
    and
    for each self-adjoint $T \in \BoundedOps{H}$
    the spectral theorem yields
    $
      \hat{f}_{n}(T)
      = \widehat{f_{n}\restr{\sigma(T)}}(T)
      = \widehat{f\restr{\sigma(T)}}(T)
      = \hat{f}(T)
    $
    for sufficiently large $n\in\naturals$.
    Thus
    ${
      \widehat{f}_{n}
      \underset{n}{\longrightarrow}
      \hat{f}
    }$
    pointwise.
    Since the $f_{n}$ are bounded continuous functions,
    the spectral maps
      $\widehat{f}_{n}$
    are strongly continuous
    (see \exempli
      \cite[Theorem~4.3.2]{Murphy1990}%
    )
    and thus Borel-measurable \wrt the norm topology.%
    \footnote{%
      By separability,
      the basic open sets under the norm topology
      are Borel-measurable (in fact $F_{\sigma}$)
      under the $\topSOT$-topology,
      since
      $
        \{T \in \FiniteRankOps{H}_{\text{s-a}} \mid \norm{T - S} < r\}
        =
        \bigcup_{r'\in(0,\:r) \cap \rationals}
        \bigcap_{\xi \in D}
        \{T \in \FiniteRankOps{H}_{\text{s-a}} \mid \norm{(T - S)\xi} \leq r'\norm{\xi}\}
      $
      for all $S \in \FiniteRankOps{H}_{\text{s-a}}$, $r >  0$,
      where $D \subseteq H$ is any countable dense subset of $H$.
      Conversely, since norm convergence implies strong convergence,
      strongly open sets are open under the norm topology.
      Since the Borel sets are generated from open sets,
      it follows that the Borel $\sigma$-algebra
      induced by the $\topSOT$- and norm topologies
      coincide.
    }
    Since the pointwise limit of Borel-measurable functions
    between separable metrisable spaces is Borel-measurable,%
    \footnoteref{ft:ptwise-limit-of-mb-is-mb:sig:article-graph-raj-dahya}
    \footnotetext[ft:ptwise-limit-of-mb-is-mb:sig:article-graph-raj-dahya]{%
      Let $\mathcal{W} \subseteq \FiniteRankOps{H}_{\text{s-a}}$
      and
      ${h, h_{1}, h_{2}, \ldots : \mathcal{W} \to \FiniteRankOps{H}_{\text{s-a}}}$
      be maps for which each $h_{n}$ is Borel-measurable and
      ${h_{n} \underset{n}{\longrightarrow} h}$
      pointwise.
      Since $\FiniteRankOps{H}_{\text{s-a}}$
      is a separable metrisable space under the operator norm,
      to show the Borel-measurability of $h$
      it suffices to demonstrate the measurability
      of the pre-images of basic open sets.
      Considering
      $
        V
        \coloneqq \oBall{S}{r}
        \coloneqq \{T \in \FiniteRankOps{H}_{\text{s-a}} \mid \norm{T - S} < r\}
      $
      where $S \in \FiniteRankOps{H}_{\text{s-a}}$ and $r > 0$,
      one has
      $
        h^{-1}(V)
        = \{
            T \in \mathcal{W}
            \mid
            \lim_{n}
            h_{n}(T)
            \in V
          \}
        = \{
            T \in \mathcal{W}
            \mid
            \norm{
              \lim_{n}
              h_{n}(T)
              -
              S
            }
            < r
          \}
        = \bigcup_{r' \in (0,\:r) \cap\rationals}
          \bigcup_{n \in \naturals}
          \bigcap_{k=n}^{\infty}
          \{
            T \in \mathcal{W}
            \mid
            \norm{
              h_{k}(T)
              -
              S
            }
            < r'
          \}
        = \bigcup_{r' \in (0,\:r) \cap\rationals}
          \bigcup_{n \in \naturals}
          \bigcap_{k=n}^{\infty}
          h_{k}^{-1}(\oBall{S}{r'})
      $,
      which is Borel-measurable.
      \Cf also
      \cite[Theorem~11.6]{Kechris1995BookDST}.
    }
    it follows that $\hat{f}$ is Borel-measurable.
  \end{proof}

\begin{e.g.}
\makelabel{e.g.:norm-measurability-of-sqrt:sig:article-graph-raj-dahya}
  Let $H$ be a separable Hilbert space.
  Since ${\reals \ni t \mapsto \sqrt{\abs{t}}}$
  is continuous,
  the map
  ${\FiniteRankOps{H}_{\text{s-a}} \in T \mapsto \sqrt{\abs{T}} \in \FiniteRankOps{H}^{+}}$
  and thus its restriction
  ${\FiniteRankOps{H}^{+} \in T \mapsto \sqrt{T} \in \FiniteRankOps{H}^{+}}$
  are Borel-measurable \wrt the norm topology.
\end{e.g.}



The \highlightTerm{Moore\==Penrose pseudo-inverse}
of an operator $T \in \BoundedOps{H}$,
when it exists,
is the unique bounded operator $T^{\dagger} \in \BoundedOps{H}$
for which
$P \coloneqq T^{\dagger}\:T$ and $Q \coloneqq T\:T^{\dagger}$
are self-adjoint
with
$P\:T^{\dagger} = T^{\dagger} = T^{\dagger}\:Q$
and
$T\:P = T = Q\:T$
(\cf
  \cite[\S{}1.1]{BenIsraelGreville2003Book}%
).

\begin{prop}
\makelabel{prop:existence-of-pinv:sig:article-graph-raj-dahya}
  Let $H$ be a Hilbert space.
  For each $T \in \FiniteRankOps{H}_{\text{s-a}}$,
  the pseudo-inverse $T^{\dagger}$ exists
  and satisfies
  $
    T\:T^{\dagger}
    = T^{\dagger}\:T
    = \Proj_{\quer{\ran}(T)}
  $.
  Moreover, if $H$ is separable, then
  ${\hat{h} : \FiniteRankOps{H}_{\text{s-a}} \ni T \mapsto T^{\dagger} \in \BoundedOps{H}}$
  is Borel-measurable \wrt the norm topology.
\end{prop}

  \begin{proof}
    By the spectral theorem for positive compact operators
    (\cf
      \cite[Theorem~3.3.8]{Pedersen1989analysisBook}%
    ),
    there exists an (at most countable, possibly empty) orthonormal family
    $\{x_{i}\}_{i \in J} \subseteq H$
    and
    $
      \sigma(T) \without \{0\}
      = \{t_{i}\}_{i \in J}
      \subseteq \reals \without \{0\}
    $,
    such that
    $
      T = \sum_{i \in J}
        t_{i}\:\ketbra{x_{i}}{x_{i}}
    $,
    whereby the sum is computed strongly
    and the empty sum is taken to be $\zeromatrix$.
    Since $T$ has finite rank, $\sigma(T)$ and thus $J$ are finite.
    One can now readily check that
    $
      T^{\dagger}
      = \sum_{i \in J}
        t_{i}^{-1}\:\ketbra{x_{i}}{x_{i}}
    $
    and that the claimed identity holds.

    We further observe the following:
    Let $h_{n} \in \Cts{\reals}$
    be defined via $h_{n}(t) = \frac{t}{t^{2} + 2^{-n}}$
    for $t \in \realsNonNeg$, $n \in \naturals$.
    By spectral theory,
    orthonormality of $\{x_{i}\}_{i \in J}$,
    and the finitude of $J$,
    one obtains
    $
      \norm{\hat{h}_{n}(T) - \hat{h}(T)}
      = \norm{\hat{h}_{n}(T) - T^{\dagger}}
      = \normLarge{
          \sum_{i \in J}
            (h_{n}(t_{i}) - t^{-1}_{i})\:\ketbra{x_{i}}{x_{i}}
        }
      = \max_{i \in J}
          \abs{h_{n}(t_{i}) - t^{-1}_{i}}
    $,
    which converges to $0$ as ${n \longrightarrow \infty}$.
    Observe that the construction of
    $\{h_{n}\}_{n\in\naturals} \subseteq \Cts{\reals}$
    was independent of $T$.
    Now if $H$ is separable,
    we may apply \Cref{prop:norm-measurability-spectral-maps:sig:article-graph-raj-dahya}
    to obtain the Borel-measurability of each $\hat{h}_{n}$.
    And since the pointwise limit of Borel-measurable functions
    between separable metrisable spaces is Borel-measurable,%
    \footnoteref{ft:ptwise-limit-of-mb-is-mb:sig:article-graph-raj-dahya}
    it follows that $\hat{h}$ is Borel-measurable.
  \end{proof}




\section[Choi\==Cholesky decomposition]{Choi\==Cholesky decomposition}
\label{app:alg-chol:sig:article-graph-raj-dahya}

\SetKwFunction{AlgorithmChoiChol}{ChoiChol}
\SetKwFunction{AlgorithmRes}{Res}

\firstparagraph
Let
${\Phi : L^{1}(H_{1}) \to L^{1}(H_{2})}$
be a CPCB FP\=/map,
where $H_{1}$, $H_{2}$ are Hilbert spaces,
whereby $H_{1}$ is separable
with ONB $\{\BaseVector{n}\}_{n \in \naturals}$
or $\{\BaseVector{i}\}_{i=1}^{N}$
for some $N \in \naturals$.
By the Choi\==Jamio{\l}kowski correspondence (\Cref{lemm:choi-jamiolkowski:sig:article-graph-raj-dahya}),
$\Phi$ can be identified with the values
$
  \{
    \Phi(\ElementaryMatrix{i}{j})
  \}_{i,j \in I}
  \subseteq \FiniteRankOps{H_{2}}
$.
\Cref{alg:bi-partite-cholesky:sig:article-graph-raj-dahya}
below captures the constructions in
\Cref{lemm:resolution:sig:article-graph-raj-dahya}
for the \highlightTerm{Choi\==Cholesky decomposition}
$
  \{
    L^{(\Phi)}_{i,j}
    = \hat{L}_{i,j}\:\sqrt{D}_{j,j}
    = \hat{L}^{(i)}_{i,j}\:D^{(j)}_{j,j}
  \}_{i,j \in I}
$
of $\Phi$.
As per the arguments in \Cref{lemm:resolution:sig:article-graph-raj-dahya},
we note that each $\hat{L}^{(n)}$ and $D^{(n)}$
may be obtained via the constructions in
\Cref{lemm:cholesky:existence:sig:article-graph-raj-dahya},
in particular
\eqcref{eq:cholesky-construction:1:sig:article-graph-raj-dahya}
and
\eqcref{eq:cholesky-construction:2:sig:article-graph-raj-dahya}.
Building on this, \Cref{alg:resolution:sig:article-graph-raj-dahya}
captures the construction in
\Cref{lemm:resolution:sig:article-graph-raj-dahya}
of the \highlightTerm{resolution}
$\{\zeta^{(\Phi)}_{i}\}_{i \in I}$
of $\Phi$.


\def\beweislabel{alg:bi-partite-cholesky:sig:article-graph-raj-dahya}
\begin{algorithm2e}[H]
  \caption{Choi\==Cholesky decomposition of CP FP\=/maps}
  \label{\beweislabel}
  \Inputs{
    Finitely many values of a CP FP\=/map
    $
      \{\Phi(\ElementaryMatrix{i}{j})\}_{i,j=1}^{n}
      \subseteq \FiniteRankOps{H_{2}}
    $
    for some index $n \in I$,
    satisfying in particular
    $
      \sum_{i,j=1}^{n}
      \ElementaryMatrix{i}{j} \otimes \Phi(\ElementaryMatrix{i}{j})
      \geq \zeromatrix
    $.
  }
  \Outputs{
    First $n \times n$ entries
    $\{L^{(\Phi)}_{i,j}\}_{i,j=1}^{n}$
    of the decomposition of $\Phi$.
  }
  \BlankLine
  \Fn{\AlgorithmChoiChol{$\{\Phi(\ElementaryMatrix{i}{j})\}_{i,j=1}^{n}$}}{
    \Line{
      Initialise
      $\{\hat{L}_{i,j}\}_{i,j=1}^{n} \coloneqq \{\delta_{ij}\onematrix\}_{i,j=1}^{n}$,
      $\{D_{i,i}\}_{i=1}^{n} \coloneqq \{\zeromatrix\}_{i=1}^{n}$,
      $\{L^{(\Phi)}_{i,j}\}_{i,j=1}^{n} \coloneqq \{\zeromatrix\}_{i,j=1}^{n}$.
    }
    \Line{
      Initialise
      $\{\hat{R}_{i,j}\}_{i,j=1}^{n} \coloneqq \{\delta_{ij}\onematrix\}_{i,j=1}^{n}$
      \Comment{
        will contain entries of $\hat{L}^{-1}$
      }
    }
    \ForEach{$i \in \{1,2,\ldots,n\}$}{
      \Comment{
        GOAL\textsubscript{1}:
        Update entries on row $i$ of $\hat{L}$, $D$, $L^{(\Phi)}$.
      }
      \ForEach{$j \in \{1,2,\ldots,i-1\}$}{
        \Line{
          Set
          $
            \hat{L}_{i,j}
            \coloneqq
            \sum_{k=1}^{i-1}
              \Phi(\ElementaryMatrix{i}{k})
              \:\hat{R}_{j,k}^{\ast}
              \:D_{j,j}^{\dagger}
          $
        }
        \Line{
          Set
            $
              L^{(\Phi)}_{i,j}
              \coloneqq
              \hat{L}_{i,j}\:L^{(\Phi)}_{j,j}
            $
        }
      }
      \BlankLine
      \Line{
        Set
        $
          D_{i,i}
          \coloneqq
          \Phi(\ElementaryMatrix{i}{i})
          -
          \sum_{j=1}^{i-1}
            L^{(\Phi)}_{i,j}
            \:(L^{(\Phi)}_{i,j})^{\ast}
        $
      }
      \Line{
        Set
        $L^{(\Phi)}_{i,i} \coloneqq \sqrt{D_{i,i}}$
      }
      \BlankLine
      \Comment{
        GOAL\textsubscript{2}:
        Update entries on row $i$ of $\hat{L}^{-1}$
      }
      \ForEach{$j \in \{1,2,\ldots,i-1\}$}{
        \LineNoNr{
          Set
          $
            \hat{R}_{i,j}
            \coloneqq
              -\sum_{k=j}^{i-1}
                \hat{L}_{i,k}
                \:\hat{R}_{k,j}
          $
        }
      }
    }
    \BlankLine
    \Return
      $\{L^{(\Phi)}_{i,j}\}_{i,j=1}^{n}$
  }
\end{algorithm2e}

The computation for Goal\textsubscript{2}
in \Cref{alg:bi-partite-cholesky:sig:article-graph-raj-dahya}
is justified by
\Cref{rem:compute-restrictions-of-inverse:sig:article-graph-raj-dahya}
and \eqcref{eq:lower-uni-tri-inverse:sig:article-graph-raj-dahya}.



\begin{algorithm2e}[H]
  \caption{Resolution of CP FP\=/maps}
  \label{alg:resolution:sig:article-graph-raj-dahya}
  \Inputs{
    Finitely many values of a CP FP\=/map
    $
      \{\Phi(\ElementaryMatrix{i}{j})\}_{i,j=1}^{n}
      \subseteq \FiniteRankOps{H_{2}}
    $
    for some index $n \in I$.
  }
  \Output{
    $n$-th element $\zeta^{(\Phi)}_{n} \in L^{2}(H_{1} \otimes H_{2}, H_{2})$
    of the resolution of $\Phi$.
  }
  \Fn{\AlgorithmRes{$\{\Phi(\ElementaryMatrix{i}{j})\}_{i,j=1}^{n}$}}{
    \BlankLine
    \Comment{
      Compute part of Choi-Cholesky decomposition
    }
    \Line{
      Set
      $
        \{L^{(\Phi)}_{i,j}\}_{i,j=1}^{n},
        \coloneqq
        \AlgorithmChoiChol(\{\Phi(\ElementaryMatrix{i}{j})\}_{i,j=1}^{n})
      $.
    }
    \BlankLine
    \Comment{
      Compute desired element of the resolution
    }
    \Line{
      Set
      $
        \zeta^{(\Phi)}_{n}
        \coloneqq
        \sum_{k=1}^{n}
          \bra{\BaseVector{k}}
          \otimes
          L^{(\Phi)}_{n,k}
      $
    }
    \BlankLine
    \Return
      $\zeta^{(\Phi)}_{n}$
  }
\end{algorithm2e}


Towards empirical verification,
\Cref{alg:bi-partite-cholesky:sig:article-graph-raj-dahya}
has been implemented in
\cite{Dahya2026CodebaseChoiCholesky}
in \href{https://www.python.org}{python} and \href{https://rust-lang.org}{rust}.



\documentpartnormal


\null


\paragraph{Acknowledgement.}
The author is grateful
  to Orr Shalit
    for constructive remarks on CPTP-maps,
  to Adalbert Fono
    for advice on computable analysis,
and
  to Andreas Maletti and Elias Zimmermann
    for helpful exchanges on computability.


\bibliographystyle{siam}
\def\bibname{References}
\bgroup

\egroup


\addresseshere
\end{document}
